\documentclass[11pt]{article}

\usepackage[english]{babel}

\usepackage[letterpaper,top=2cm,bottom=2cm,left=3cm,right=3cm,marginparwidth=1.75cm]{geometry}

\usepackage{amsmath}
\usepackage{amsfonts}
\usepackage{amsthm}

\usepackage{graphicx}
\usepackage{xcolor}
\usepackage{subcaption}
\usepackage{multirow}
\usepackage{booktabs}
\usepackage{enumitem}
\usepackage{placeins}

\usepackage[title]{appendix}

\usepackage{algorithm}
\usepackage{float}
\floatstyle{plaintop}
\restylefloat{algorithm}

\usepackage{tikz}
\usetikzlibrary{decorations.pathreplacing}

\usepackage[colorlinks=true, allcolors=blue]{hyperref}

\numberwithin{equation}{section}
\allowdisplaybreaks[4]

\newtheorem{theorem}{Theorem}[section]

\newtheorem{proposition}{Proposition}[section]

\newtheorem{remark}{Remark}[section]
\newtheorem{example}{Example}[section]

\title{Weighted Inverse Lax-Wendroff Boundary Treatment of Discontinuous
Galerkin Methods for Conservation Laws }
\author{
Yongjie Bi
    \footnote{School of Mathematical Sciences, University of Science and Technology of China, Hefei, Anhui 230026, China.
    E-mail: {\tt yjbi@mail.ustc.edu.cn}.}
    \footnote{Laoshan Laboratory, Qingdao 266237, China.},
\and Yan Jiang
    \footnote{School of Mathematical Sciences, University of Science and Technology of China, Hefei, Anhui 230026, China.
    Email: {\tt jiangy@ustc.edu.cn}.
    Research supported by NSFC grant 12271499.}
    \footnote{Laoshan Laboratory, Qingdao 266237, China.},
\and Yong Liu
    \footnote{ICMSEC, State Key Laboratory of Mathematical Sciences (SKLMS), Academy of Mathematics and Systems Science, Chinese Academy of Sciences, Beijing 100190, P.R. China.  E-mail: {\tt yongliu@lsec.cc.ac.cn}. Research supported by NSFC grant 12571395, 12288201, the Strategic Priority Research Program of the Chinese Academy of Sciences under the Grant No. XDB0640000, and the Youth Innovation Promotion Association (CAS).} 
}
\date{}

\begin{document}
\maketitle

\noindent \textbf{Abstract:}
In this paper, we propose a weighted inverse Lax-Wendroff (WILW) boundary treatment for the discontinuous Galerkin (DG) method on unfitted meshes to efficiently solve hyperbolic conservation laws in complex geometries. The proposed method employs the standard DG scheme for interior cells and reconstructs high-order approximation polynomials via the ILW principle for cut cells near boundaries to impose numerical boundary conditions, effectively eliminating the time-step restriction typically caused by small cut cells. In particular, to address the sensitivity of numerical errors to the geometric size of cut cells in the basic ILW scheme, we raise the reconstruction order at the boundary, ensuring that accuracy becomes independent of the cut‑cell size. Furthermore, it incorporates a weighted least‑squares reconstruction to reduce the need for complex high‑order boundary derivatives during construction. As a result, the method maintains high‑order accuracy while significantly improving computational efficiency for multi‑dimensional problems. Finally, the stability of the proposed method is theoretically validated through linear stability analysis, and the effectiveness and robustness of the proposed scheme are numerically verified through a series of one-dimensional and two-dimensional numerical experiments for scalar and system equations.

\vspace{1ex}

\section{Introduction}
Hyperbolic conservation laws accurately describe many practical fluid dynamics problems on complex geometric domains. During numerical simulations, unfitted meshes such as uniform Cartesian meshes offer a widely adopted and effective alternative to avoid the significant challenges and computational costs associated with generating high-quality body-fitted meshes for complex regions or moving boundaries. However, in an unfitted mesh framework, arbitrary physical boundaries typically intersect the background grid, resulting in irregular cut cells near the interface. Because the computational boundaries do not coincide with the physical boundaries, numerical boundary conditions must be meticulously handled to maintain the optimal order of accuracy. Furthermore, the geometric sizes of these cut cells are often highly arbitrary, with volume fractions that can be orders of magnitude smaller than those of standard interior cells. This disparity gives rise to the well-known ``small-cell'' problem, which imposes severe time-step restrictions for time-dependent problems using explicit time-marching methods. In this paper, we propose an effective high-order discontinuous Galerkin (DG) boundary treatment for hyperbolic conservation laws on unfitted meshes. In particular, we develop a new type of inverse Lax-Wendroff (ILW) method to further reduce the computational complexity.

The DG method is a class of finite element methods employing discontinuous basis functions, originally introduced by Reed and Hill \cite{reed1973}. Due to its high-order accuracy, low numerical dispersion and dissipation, significant geometric flexibility, and excellent parallel scalability, Runge-Kutta DG (RKDG) methods \cite{rkdg1, rkdg2, rkdg3, rkdg4, rkdg5} have achieved great success in solving various hyperbolic conservation laws. However, when applying DG methods on unfitted meshes, one must meticulously handle the cut cells generated by geometric intersections near the boundaries. To overcome the small-cell problem, various boundary treatment strategies have been proposed in the literature. These methods include implicit time stepping \cite{bastian2011}, { cell merging and agglomeration \cite{muller2017, qin2013, schoeder2020, Chenliu2023}}, ghost penalty stabilization \cite{bastian2011, fu2021, fu2022, gurkan2020, hansbo2014, sticko2016, sticko2019}, state redistribution \cite{giuliani2022}, and the shifted boundary method \cite{song2018}. Nevertheless, these techniques still face significant challenges in controlling the computational costs of complex boundary reconstruction and maintaining high-order accuracy near the boundaries. 

The ILW method has gained significant attention as an effective technique for imposing high-order numerical boundary conditions on Cartesian grids. Tan and Shu \cite{tan2010} originally proposed this approach to provide high-order boundary schemes for hyperbolic conservation laws within the finite difference framework. The ILW method utilizes the governing equations to convert normal derivatives at the boundary into time and tangential derivatives, subsequently determining values at ghost points through Taylor series expansions. Following the introduction, numerous researchers significantly refined the approach through extensive studies. To circumvent the intricate algebraic manipulations required by the standard ILW procedure for nonlinear systems, particularly in high-dimensional cases, the simplified ILW (SILW) method was proposed in \cite{tan2012, tan2013}. This approach constructs high-order derivatives directly by utilizing extrapolation techniques. 
In \cite{lu2021}, Lu et al. proposed an ILW method to handle problems involving changing wind directions at the boundaries. To ensure strict mass conservation, Ding et al. \cite{ding2020} modified the numerical fluxes near the boundaries, leading to a conservative ILW formulation. Furthermore, Li et al. \cite{li2023} integrated the ILW boundary treatment with the fast sweeping method, enabling efficient computation of the steady-state solutions of hyperbolic conservation laws. 
More recently, considerable effort has been devoted to reducing the computational cost of ILW boundary treatments. Liu et al. \cite{liu2024} developed an ILW technique for finite difference WENO schemes, in which the number of low-order terms involved in the reconstruction procedure is significantly reduced, thereby simplifying the implementation and improving computational efficiency. 
Along the same line, Zhu et al.\cite{zhu2026} introduced a new ILW method for finite difference HWENO schemes, where a least-squares formulation is employed to optimize the extrapolation procedure, effectively reducing computational complexity while preserving high-order accuracy and stability.
The ILW framework has also been extended to a broader class of high-order numerical methods. For example, it has been incorporated into finite volume WENO methods \cite{zhu2025} and one-dimensional DG methods \cite{yangdgilw}. However, the approach in \cite{yangdgilw} is difficult to apply to nonlinear systems and high-dimensional problems because of its high computational complexity. 
Beyond hyperbolic conservation laws, the ILW boundary treatments have found applications in a variety of other problems, including the convection-diffusion equations \cite{litt2022, litt2017, lu2016}, the Boltzmann equation \cite{Filbet2013}, and moving boundary problems \cite{cheng2021, liu2022, liu2023, tan2011, liuzp2025}. Additionally, linear stability analyses of (S)ILW methods have been carried out in \cite{litt2022, litt2016, litt2017, litt2024, vilar2015}, providing valuable theoretical insights for the development of stable numerical boundary treatments.

In this paper, we propose a numerical framework combining the ILW boundary treatment with the DG method to solve hyperbolic conservation laws on unfitted meshes and complex geometries efficiently. We employ the standard DG method for interior cells. Meanwhile, small cut cells near the boundaries are treated as special ghost cells where high-order approximation polynomials are reconstructed through ILW principles to impose numerical boundary conditions. This approach overcomes the small cut-cell problem and avoids computing numerical integrations on curved elements. However, preliminary studies indicate that while the basic ILW-DG scheme consistently achieves the optimal $(k+1)$-order convergence accuracy, the error depends significantly on the geometric size of the ghost cells. In \cite{yangdgilw}, a post-processing based on local conservation was developed to address this issue for 1D scalar problems and linear systems. This algorithm employs one additional normal derivative term at boundary. However, this procedure becomes overly complicated for nonlinear systems, particularly for two-dimensional problems. In fact, this conservative correction reduces the errors arising from the boundary treatment and, in some special cases, improves the accuracy of the boundary treatment by one order. Motivated by this observation, in this work, we propose a correction strategy by increasing the order of accuracy for the boundary reconstruction polynomials. 
More importantly, to extend the method efficiently to multi-dimensional cases, we propose a new class of weighted ILW (WILW) methods based on weighted least-squares reconstruction. This approach reduces the reliance on complex high-order boundary derivatives during polynomial construction. This strategy enhances computational efficiency significantly while maintaining the accuracy of the algorithm. Additionally, this paper provides a linear stability analysis that theoretically verifies the robustness of the proposed algorithms for second-, third-, and fourth-order methods ($k=1,2,3$). Numerical experiments show that our methods perform well in handling complex geometry boundaries.

The remainder of this paper is organized as follows. In Section 2, we describe the construction of (W)ILW methods for one-dimensional scalar conservation laws within the DG framework. This section provides a linear stability analysis of the numerical schemes through eigenvalue spectrum visualization and extends the methodology to one dimensional systems. Section 3 extends these algorithms to two-dimensional systems. In Section 4, a series of numerical experiments demonstrates the stability and high order accuracy of the proposed methods. Finally, Section 5 summarizes the work and provides concluding remarks.

\section{ILW method for one-dimensional conservation laws}
Consider the following one-dimensional scalar hyperbolic conservation law on domain $\Omega=(a,b)$:
\begin{equation}\label{1D Scalar Equation}
\begin{cases}
u_t + f(u)_x = 0, \quad x \in (a, b), \quad t > 0, \\
u(x, 0) = u_0(x), \quad x \in [a, b], \\
u(a, t) = g(t), \quad t > 0.
\end{cases}
\end{equation}
We assume that $f'(u(a,t)) \geq \sigma > 0 $ and $ f'(u(b,t)) \geq 0 $ for $ t > 0.$ Under these assumptions, $x = a$ is an inflow boundary where a Dirichlet boundary condition is prescribed, while $x = b$ is an outflow boundary where no boundary condition is required.

The domain is partitioned into a uniform mesh, as illustrated in Fig. \ref{Domain decomposition}.
\begin{equation}\label{1Dmesh}
a + \delta_1 = x_{\frac{1}{2}} < x_{\frac{3}{2}} < \cdots < x_{N+\frac{1}{2}} = b - \delta_2  
\end{equation}
with the uniform mesh size $h = x_{j+1/2}-x_{j-1/2} = (b-a-\delta_1-\delta_2)/N$, $j=1, \ldots, N$, $N \in \mathbb{N}$. Note that the physical boundary is allowed to be not coinciding with grid points, and the distance $\delta_{1,2}$ can be any values satisfying $0 \le \delta_{1,2} < h$. Let $I_j = [x_{j-1/2}, x_{j+1/2}]$, $j = 1, \dots, N$, and the computational domain is defined as $\widetilde{\Omega} = [x_{1/2}, x_{N+1/2}] = \cup_{j=1}^{N} I_j$, within which the numerical solution is computed using the standard DG method. The residual segments between $\Omega$ and $\widetilde{\Omega}$ are denoted by $\tilde{I}_0 = [a, x_{1/2}]$ and $\tilde{I}_{N+1} = [x_{N+1/2}, b]$.

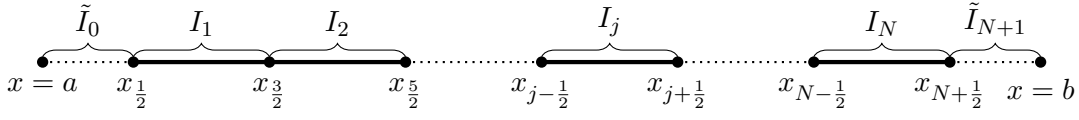
\begin{figure}[htbp]
    \centering
    \begin{tikzpicture}[
        scale=1.2,
        dot/.style={circle, fill, inner sep=1.5pt},
        brace/.style={decorate, decoration={brace, amplitude=5pt, raise=2pt}}
    ]
        
        \coordinate (A) at (0,0);
        \coordinate (X1) at (1,0);
        \coordinate (X2) at (2.5,0);
        \coordinate (X3) at (4,0);
        \coordinate (Xj1) at (5.5,0);
        \coordinate (Xj2) at (7,0);
        \coordinate (XN1) at (8.5,0);
        \coordinate (XN2) at (10,0);
        \coordinate (B) at (11,0);

        \draw[thick, dotted] (A) -- (X1);
        \draw[ultra thick] (X1) -- (X3);
        \draw[thick, dotted] (X3) -- (Xj1);
        \draw[ultra thick] (Xj1) -- (Xj2);
        \draw[thick, dotted] (Xj2) -- (XN1);
        \draw[ultra thick] (XN1) -- (XN2);
        \draw[thick, dotted] (XN2) -- (B);

        \node[dot, label=below:{$x=a$}] at (A) {};
        \node[dot, label=below:{$x_{\frac{1}{2}}$}] at (X1) {};
        \node[dot, label=below:{$x_{\frac{3}{2}}$}] at (X2) {};
        \node[dot, label=below:{$x_{\frac{5}{2}}$}] at (X3) {};
        \node[dot, label=below:{$x_{j-\frac{1}{2}}$}] at (Xj1) {};
        \node[dot, label=below:{$x_{j+\frac{1}{2}}$}] at (Xj2) {};
        \node[dot, label=below:{$x_{N-\frac{1}{2}}$}] at (XN1) {};
        \node[dot, label=below:{$x_{N+\frac{1}{2}}$}] at (XN2) {};
        \node[dot, label=below:{$x=b$}] at (B) {};

        \draw[brace] (A) -- (X1) node[midway, above=6pt] {$\tilde{I}_0$};
        \draw[brace] (X1) -- (X2) node[midway, above=6pt] {$I_1$};
        \draw[brace] (X2) -- (X3) node[midway, above=6pt] {$I_2$};
        \draw[brace] (Xj1) -- (Xj2) node[midway, above=6pt] {$I_j$};
        \draw[brace] (XN1) -- (XN2) node[midway, above=6pt] {$I_N$};
        \draw[brace] (XN2) -- (B) node[midway, above=6pt] {$\tilde{I}_{N+1}$};

    \end{tikzpicture}
    \caption{Domain decomposition}
    \label{Domain decomposition}
\end{figure}

\indent Let $P^k(I_j)$ be the space of polynomials of degree at most $k\ge0$ on $I_j$, and the DG finite-element space is defined as
\begin{equation}
V_h^k = \left\{ v \colon v|_{I_j} \in P^k(I_j), \ j = 1, \cdots, N \right\}.
\end{equation}
The semi-discrete DG method for solving $\eqref{1D Scalar Equation}$ is defined as follows: find the unique function $u_h(\cdot,t)\in V_h^k$ such that for all test functions $v_h\in V_h^k$ and all $1\le j\le N$, we have

\begin{equation}\label{semi}
\int_{I_j} (u_h)_t v_h \, dx - \int_{I_j} f(u_h) (v_h)_x \, dx + \hat{f}_{j+1/2} (v_h)_{j+1/2}^-  - \hat{f}_{j-1/2} (v_h)_{j-1/2}^+ = 0.
\end{equation}
Unless otherwise specified, we generally use the Lax-Friedrichs flux at the cell interface
\begin{equation}
\begin{aligned}
\hat{f}_{j+1/2} =& \hat{f}^{\,\text{LF}}\left( (u_h)_{j+1/2}^-, (u_h)_{j+1/2}^+ \right)\\ 
=& \frac{1}{2} \left( f((u_h)_{j+1/2}^-) + f((u_h)_{j+1/2}^+) - \alpha ((u_h)_{j+1/2}^+ - (u_h)_{j+1/2}^- ) \right), 
\end{aligned}
\end{equation}
where $(u_h)_{j+1/2}^-$ (resp. $(u_h)_{j+1/2}^+$) denote the limit value of $u_h$ at $x_{j+1/2}$ from the element $I_j$ (resp. $I_{j+1}$), and $\alpha = \max_u |f'(u)|$.

The semi-discrete DG scheme $\eqref{semi}$ can be rewritten as the first-order ordinary differential equation system $u_t =\mathcal{L}(u)$, where the operator $\mathcal{L}(u)$ arises from the spatial discretization. The third-order TVD Runge-Kutta method is applied for the time discretization to evolve the solution from time $t_n$ to $t_{n+1} = t_n + \Delta t$:
\begin{equation}\label{RK}
\begin{cases}
u^{(1)} = u^n + \Delta t \mathcal{L}(u^n), \\
u^{(2)} = \frac{3}{4} u^n + \frac{1}{4} u^{(1)} + \frac{1}{4} \Delta t \mathcal{L}(u^{(1)}), \\
u^{n+1} = \frac{1}{3} u^n + \frac{2}{3} u^{(2)} + \frac{2}{3} \Delta t \mathcal{L}(u^{(2)}).
\end{cases}
\end{equation}
As demonstrated in \cite{rkdg2}, to ensure the stability of the RKDG method with $P^k$ elements for periodic problems, the time step $\Delta t$ must satisfy: $\Delta t \le \frac{1}{2k+1} \frac{h}{\alpha}$. For the hyperbolic problem \eqref{1D Scalar Equation} with Dirichlet boundary conditions, \cite{carpenter1995} indicates that the following temporal correction to the boundary conditions is necessary to maintain the third-order accuracy of \eqref{RK}:
\begin{equation}\label{eq:bc_rk}
\left\{
\begin{array}{l}
u^n \sim g(t_n), \\
u^{(1)} \sim g(t_n) + \Delta t g'(t_n), \\
u^{(2)} \sim g(t_n) + \frac{\Delta t}{2} g'(t_n) + \frac{(\Delta t)^2}{4} g''(t_n).
\end{array}
\right.
\end{equation}

Notably, the standard RKDG scheme can be directly applied to compute numerical solutions on the two cut cells $\tilde{I}_0$ and $\tilde{I}_{N+1}$. However, to ensure stability, the time step $\Delta t$ would be restricted to a scale proportional to $\delta=\min(\delta_1, \delta_2)$. If $\delta$ is extremely small, this leads to the ``small-cell'' problem. 
To address this, we develop the inverse Lax-Wendroff (ILW) method to construct polynomials on $\tilde{I}_0$ and $\tilde{I}_{N+1}$ via boundary treatment, and further define the numerical fluxes $\hat{f}_{1/2}$ and $\hat{f}_{N+1/2}$. Consequently, the scheme can adopt the same time step as the standard DG method for periodic problems, meaning $\Delta t$ depends only on $h$ and is independent of $\delta_{1,2}$.
Our algorithm is incorporated at each stage of the Runge-Kutta time discretization using the corresponding boundary conditions \eqref{eq:bc_rk}.
For simplicity, in the following discussion we illustrate our boundary treatment design by considering the spatial semi‑discrete scheme with boundary condition $g(t)$, and the extension to the fully discrete Runge--Kutta scheme follows directly.

\subsection{(S)ILW boundary treatment at the inflow boundary} \label{subsec_SILW}
In this subsection, we consider the (S)ILW boundary treatment for the 1D scalar conservation law $\eqref{1D Scalar Equation}$. We focus on the inflow boundary at $x=a$ and the associated small cell $\tilde{I}_0 = [a, a+\delta]$. The standard DG method would be used on the interior elements, and the (S)ILW procedure is employed to construct a reconstruction polynomial $p(x,t)$ on $\tilde{I}_0$. This polynomial is then used to define the numerical flux at the computational boundary $x_{1/2}$:
$$\hat{f}_{1/2} = \hat{f} \left( p \left( x_{1/2}^-, t \right), u_h \left( x_{1/2}^+, t \right) \right).$$

\subsubsection{Revisiting the (S)ILW Methods}\label{subsec_SILW_ori}
In contrast to the Lax-Wendroff scheme, which expresses time derivatives in terms of spatial derivatives, the ILW method utilizes the governing equation to convert time derivatives into spatial derivatives. Consequently, high-order spatial derivatives at $x = a$ are obtained by repeatedly using the PDE and the boundary conditions. For example,
\begin{equation}
\begin{cases}
u|_{x=a} = g(t), \\
\partial_x u|_{x=a} = \frac{g'(t)}{-f'(g(t))}, \\
\partial_x^{2} u|_{x=a} = \frac{f'(g(t)) g''(t) - 2 f''(g(t)) g'(t)^2}{f'(g(t))^3}, \\
\quad \vdots
\end{cases}
\end{equation}

The polynomial $p(x,t) \in P^k$ is determined by matching the spatial derivatives at the boundary $x=a$:
$$\partial_x^{m} p|_{x=a} = \partial_x^{m} u|_{x=a}, \quad m =0,1,\cdots,k, $$
which yields the Taylor expansion:
\begin{equation}\label{ILW}
p(x,t) = \sum_{m=0}^{k} \frac{(x-a)^m}{m!} \partial_x^{m} u|_{x=a}.
\end{equation}
While achieving arbitrary accuracy, this approach leads to complex high-order derivatives. To avoid the heavy algebraic manipulations involved, we can employ a simplified ILW (SILW) method, which utilizes information from the interior domain.

Define the cell averages of the numerical solution and its derivative as:
$$(\bar{u}^{(m)})_j= \frac{1}{h}\int_{I_j} \partial_x^{m} u_h(x) dx,\quad m=0,1,2,\cdots,k.$$
The SILW1 procedure determines the unique polynomial $p(x,t) \in P^k$ on $\tilde{I}_0$ by satisfying:
\begin{equation}\label{SILW1}
\begin{cases}
\partial_x^{m} p|_{x=a} = \partial_x^{m} u|_{x=a}, \quad m = 0, 1, \cdots, k-1,\\
\int_{I_1} p(x,t) \, dx = h (\bar{u}^{(0)})_1.
\end{cases}
\end{equation}
Similarly, the SILW2 polynomial is defined by matching $k-1$ boundary derivatives and two interior integral quantities:
\begin{equation}\label{SILW2}
\begin{cases}
\partial_x^{m} p|_{x=a} = \partial_x^{m} u|_{x=a}, \quad m = 0, 1, \cdots, k-2, \\
\int_{I_1} p(x, t) \, dx = h (\bar{u}^{(0)})_1, \\ \int_{I_1} \partial_xp(x, t) \, dx = h(\bar{u}^{(1)})_1.
\end{cases}
\end{equation}

These SILW boundary treatments can avoid the small-cell problem, and their stability was analyzed via eigenvalue spectrum visualization in \cite{yangdgilw}.
However, numerical experiments show that the error magnitudes of ILW \eqref{ILW} and SILW \eqref{SILW1}-\eqref{SILW2}, are all sensitive to $\delta/h$.  As a test example, we solve \eqref{1D Scalar Equation} with $f(u)=u, a=0, b=2\pi, u_0(x)=-\sin(x)$, and $g(t)=\sin(t)$ using $N=320$ cells and $P^k$ elements for $k=1, 2, 3$. We set the right boundary aligning with the mesh that is $\delta_2=0$ and the left cut cell with length $\delta_1 = \delta \in [0, h)$. The $L^2$ errors on the computational domain $\widetilde{\Omega}$ at the final time $T = 3$ are shown in Fig. \ref{error_ori}, revealing that the error increases significantly as $\delta/h$ approaches 1 for $k \ge 1$. It was found in \cite{yangdgilw} that enforcing local conservation on $\tilde{I}_0$ can slow down the error growth. However, this process requires introducing boundary derivatives of one order higher. Moreover, the procedure itself is too complicated and has thus only been applied to one-dimensional scalar problems and linear systems. Indeed, through this conservative correction, errors caused by the boundary treatment are diminished; in particular, under special circumstances, the accuracy of the boundary treatment is enhanced by one order. Inspired by this finding, one potential strategy to restrain error growth is to use a higher-order reconstructed polynomial.

\begin{figure}[htbp]
    \centering
    \begin{subfigure}{0.32\textwidth}
        \centering
        \includegraphics[width=\textwidth]{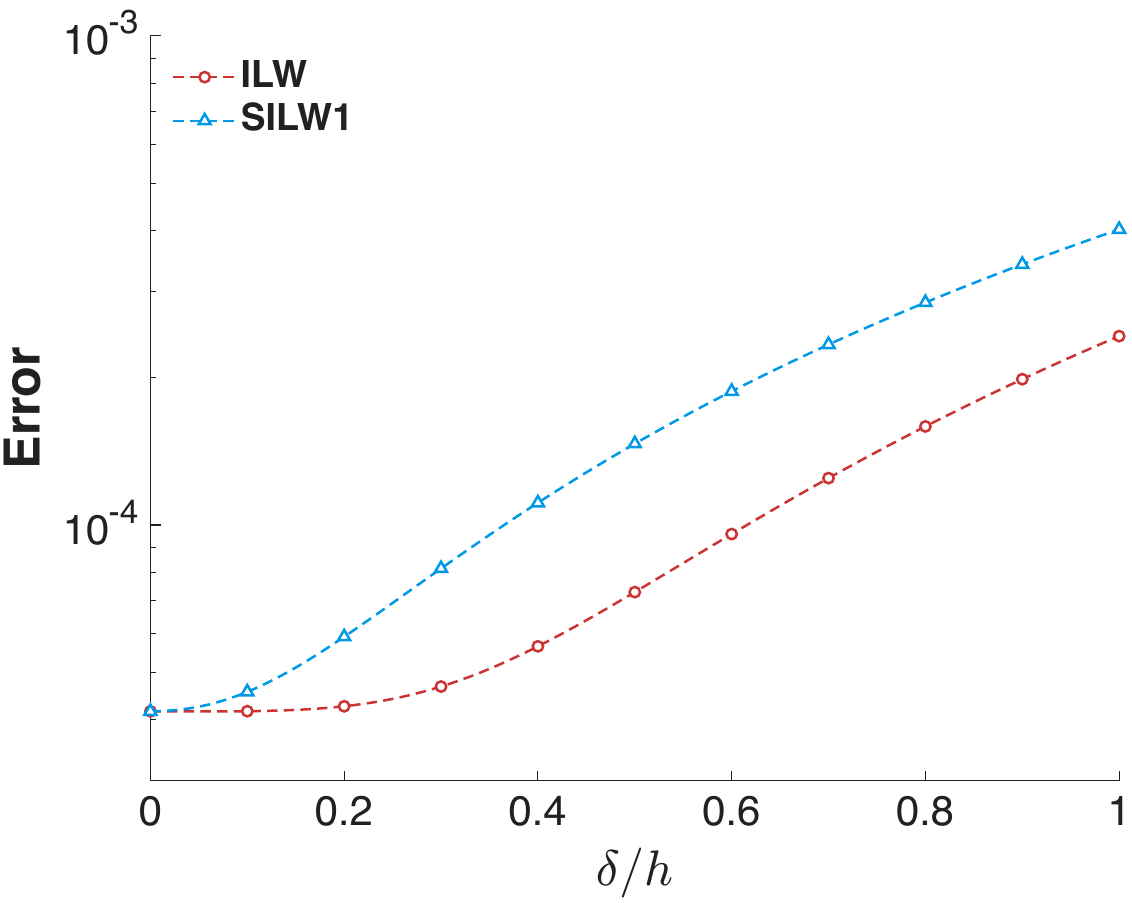}
        \caption{$P^1$}
    \end{subfigure}
    \hfill
    \begin{subfigure}{0.32\textwidth}
        \centering
        \includegraphics[width=\textwidth]{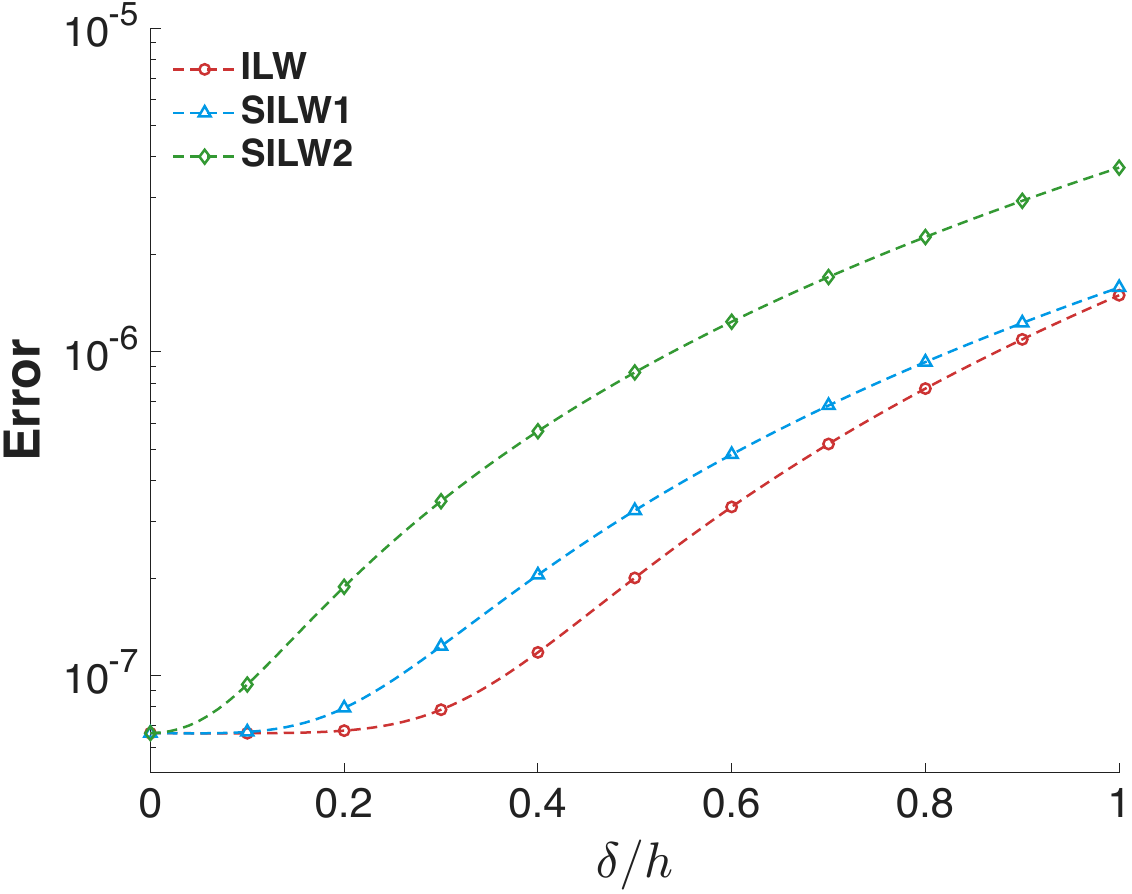}
        \caption{$P^2$}
    \end{subfigure}
    \hfill
    \begin{subfigure}{0.32\textwidth}
        \centering
        \includegraphics[width=\textwidth]{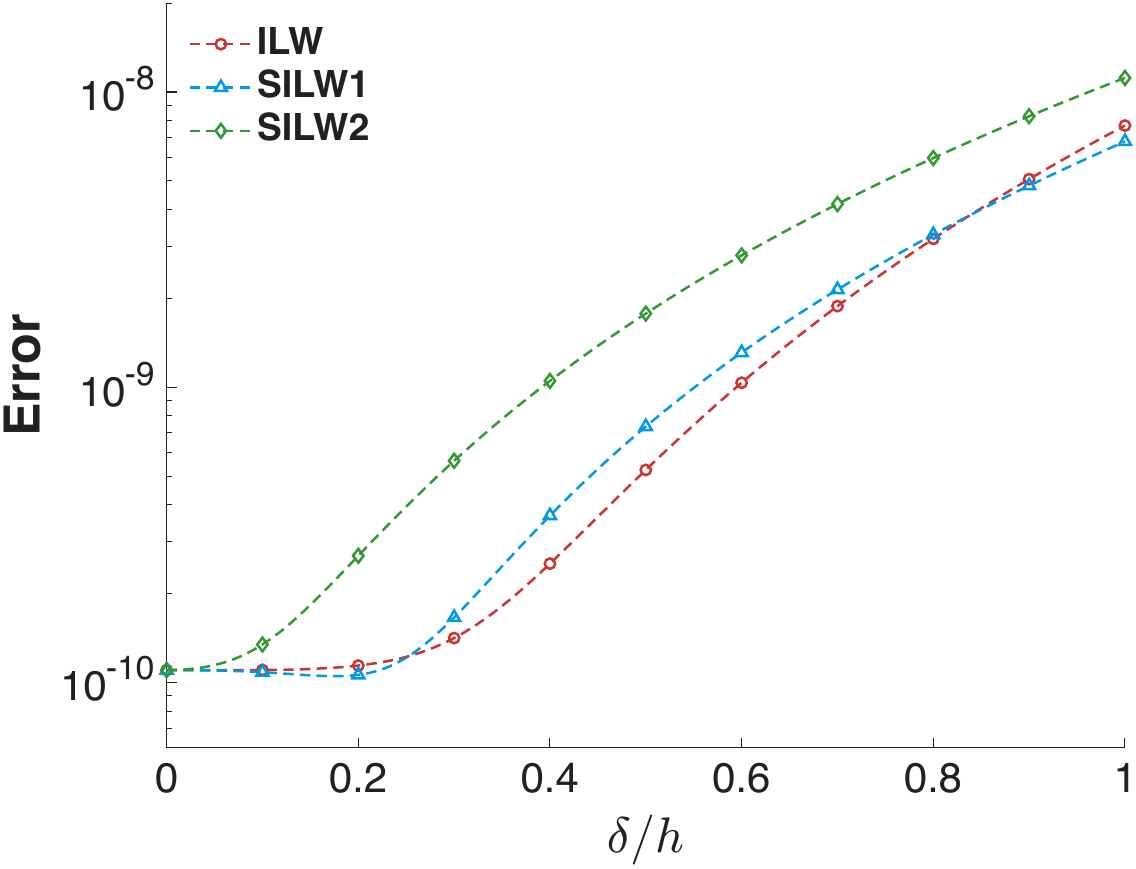}
        \caption{$P^3$}
    \end{subfigure}
    \caption{Errors of the test example with different $\delta/h$.}
    \label{error_ori}
\end{figure}
\FloatBarrier

\subsubsection{One-order-higher boundary reconstruction}
This straightforward procedure simply increases the accuracy order by one. Specifically, for a standard DG method using $P^k$ elements on interior cells, we employ the (S)ILW procedure to construct a reconstruction polynomial $p(x,t) \in P^{k+1}(\tilde{I}_0)$.

 We define ILW+, SILW1+ and SILW2+ as applying the one-order-higher (S)ILW boundary treatments \eqref{ILW}-\eqref{SILW2}. For ILW+, we need one more boundary term to construct the Taylor expansion
\begin{equation}
p(x,t) = \sum_{m=0}^{k+1} \frac{(x-a)^m}{m!} \partial_x^{m} u|_{x=a}.
\end{equation}
The SILW1+ procedure determines the unique polynomial $p(x,t) \in P^{k+1}$ on $\tilde{I}_0$ by satisfying:
\begin{equation}\label{SILW1+}
\begin{cases}
\partial_x^{m} p|_{x=a} = \partial_x^{m} u|_{x=a}, \quad m = 0, 1, \cdots, k,\\
\int_{I_1} p(x,t) \, dx = h (\bar{u}^{(0)})_1.
\end{cases}
\end{equation}
The SILW2+ polynomial is defined by matching $k$ boundary derivatives and two interior integral quantities:
\begin{equation}\label{SILW2+}
\begin{cases}
\partial_x^{m} p|_{x=a} = \partial_x^{m} u|_{x=a}, \quad m = 0, 1, \cdots, k-1, \\
\int_{I_1} p(x, t) \, dx = h (\bar{u}^{(0)})_1, \\ \int_{I_1} \partial_xp(x, t) \, dx = h(\bar{u}^{(1)})_1.
\end{cases}
\end{equation}
We apply the proposed method to the same test problem in Section \ref{subsec_SILW_ori}, and present the errors as a function of $\delta/h$ after applying the one-order-higher (S)ILW boundary treatment in Fig. \ref{error_add}. Compared with the original results, the magnitude of the error no longer exhibits significant growth as $\delta/h$ approaches $1$.

\begin{figure}[htbp]
    \centering
    \begin{subfigure}{0.32\textwidth}
        \centering
        \includegraphics[width=\textwidth]{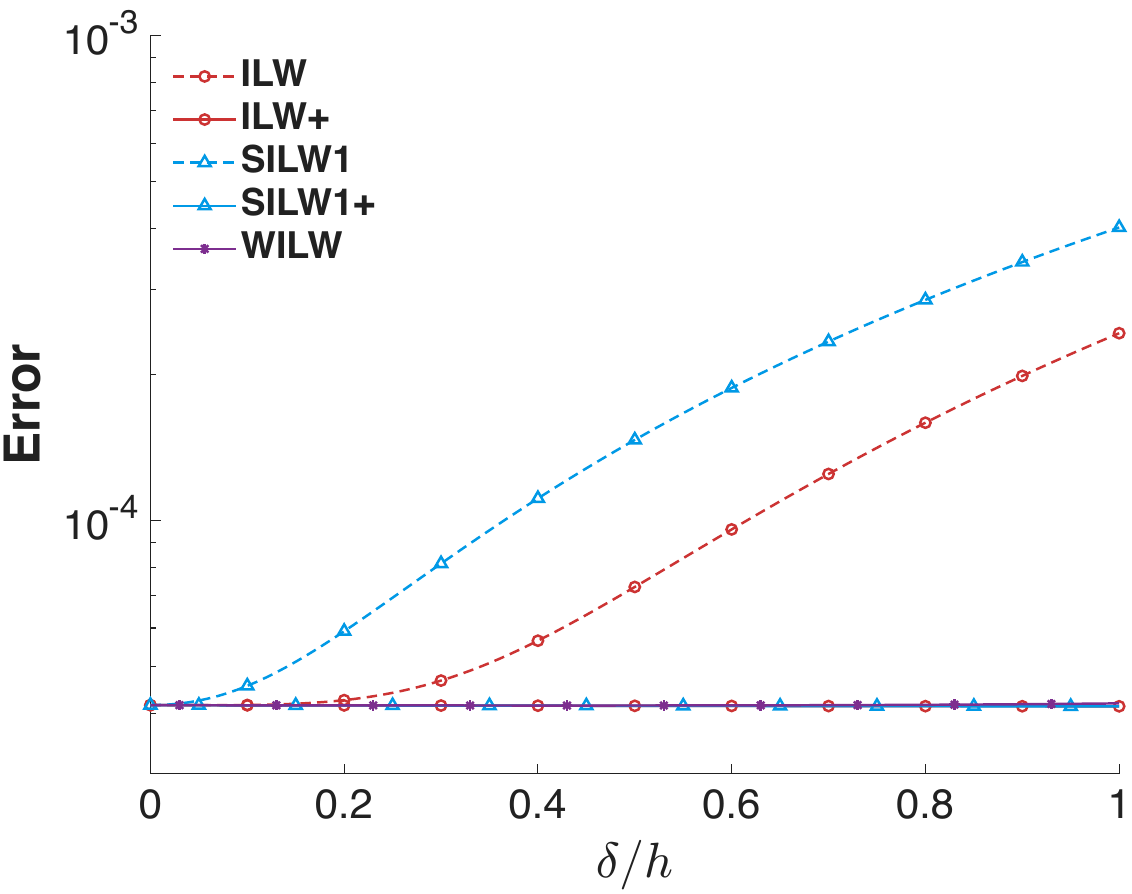}
        \caption{$P^1$}
    \end{subfigure}
    \hfill
    \begin{subfigure}{0.32\textwidth}
        \centering
        \includegraphics[width=\textwidth]{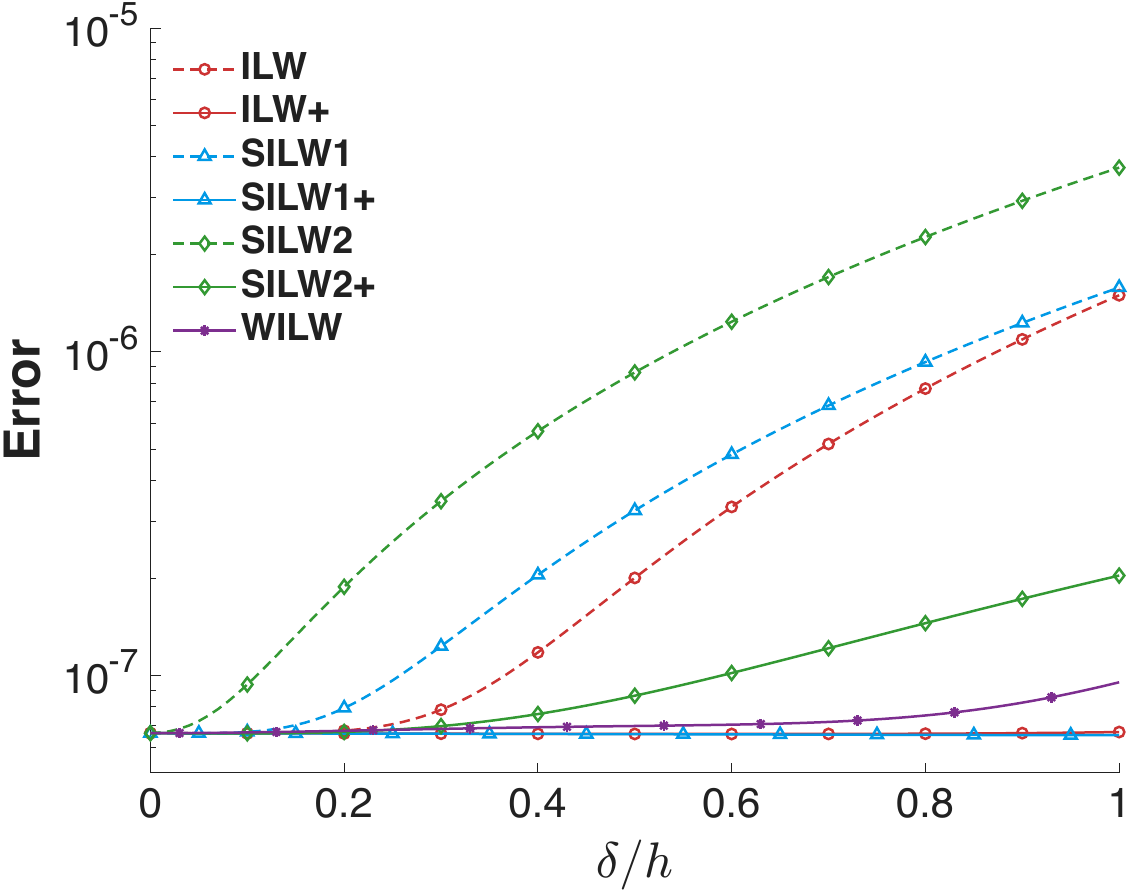}
        \caption{$P^2$}
    \end{subfigure}
    \hfill
    \begin{subfigure}{0.32\textwidth}
        \centering
        \includegraphics[width=\textwidth]{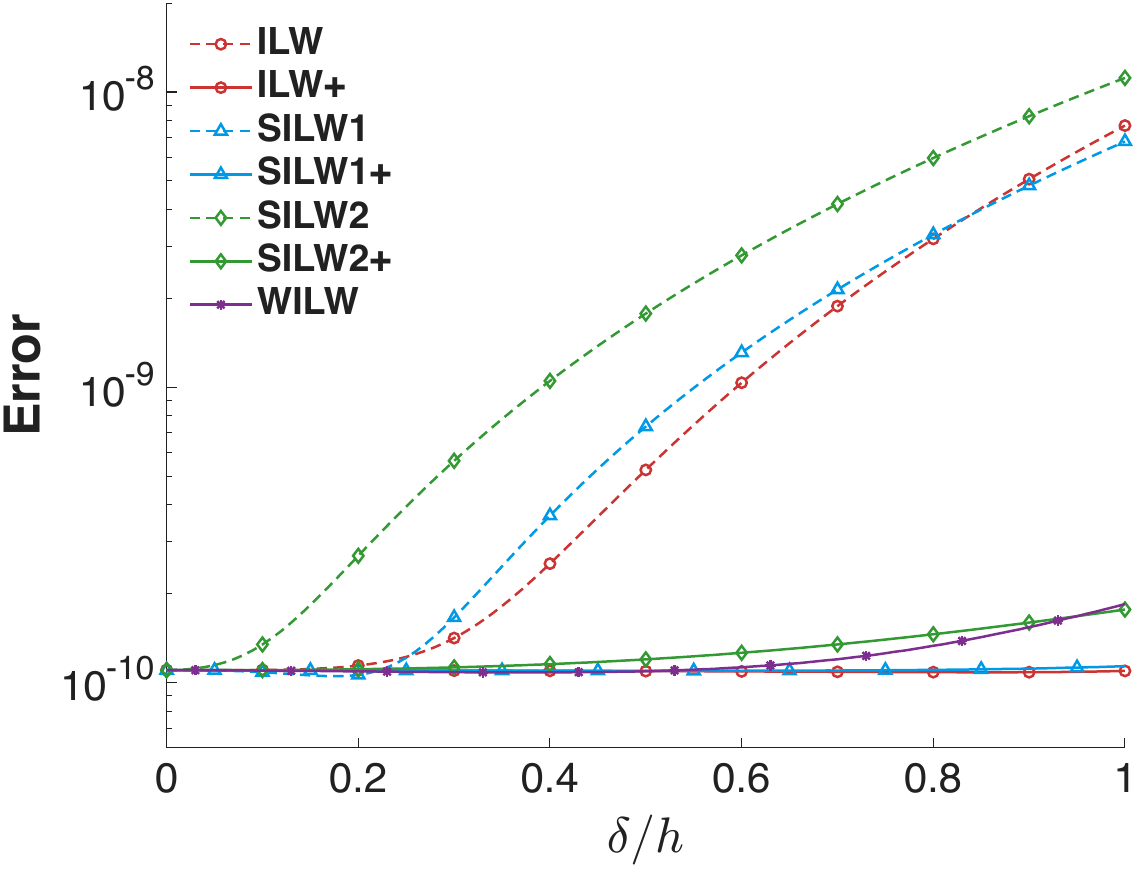}
        \caption{$P^3$}
    \end{subfigure}
    \caption{Errors of the test example with different $\delta/h$. ``+" represents the (S)ILW boundary treatment with one-order-higher accuracy.}
    \label{error_add}
\end{figure}

However, adding boundary derivatives increases computational cost. Next, we propose a new SILW method that uses information from two interior cells, minimizing the need for high-order boundary derivatives. 

We remark that, using one additional piece of information from the interior cell instead of the boundary derivatives makes the scheme more efficient, but potentially unstable. For example, when $k=1$, we can reconstruct $p(x,t)\in P^2$ satisfying 
\begin{equation*}
p|_{x=a} = u|_{x=a}, \quad 
\int_{I_1} p(x, t) \, dx = h (\bar{u}^{(0)})_1, \quad 
\int_{I_1} \partial_xp(x, t) \, dx = h(\bar{u}^{(1)})_1.
\end{equation*}
However, the scheme becomes unstable as $\delta$ approaching 0.

\subsection{Weighted ILW boundary treatment}
As discussed above, using fewer boundary terms $\partial^{(m)}_x u|_{x=a}$ enhances the efficiency of the scheme. However, extracting more information from a single interior cell without extra boundary terms yields instability. To balance the efficiency and stability of the algorithm, we incorporate information from two interior cells and determine the required degrees of freedom using a stable least squares method.

For example, considering $k=2$, our scheme requires constructing a cubic polynomial on $\tilde{I}_0$. We impose one boundary constraint along with two constraints each from $I_1$ and $I_2$ as follows:
\begin{equation}\label{constraints(k=2)}
\begin{cases}
p(a,t) = u|_{x=a}, \\
\int_{I_i} \partial_x^{m} p(x,t) \, dx = h (\bar{u}^{(m)})_i,\quad m=0,1,~i=1,2.
\end{cases}
\end{equation}
Compared with SILW1+ \eqref{SILW1+} and SILW2+ \eqref{SILW2+}, this construction reduces the number of required boundary derivatives by two and one, respectively. However, numerical results show that the fully discrete DG scheme is unstable if solving this overdetermined system directly via least squares. 

We observe that using cell averages of higher-order derivatives of the interior cells imposes more stringent stability requirements on the algorithm. Therefore, we assign appropriate positive weights $C_1$, $C_2$, $\tilde C_1$ and $\tilde C_2$ to the constraint equations and determine the reconstruction polynomial $p(x,t)\in P^{3}(\tilde{I}_0)$ by
\begin{align*}
    p(x,t)=\operatorname*{argmin}_{q(x,t)\in P^3}&\left\{
    (q(a,t)-u(a,t))^2 + \sum_{i=1}^2C_i\left(\int_{I_i}q(x,t)\,dx - h(\bar{u}^{(0)})_i\right)^2 \right.\\
    &\left. + \sum_{i=1}^2\tilde{C}_i\left(\int_{I_i}\partial_x q(x,t)\,dx - h(\bar{u}^{(1)})_i\right)^2
    \right\}.
\end{align*}
This leads to a new ILW method: the weighted ILW (WILW) method. Note that the choice of weights is not unique. Numerical experiments indicate that weight choices of the form $[C_1, C_2, \tilde C_1, \tilde C_2]=[\mathcal O(1), \mathcal O(1), \mathcal O(h^2), \mathcal O(h^2)]$, mitigating the impact of higher‑order derivatives, often achieve the desired stability and accuracy properties.

Through our experiments, the proposed idea can be extended to schemes with different $k$, yielding stable and efficient results. For the case  $k=1$, we utilize the constraints $u(a,t)$, $(\bar{u}^{(0)})_1$, and $(\bar{u}^{(0)})_2$ to construct $p(x,t)\in P^2(\tilde{I}_0)$:
\begin{equation}\label{constraints(k=1)}
\begin{cases}
p(a,t) = u|_{x=a}, \\
\int_{I_i} p(x,t) \, dx = h (\bar{u}^{(0)})_i,\quad i=1,2.
\end{cases}
\end{equation}
In this case, weighting is not required. For $k=3$, we utilize the constraints $u(a,t)$, $u_x(a,t)$, and $(\bar{u}^{(m)})_i,\, m=0,1,\, i=1,2$, to construct $p(x,t)\in P^4(\tilde{I}_0)$:
\begin{equation}\label{constraints(k=3)}
\begin{cases}
\partial_x^{m} p|_{x=a} = \partial_x^{m} u|_{x=a}, \quad m = 0, 1,\\
\int_{I_i} \partial_x^{m} p(x,t) \, dx = h (\bar{u}^{(m)})_i,\quad m=0,1,~i=1,2.
\end{cases}
\end{equation}
We assign $\mathcal{O}(1)$ weights to the equations for $(\bar{u}^{(0)})_1$ and $(\bar{u}^{(0)})_2$, and $\mathcal{O}(h^2)$ weights to those for $(\bar{u}^{(1)})_1$ and $(\bar{u}^{(1)})_2$. To avoid introducing adjustable parameters, we adopt the simplest choice $[1,1,h^2,h^2]$ in all numerical examples. Although further tuning of  the $\mathcal{O}(h^2)$ coefficients may yield better results, this choice is already sufficient to meet expectations. In fact, for $k=4$ and higher orders, similarly incorporating high-order derivative information from $I_1$ and $I_2$ while reducing the reliance on boundary information under the same weighting strategy also yields good numerical results.

For the same test example, Fig. \ref{error_add} illustrates the errors versus $\delta/h$ using the higher-order WILW treatment. In comparison, the WILW method shows good numerical performance but requires less computational effort for high-order spatial derivatives.

\subsection{Boundary treatment at the outflow boundary}\label{1D Outflow}
Next, we discuss the treatment of outflow boundary. Characteristic theory implies downwind information does not affect the interior of the domain. For scalar equations, the numerical flux $\hat{f}_{N+1/2}=f(u_h(x_{N+1/2}^-),t)$ uses only interior values.Note that for systems, inflow and outflow conditions are coupled according to the signs of the Jacobian eigenvalues. Consequently, we must obtain information for the outflow components at the boundary.

At the outflow boundary $x=b$, the solution is typically approximated by extending the interior polynomial $u_h$ from $I_N$ to $\tilde I_{N+1} = [x_{N+1/2}, b]$:
\begin{equation}
p(x, t) = \left. u_h \right|_{I_N}(x), \quad x \in \tilde I_{N+1}. 
\end{equation}
However, direct extension still makes numerical errors sensitive to $\delta/h$, especially for systems. Since inflow and outflow conditions mix at the boundary, errors from outflow variables can transmit into inflow variables, through the boundary treatment, even though the inflow part prevents significant error growth. Therefore, we follow the WILW approach and reconstruct $p(x,t) \in P^{k+1}(\tilde{I}_{N+1})$ using information from cells $I_N$ and $I_{N-1}$.

For the case $k=1$, we employ the constraints $\{ (\bar{u}^{(m)})_i \}_{i=N-1,N}^{m=0,1}$:
\begin{equation}
\int_{I_i}  \partial_x^{m} p(x,t) \, dx = h (\bar{u}^{(m)})_i, \quad m=0,1, \quad i=N-1,N.
\end{equation}
These constraints are weighted by $\mathcal{O}(1)$ for $m=0$ and $\mathcal{O}(h^2)$ for $m=1$.
For $k=2$, we impose the constraints $\{ (\bar{u}^{(m)})_i \}_{i=N-1,N}^{m=0,1,2}$:
\begin{equation}
\int_{I_i}  \partial_x^{m} p(x,t) \, dx = h (\bar{u}^{(m)})_i, \quad m=0,1,2, \quad i=N-1,N.
\end{equation}
The corresponding weights are $\mathcal{O}(1)$, $\mathcal{O}(h)$, and $\mathcal{O}(h^2)$ for $m=0, 1, 2$, respectively.
For $k=3$, we use $\{ (\bar{u}^{(m)})_i \}_{i=N-1,N}^{m=0,1,2,3}$:
\begin{equation}
\int_{I_i}  \partial_x^{m} p(x,t) \, dx = h (\bar{u}^{(m)})_i, \quad m=0,1,2,3, \quad i=N-1,N.
\end{equation}
The weights are $\mathcal{O}(1)$ for $m=0,1$, $\mathcal{O}(h)$ for $m=2$, and $\mathcal{O}(h^2)$ for $m=3$.

More generally, we utilize information from both $I_{N-1}$ and $I_N$, where lower-order derivative constraints are assigned larger weights, while higher-order ones are assigned progressively smaller weights. Finally, the numerical solution $u_h$ on $\tilde{I}_{N+1}$ is defined by $p(x,t)$.

\subsection{Stability analysis}
In this section, we analyze the linear stability of the WILW methods for both semi-discrete and fully discrete cases through eigenvalue spectrum analysis of the corresponding discretized operators \cite{eigenvalue_spectrum}. To simplify the discussion on inflow treatment, we consider a linear scalar conservation law by setting $f(u) = u$, $a = 0$, and $g(t) = 0$ in \eqref{1D Scalar Equation}. The mesh configuration follows \eqref{1Dmesh} with $\delta_1 = \delta \in [0, h)$ and $\delta_2 = 0$. We employ the DG schemes with $V_h^k$ ($k = 1, 2, 3$) and use the standard upwind flux for interior interfaces: $\hat{u}_{j+1/2} = u_h(x_{j+1/2}^-)$ for $j = 1, \dots, N$. The boundary flux $\hat{u}_{1/2}$ is then determined by the proposed WILW methods.

For the DG schemes with $V_h^k$, we define $k + 1$ equally spaced nodes in each cell $I_j$:
$$\hat{x}_{j}^{(i)} = x_{j} + \frac{2i - k}{2(k + 1)}h, \quad i = 0,1,\cdots, k.$$
Using a local Lagrange basis, the standard semi-discrete DG scheme \eqref{semi} with the upwind flux at interior interfaces can be written as
\begin{equation}
    \frac{\mathrm{d}\mathbf{U}_{j}}{\mathrm{d}t} = \frac{1}{h} \left( \mathbf{A} \mathbf{U}_{j-1} + \mathbf{B} \mathbf{U}_{j} \right), \quad j = 2,3, \cdots, N,
\end{equation}
where $\mathbf{U}_{j}$ denotes the vector of nodal values of $u_h$ at the local nodes $\{\hat x_j^{(i)}\}_{i=0}^k$ on $I_j$. The matrices $\mathbf A$ and $\mathbf B$ are constant matrices of size $(k+1) \times (k+1)$.

For the proposed WILW method, the numerical flux $\hat{u}_{1/2}$ is determined by $u_h$ in $I_1$ , $I_2$ and the zero boundary condition. The semi-discrete scheme on $I_1$ can then be written as:
$$\frac{\mathrm{d}\mathbf{U}_1}{\mathrm{d}t} = \frac{1}{h}(\mathbf{C} \mathbf{U}_{1}+\mathbf{D} \mathbf{U}_{2}),$$
where $\mathbf{C}$ and $\mathbf{D}$ are $(k+1) \times (k+1)$ constant matrices. Finally, the semi-discrete scheme yields a linear system in matrix-vector form:
\begin{equation}\label{ESV_ODE1}
    \frac{\mathrm{d}\mathbf{U}}{\mathrm{d}t} = \frac{1}{h} \mathbf{Q} \mathbf{U},
\end{equation}
where $\mathbf{U} = ( \mathbf{U}_1^{\top}, \mathbf{U}_2^{\top}, \dots, \mathbf{U}_N^{\top} )^{\top}$ is the vector of nodal values, and the block matrix $\mathbf{Q}$ is defined as:
\begin{equation}\label{Q} 
\mathbf{Q} = \begin{pmatrix}
\mathbf{C} & \mathbf{D} & & & \\
\mathbf{A} & \mathbf{B} & & & \\
 & \mathbf{A} & \mathbf{B} & & \\
 & & \ddots & \ddots & \\
 & & & \mathbf{A} & \mathbf{B}
\end{pmatrix}.
\end{equation}
This system incorporates both the interior discretization and the inflow boundary treatment.

We apply the normal mode analysis to \eqref{ESV_ODE1} to obtain the eigenvalue problem. Assuming a
solution of the form $u(x,t) = e^{st}u_0(x)$ and $\tilde{s} =hs$, the semi-discrete scheme yields
$$\tilde{s}\mathbf{U}=\mathbf{Q}\mathbf{U},$$
indicating $\tilde{s}$ is the eigenvalue of $\mathbf{Q}$. The semi-discrete scheme with the prescribed boundary conditions is stable if the entire eigenvalue spectrum of the coefficient matrix $\mathbf{Q}$ lies in the left half-plane, i.e., $Re(\tilde{s})\le0$. 

In contrast to the SILW method, whose coefficient matrix is block lower triangular, i.e., $\mathbf{D}=\mathbf{0}$ in \eqref{Q}, the matrix $\mathbf{Q}$ of the WILW method is no longer strictly block lower triangular. Owing to the block structure of $\mathbf{Q}$, its determinant can be factorized as 
$$\det(\mathbf{Q})=\det
\begin{pmatrix}
\mathbf{C} & \mathbf{D} \\
\mathbf{A} & \mathbf{B}
\end{pmatrix}
\cdot\det
\begin{pmatrix}
\mathbf{B} &  \\
\mathbf{A} & \mathbf{B} \\
 & \ddots & \ddots \\
 &  & \mathbf{A} & \mathbf{B}
\end{pmatrix}.$$
Notice that the matrix $\mathbf{B}$ corresponds to the standard DG discretization for interior cells, and its eigenvalues $\kappa^B$ are well-known to satisfy the stability condition. Therefore, we only need to focus on the eigenvalues of the coupled boundary matrix:
\begin{equation}\label{CDAB}
\begin{pmatrix}
\mathbf{C} & \mathbf{D}\\
\mathbf{A} &\mathbf{B}\end{pmatrix}.
\end{equation} 

Another distinction is that, for $k=2,3$, the WILW method employs weighted least-squares reconstruction. The resulting coefficients depend not only on the ratio $\eta = \delta/h \in [0,1)$ but also explicitly on the mesh size $h$. To elucidate this dependency, a detailed explanation is provided for the $k = 2$ case. According to \eqref{constraints(k=2)}, the reconstruction polynomial is given by  
$$p(x,t) = u(a,t) + R(t)(x-a) + Q(t)(x-a)^2 + S(t)(x-a)^3,$$
where the coefficients $R(t)$, $Q(t)$, and $S(t)$ are determined via a weighted least-squares procedure. Under zero boundary condition, the interface value can be expressed as a linear combination of the cell moments:  
$$\hat{u}_{1/2}=p(x_{1/2}, t)=\delta R(t)+\delta^2 Q(t)+\delta^3S(t) = 
L_1 (\bar{u}^{(0)})_1+L_2 h(\bar{u}^{(1)})_1+L_3 (\bar{u}^{(0)})_2+L_4 h(\bar{u}^{(1)})_2.$$
Here, the coefficients $L_1, L_2, L_3,$ and $L_4$ depend on both $\eta$ and $h$ as shown in Appendix \ref{MatrixQ}. This implies that the $3 \times 3$ matrices $\mathbf{C}$ and $\mathbf{D}$ depend on both $h$ and $\eta$. The scheme is regarded as stable if $Re(\tilde{s})\leq 0$ for all $\eta \in[0,1)$ and $h>0$.

The following theorem establishes the semi-discrete linear stability of the DG method with WILW boundary treatment.

\begin{theorem}
    For the linear scalar conservation law \eqref{1D Scalar Equation}, all eigenvalues of the semi-discrete matrix $\mathbf{Q}$ of the DG method with WILW boundary treatment
    for $k=1,2,3$ lie in the left half-plane. 
    Hence, the semi-discrete scheme \eqref{ESV_ODE1} is stable.
\end{theorem}

\begin{proof}
We first consider the case $k=1$, which does not involve weighted least squares. The boundary treatment matrices $\mathbf{C}$ and $\mathbf{D}$ therefore depend only on $\eta\in[0,1)$ and are given by
\begin{equation}\label{AB_WILW_1}
\renewcommand{\arraystretch}{1.25}
\mathbf{A}=\begin{pmatrix}
-\frac{5}{4} & \frac{15}{4} \\
\frac{1}{4} & -\frac{3}{4}
\end{pmatrix},
\quad
\mathbf{B}=\begin{pmatrix}
-\frac{7}{4} & -\frac{3}{4} \\
\frac{11}{4} & -\frac{9}{4}
\end{pmatrix},
\end{equation}
and
\begin{equation}\label{CD_WILW_1}
\renewcommand{\arraystretch}{1.25}
\mathbf{C}=\mathbf{B}+\frac{\eta(9\eta+14)}
{2(3\eta^2+6\eta+2)}
\begin{pmatrix}
\frac{5}{4} & \frac{5}{4} \\
-\frac{1}{4} & -\frac{1}{4}
\end{pmatrix},
\quad
\mathbf{D}=\frac{-3\eta^2-2\eta}{2(3\eta^2+6\eta+2)}
\begin{pmatrix}
\frac{5}{4} & \frac{5}{4} \\
-\frac{1}{4} & -\frac{1}{4}
\end{pmatrix}.
\end{equation}
The characteristic polynomial of the matrix \eqref{CDAB} is given by
$$\chi(s)=s^4+\frac{39\eta^2+82\eta+32}{d(\eta)}s^3
+\frac{72\eta^2+192\eta+112}{d(\eta)}s^2+\frac{144\eta+192}{d(\eta)}s+\frac{144}{d(\eta)},$$
where $d(\eta)=6\eta^2+12\eta+4$. To analyze the eigenvalue distribution, we employ the Routh--Hurwitz criterion (Appendix~\ref{Routh-Hurwitz Criterion}). The entries in the first column of the corresponding Routh table are
\begin{align*}
R_{1,1}
&=1,\\
R_{2,1}
&=
\frac{39\eta^2+82\eta+32}
     {6\eta^2+12\eta+4},\\
R_{3,1}
&=
\frac{
1404\eta^4+6264\eta^3+9768\eta^2+6224\eta+1408
}{
117\eta^4+480\eta^3+666\eta^2+356\eta+64
},\\
R_{4,1}
&=
\frac{
25272\eta^5+132759\eta^4+268596\eta^3+263484\eta^2
+127488\eta+24576
}{
1053\eta^6+6804\eta^5+17424\eta^4+22452\eta^3
+15276\eta^2+5224\eta+704
},\\
R_{5,1}
&=
\frac{72}
     {3\eta^2+6\eta+2}.
\end{align*}
Since $\eta\in[0,1)$, all first-column entries are positive. Hence, no sign changes occur in the first column of the Routh table.
By the Routh--Hurwitz criterion, all roots of the characteristic
polynomial lie in the left half-plane.

For $k=2,3$, the boundary treatment matrices $\mathbf{C}$ and $\mathbf{D}$ depend on both $\eta\in[0,1)$ and $h$. Without loss of generality, we assume $h\in(0,1)$. The explicit expressions for the matrices $\mathbf{A}$, $\mathbf{B}$, $\mathbf{C}$, and $\mathbf{D}$ are provided in Appendix~\ref{MatrixQ}. Similarly, we construct the corresponding Routh table to analyze the eigenvalue distribution of the matrix \eqref{CDAB}. Symbolic computations show that the first column entries can be expressed as rational functions of
$$R_{i,1}=\frac{b_i(\eta,h)}{a_i(\eta,h)},\quad 2\le i\le 2k+3.$$
These expressions are extremely complex and involve a mixture of addition and subtraction terms, and we aim to prove that they remain positive. To avoid numerical errors, we adopt a robust algebraic approach using the Bernstein polynomial basis expansion to verify that $a_i(\eta, h)$ and $b_i(\eta, h)$ share the same sign and lack real roots in the domain. By transforming from the standard monomial basis to the Bernstein basis, the task of verifying continuous positivity over $(\eta,h)\in[0,1)\times(0,1)$ reduces to checking the signs of discrete Bernstein coefficients. Appendix~\ref{Bernstein Basis Expansion} provides further details on this approach. For $k=2,3$, rigorous symbolic computations establish that all Bernstein coefficients of the numerator and denominator polynomials are strictly positive, implying $R_{i,1}>0$ for $2\le i\le 2k+3$. It follows from the Routh--Hurwitz criterion that all eigenvalues lie in the left half-plane.

\end{proof}

Furthermore, we analyze the stability of the full discrete scheme with the third-order TVD RK time
discretization \eqref{RK}. The fully discrete system can be expressed in the following matrix-vector form:
\begin{equation}\label{Fully-discrete-scheme}
\mathbf{U}^{n+1}=g(\mathbf{Q})\mathbf{U}^n,
\end{equation}
where $g(\mathrm{X})$ is a matrix function
$$g(\mathbf{X}) = \mathbf{I} + \lambda \mathbf{X} + \frac{1}{2!} (\lambda \mathbf{X})^2 + \frac{1}{3!} (\lambda \mathbf{X})^3,$$
and $\lambda=\frac{\Delta t}{h}=\frac{1}{2k+1}$ is the CFL number. The full-discrete scheme is stable if the entire eigenvalue spectrum of $g(\mathbf{Q})$ is contained within the unit disk, i.e., $\rho(g(\mathbf{Q}))\le1$. As in the semi-discrete case, we only need to focus on the eigenvalues of
\begin{equation}\label{gCDAB}
g\left(\begin{matrix}
\mathbf{C} & \mathbf{D} \\
\mathbf{A} & \mathbf{B}
\end{matrix}\right) .
\end{equation}

To verify that the spectrum of \eqref{gCDAB} lies within the unit disk, we apply a M\"{o}bius transformation to the characteristic polynomial $\chi(z)$, which defines a bijective mapping on the extended complex plane. This transformation maps roots within the unit disk onto the left half-plane, allowing the Routh--Hurwitz criterion to be employed for the analysis.

\begin{remark}
The M\"{o}bius transformation $s = \frac{z - 1}{z + 1}$ maps the interior of the unit disk to the left half-plane. For any $z \in \mathbb{C} \setminus \{-1\}$, a direct calculation yields:
$$\operatorname{Re}(s) = \frac{|z|^2 - 1}{|z + 1|^2}.$$
It follows that the sign of $\text{Re}(s)$ is determined solely by $|z|^2 - 1$. Specifically:
\begin{itemize}
    \item $|z| < 1 \iff \operatorname{Re}(s) < 0$ (unit disk to left half-plane);
    \item $|z| > 1 \iff \operatorname{Re}(s) > 0$ (disk exterior to right half-plane);
    \item $|z| = 1 \iff \operatorname{Re}(s) = 0$ (unit circle to imaginary axis).
\end{itemize}
\end{remark}

We transform the $n$-th degree characteristic polynomial $\chi(z)$ into $\hat{\chi}(s)$ via the M\"{o}bius transformation:
\begin{equation}\label{chi_s}
\hat{\chi}(s) = (1-s)^n \chi\left(\frac{1+s}{1-s}\right).
\end{equation}
The factor $(1-s)^n$ clears the denominators, ensuring $\hat{\chi}(s)$ is a polynomial of degree $n$. Consequently, the Schur stability of $\chi(z)$ (roots within the unit disk) is equivalent to the Hurwitz stability of $\hat{\chi}(s)$ (roots in the left half-plane), which is verified using the Routh--Hurwitz criterion.

The following theorem establishes the fully discrete linear stability of the DG method with WILW boundary treatment.
\begin{theorem}
For the linear scalar conservation law \eqref{1D Scalar Equation}, all eigenvalues of the fully discrete matrix $g(\mathbf{Q})$ associated with the DG method with WILW boundary treatment and the third-order TVD Runge--Kutta method \eqref{RK} lie within the unit disk under the CFL condition $\Delta t=\frac{h}{2k+1}$ for $k=1,2,3$.
\end{theorem}

\begin{proof}
For $k=1$, we first construct \eqref{gCDAB} using \eqref{AB_WILW_1} and \eqref{CD_WILW_1}, and apply the M\"{o}bius transformation to its characteristic polynomial $\chi(z)$ to obtain $\hat{\chi}(s)$ in \eqref{chi_s}. The corresponding Routh table has first-column entries (listed in Appendix~\ref{Routh-Table}) that are positive for all $\eta\in[0,1)$. It therefore follows from the Routh--Hurwitz criterion that all roots of $\hat{\chi}(s)$ lie in the left half-plane, and hence all eigenvalues of \eqref{gCDAB} lie within the unit disk.

For $k=2,3$, we repeat the above procedure using the explicit expressions of $\mathbf{A}$, $\mathbf{B}$, $\mathbf{C}$, and $\mathbf{D}$ provided in Appendix~\ref{MatrixQ} to construct the corresponding Routh tables, and employ the Bernstein polynomial basis expansion to establish the positivity of the first-column entries over $(\eta,h)\in[0,1)\times(0,1)$. Rigorous symbolic computations confirm this property. 

\end{proof}

\subsection{One-dimensional systems}
Next, we consider the one-dimensional compressible Euler equation:
\begin{equation}\label{Euler_equations}
    \mathbf{U}_t + \mathbf{F}(\mathbf{U})_x = 0, \quad x \in (a, b), \, t > 0,
\end{equation}
with appropriate boundary conditions and initial conditions. The conservative variables $\mathbf{U}$ and the flux $\mathbf{F}(\mathbf{U})$ are defined as
$$    \mathbf{U} = \begin{pmatrix} U_1 \\[2pt]
    U_2 \\[2pt]
    U_3 \end{pmatrix} = \begin{pmatrix} \rho \\[2pt] \rho u \\[2pt] E \end{pmatrix},\quad  \mathbf{F}(\mathbf{U}) = \begin{pmatrix}
U_2 \\[2pt]
(\gamma - 1) U_3 + \frac{3 - \gamma}{2} \frac{U_2^2}{U_1} \\[2pt]
\left(\gamma U_3 - \frac{\gamma - 1}{2} \frac{U_2^2}{U_1}\right) \frac{U_2}{U_1}
\end{pmatrix}
= \begin{pmatrix}
\rho u \\[2pt]
\rho u^2 + p \\[2pt]
u(E + p)
\end{pmatrix},$$
And the variables $\rho$, $u$, $p$, and $E$ represent the density, velocity, pressure, and total energy, respectively. The equation of state is defined to close the system of equations,
$$ E=\frac{p}{\gamma-1}+\frac12\rho u^2$$
where the ratio of the specific heat $\gamma = 1.4$ for air. Let $\mathbf{A}(\mathbf{U}) = \partial \mathbf{F}(\mathbf{U}) / \partial \mathbf{U}$ denote the Jacobian matrix, which is diagonalizable as
$$\mathbf{A}(\mathbf{U}) = \mathbf{L}^{-1} \boldsymbol{\Lambda} \mathbf{L},\quad\boldsymbol{\Lambda}(\mathbf{U}) = \text{diag}\{u - c, u, u + c\},\quad c=\sqrt{\gamma p/\rho},$$
and
$$\mathbf{L}(\mathbf{U})=\begin{pmatrix}
\mathbf{l}_1(\mathbf{U}) \\[2pt]
\mathbf{l}_2(\mathbf{U}) \\[2pt]
\mathbf{l}_3(\mathbf{U})
\end{pmatrix}= \begin{pmatrix}
l_{1,1}(\mathbf{U}) & l_{1,2}(\mathbf{U}) & l_{1,3}(\mathbf{U}) \\[2pt]
l_{2,1}(\mathbf{U}) & l_{2,2}(\mathbf{U}) & l_{2,3}(\mathbf{U}) \\[2pt]
l_{3,1}(\mathbf{U}) & l_{3,2}(\mathbf{U}) & l_{3,3}(\mathbf{U})
\end{pmatrix}.$$

The number of required boundary conditions is determined by the signs of eigenvalues. We employ the superscript $*$ to present values located at $x = a$ to ease the notation. To illustrate the algorithmic procedure, we consider the subsonic inflow at the boundary $x = a$ as an example, characterized by $u^* - c^* < 0 < u^* < u^* + c^*$. Consequently, two boundary conditions need to be imposed at $x = a$:
$$U_1^* = g_1(t),\quad U_2^* = g_2(t), \quad t > 0.$$
Given the domain partition \eqref{1Dmesh}, we need to construct an approximation polynomial $p_i(x)\in P^{k+1}(\tilde{I}_0)$ and further obtain
\begin{equation}\label{1D Taylor expansion}
U_i(x_{1/2}^-,t)=p_i(x_{1/2},t)+\mathcal{O}(h^{k+2}),\quad i=1,2,3.
\end{equation}
To achieve this, we use the proposed WILW method on each component of $\mathbf{U}$. Therefore, it is necessary to determine the spatial derivatives $U_i^{(m)}$ at $x=a$ for $m = 0, 1, \cdots, k_d$. 
Based on the analysis for scalar problems, we take $k_d=0$ for $k=1,2$, and $k_d=1$ for $k=3$.

Define the characteristic variables via characteristic decomposition $\mathbf{V} = \mathbf{L}\mathbf{U} = (V_1, V_2, V_3)^{\top}$. 
At this boundary, $V_1$ represents the outflow component, while $V_2$ and $V_3$ are inflow components. Applying the treatment from Section \ref{1D Outflow} to $V_1$ yields the derivatives at $x=a$, $V_1^{*(m)}$ for $m=0,\dots,k$.  
For the other variables, we draw on the idea of inverse Lax-Wendroff and use boundary conditions and the PDE to construct constraints.

First, since the boundary state at $x=a$ is unknown, we approximate it with $\mathbf{U}_{1/2} = \mathbf{U}(x_{1/2}, t)$ for the local characteristic decomposition. Combining the prescribed boundary conditions with the extrapolated $V_1^*$ yields the linear system for the point value $\mathbf{U}^*$:
\begin{equation}
\begin{cases}U_1^{*} = g_1(t),\\[2pt] U_2^{*} = g_2(t),\\[2pt]
\mathbf{l}_1(\mathbf{U}_{1/2})\cdot \mathbf{U}^*= V_1^*.\\\end{cases}
\end{equation}

Next, we apply the inverse Lax-Wendroff procedure combining the PDE with the boundary conditions for $U_1$, $U_2$,
and the extrapolated $V_1^{*(1)}$, yielding the following linear system for the spatial derivatives $\mathbf{U}^{*(1)}$:
\begin{equation}\label{1DEuler_x}
    \mathcal{A}_1\mathbf{U}^{*(1)}=\mathbf{b}_1,
\end{equation}
where the coefficient matrix $\mathcal{A}_1$ and $\mathbf{b}_1$ are given by
$$\mathcal{A}_1 = 
\begin{pmatrix}
0 & 1 & 0 \\[2pt]
\frac{\gamma - 3}{2} \left( \frac{U_2^{*}}{U_1^{*}} \right)^2 & 
\frac{(3 - \gamma) U_2^{*}}{U_1^{*}} & 
\gamma - 1 \\[2pt]
l_{1,1}(\mathbf{U}_{1/2}) & l_{1,2}(\mathbf{U}_{1/2}) & l_{1,3}(\mathbf{U}_{1/2})
\end{pmatrix},\quad \mathbf{b}_1 = 
\begin{pmatrix}
- g_1'(t) \\[2pt]
- g_2'(t) \\[2pt]
V_1^{*(1)}
\end{pmatrix}.$$
Thus, we can obtain $\mathbf{U}^{*(1)}$ via solving the above linear system. 

Note that, by repeatedly applying the Euler equations, we construct a linear system for the high-order spatial derivatives $\mathbf{U}^{*(m)}$. The coefficient matrix depends solely on $\mathbf{U}^*$, while the right-hand side involves the derivatives of $g_1, g_2$, extrapolated $V_1$, and lower-order spatial derivatives. Hence, we can construct $p_i(x)$ via the SILW method. However, the SILW require more boundary terms, i.e., large $k_d$, and those high-order spatial derivatives for nonlinear systems are exceedingly complex, imposing a heavy computational burden. In contrast, the proposed WILW method minimizes the value of $k_d$, simplifying the calculation and enhancing efficiency while preserving numerical stability.

In certain cases, linear systems at outflow-inflow transitions may become ill-conditioned as characteristic speed approaches zero. To maintain robustness, we adopt the least-squares framework from \cite{tan2010}, which introduces an additional extrapolation equation for the corresponding characteristic direction to resolve directional ambiguity.

The complete algorithm for 1D Euler equations is summarized as follows.
\vspace{1em}
\begingroup
\captionsetup{hypcap=false}
\phantomsection
\captionof{algorithm}{WILW method for 1D Euler equations}
\label{1D algorithm}
\begin{enumerate}[label=\textbf{Step \arabic*.}, leftmargin=1.5cm]
     \item Compute the eigenvalues $\lambda_i(\mathbf{U}_{1/2})$ and left eigenvector matrix $\mathbf{L}(\mathbf{U}_{1/2})$ of the Jacobian matrix $\mathbf{A}(\mathbf{U}_{1/2})$. Identify the prescribed inflow boundary conditions $g_1(t)$ and $g_2(t)$ according to the signs of these eigenvalues.
    \item Following Section \ref{1D Outflow}, we construct polynomials for the outflow characteristic variables and extrapolate them to the boundary, yielding $V_1^{*(m)}$ for $m=0, \dots, k$.
    \item Determine the boundary state $\hat{\mathbf{U}}^*$ by solving the linear system coupling the prescribed boundary conditions and the extrapolated $V_1^*$.
    \item Apply the inverse Lax-Wendroff procedure to construct the linear system and solve for the spatial derivatives $\mathbf{U}^{*(m)}$.
    \item Construct the $(k+1)$-th degree polynomials $p_i(x, t)$ ($i=1, 2, 3$) via WILW, and evaluate them at $x_{1/2}$ according to Equation \eqref{1D Taylor expansion} to yield the interface values.
\end{enumerate}
\endgroup

\section{WILW method for two-dimensional conservation laws}
The proposed approach can be easily generalized to two-dimensional problems. We consider the two-dimensional compressible Euler equations with appropriate boundary and initial conditions:
\begin{equation}\mathbf{U}_t + \mathbf{F}(\mathbf{U})_x + \mathbf{G}(\mathbf{U})_y = 0,\quad (x,y)\in\Omega,\; t>0,\end{equation}
where the state vector and flux vectors are given by
\begin{equation}\label{2D_Euler_equation}
    \mathbf{U} =\begin{pmatrix}
U_1 \\[2pt]
U_2 \\[2pt]
U_3 \\[2pt]
U_4
\end{pmatrix}=
\begin{pmatrix}
\rho \\[2pt]
\rho u \\[2pt]
\rho v \\[2pt]
E
\end{pmatrix},\qquad
\mathbf{F}(\mathbf{U}) =
\begin{pmatrix}
\rho u \\[2pt]
\rho u^{2} + p \\[2pt]
\rho u v \\[2pt]
u(E + p)
\end{pmatrix},\qquad
\mathbf{G}(\mathbf{U}) =
\begin{pmatrix}
\rho v \\[2pt]
\rho u v \\[2pt]
\rho v^{2} + p \\[2pt]
v(E + p)
\end{pmatrix}.
\end{equation}
The variables $\rho,~u,~v,~p$ and $E$ represent the density, $x$-velocity, $y$-velocity, pressure and total energy, respectively. The equation of state is defined by
$$E=\frac{p}{\gamma-1}+\frac12\rho(u^2+v^2),$$
where $\gamma = 1.4$ for air at ordinary temperatures.

The physical domain $\Omega$ can feature an arbitrarily complex geometry. We assume that $\Omega$ is overlaid by a uniform unfitted Cartesian mesh with grid sizes 
\begin{equation}
    \Delta x=x_{i+\frac{1}{2}}-x_{i-\frac{1}{2}},\quad \Delta y=y_{j+\frac{1}{2}}-y_{j-\frac{1}{2}},
\end{equation}
 where $\Delta x$ and $\Delta y$ are assumed to be the same, i.e. $\Delta x = \Delta y =h$. Based on this background grid, the effective computational domain $\tilde{\Omega}$ is defined as the subset of grid cells entirely contained within the physical domain $\Omega$, that is $\tilde{\Omega}=\{K_{i,j}: K_{i,j}=(x_{i-\frac 1 2},x_{i+\frac 12}) \times  (y_{j-\frac 12},y_{j+\frac 12}) \subset \Omega\}$. As illustrated in Fig. \ref{2D Domain}, for a circular physical domain $\Omega$ overlaid by a uniform Cartesian mesh, the effective computational domain $\tilde{\Omega}$ (represented by the blue region) consists of all grid cells whose boundaries lie completely within $\Omega$.

Since the two-dimensional DG method requires exterior values at Gauss points $P_G$ on the boundary of the effective computational domain $\partial\tilde{\Omega}$ to construct numerical fluxes, we utilize the nearest point $P_B$ on the physical boundary $\partial\Omega$ to determine the required boundary information.

Similar to \cite{tan2010, tan2012, liu2024}, we first perform a local coordinate transformation to reduce the 2D problem to a 1D case, where a 1D WILW reconstruction is applied along the normal direction. For a given point $P_G$, we identify a corresponding foot point $P_B \in \partial\Omega$ such that the outward normal vector $\mathbf{n} = (\cos \theta, \sin \theta)$ at $P_B$ passes through $P_G$, as illustrated in Fig. \ref{Local coordinate system}. A local coordinate system is then established at $P_B$ as the origin, where the $\hat{x}$-axis is aligned with $\mathbf{n}$ and the $\hat{y}$-axis is oriented tangentially to the boundary $\partial\Omega$. This coordinate rotation is expressed as
\begin{equation}\label{coordinate rotation}
\begin{pmatrix} \hat{x} \\[2pt] \hat{y} \end{pmatrix} = \begin{pmatrix} \cos\theta & \sin\theta \\[2pt] -\sin\theta & \cos\theta \end{pmatrix} \begin{pmatrix} x \\[2pt] y \end{pmatrix}.
\end{equation}
In this local frame, $\delta = |P_G - P_B|$ represents the offset distance between the computational and physical boundaries. Based on our selection strategy for the effective computational domain, the offset distance satisfies $\delta \in [0, \hat{h})$, where $\hat{h} = \sqrt{\Delta x^2 + \Delta y^2}$. Since the WILW reconstruction requires information from interior cells, we define two cells $(P_G, Q)$ and $(Q, R)$ with a length of $\hat{h}$ along the $-\hat{x}$ direction starting from $P_G$. Consequently, the ratio $\delta/\hat{h} \in [0, 1)$ matches the configuration of the one-dimensional case.

\begin{figure}[htbp]
    \centering 
    \begin{subfigure}{0.40\textwidth}
        \centering
        \includegraphics[width=\textwidth]{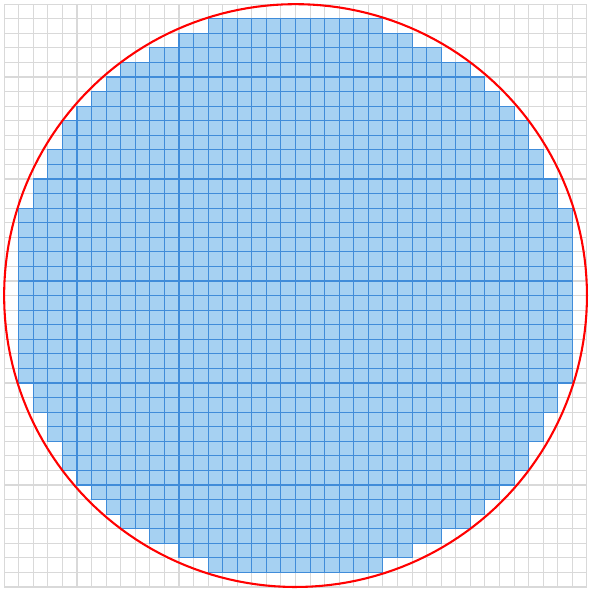}
        \caption{Circular domain}
        \label{2D Domain}
    \end{subfigure}
    \hspace{1cm}
    \begin{subfigure}{0.40\textwidth}
        \centering
        \includegraphics[width=\textwidth]{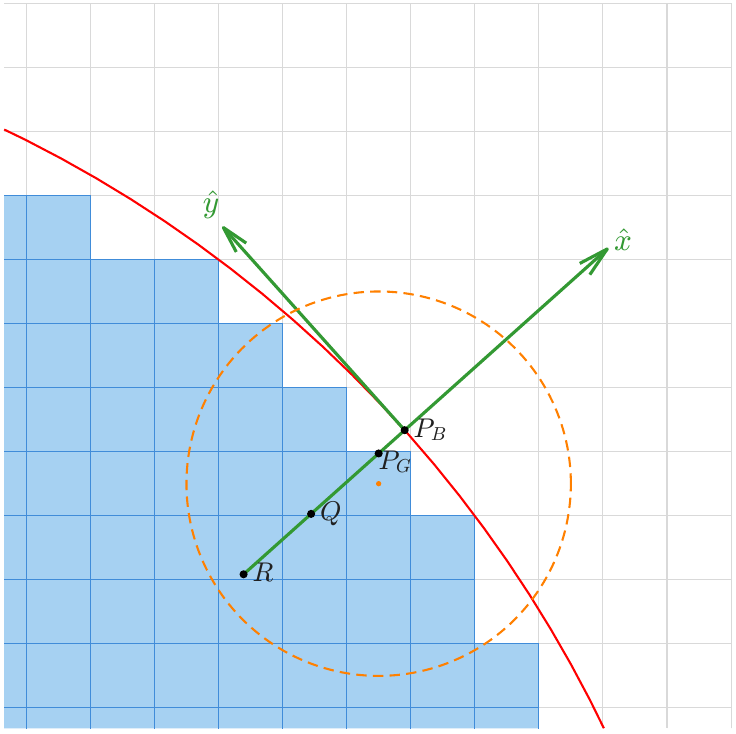}
        \caption{Local coordinate system}
        \label{Local coordinate system}
    \end{subfigure}   
    \caption{Schematic of the 2D algorithm.}
    \label{Schematic of the 2D algorithm}
\end{figure}

In this local coordinate system, the Euler equations are expressed as
\begin{equation}
\mathbf{\hat{U}}_t + \mathbf{F}(\mathbf{\hat{U}})_{\hat{x}} + \mathbf{G}(\mathbf{\hat{U}})_{\hat{y}} = \mathbf{0}.
\end{equation}
The state vector $\hat{\mathbf{U}}$ and the transformed velocity components $(\hat{u}, \hat{v})$ are defined as
$$\hat{\mathbf{U}} = \begin{pmatrix} \hat{U}_1 \\ \hat{U}_2 \\ \hat{U}_3 \\ \hat{U}_4 \end{pmatrix}= \begin{pmatrix} \rho \\ \rho \hat{u} \\ \rho \hat{v} \\ E \end{pmatrix}, \qquad \begin{pmatrix} \hat{u} \\ \hat{v} \end{pmatrix} = \begin{pmatrix} \cos\theta & \sin\theta \\ -\sin\theta & \cos\theta \end{pmatrix} \begin{pmatrix} u \\ v \end{pmatrix}.
$$
We denote the Jacobian matrices of the normal flux $\mathbf{F}(\hat{\mathbf{U}})$ and the tangential flux $\mathbf{G}(\hat{\mathbf{U}})$ in the local coordinate system as $\mathbf{A}(\hat{\mathbf{U}})$ and $\mathbf{B}(\hat{\mathbf{U}})$, respectively. These matrices are given by
$$\mathbf{A}(\hat{\mathbf{U}}) = \frac{\partial\mathbf{F}(\hat{\mathbf{U}})}{\partial\hat{\mathbf{U}}} =  \begin{pmatrix} \mathbf{a}_1(\hat{\mathbf{U}}) \\ \mathbf{a}_2(\hat{\mathbf{U}}) \\ \mathbf{a}_3(\hat{\mathbf{U}}) \\ \mathbf{a}_4(\hat{\mathbf{U}}) \end{pmatrix}, \quad  \mathbf{B}(\hat{\mathbf{U}}) = \frac{\partial\mathbf{G}(\hat{\mathbf{U}})}{\partial\hat{\mathbf{U}}} =  \begin{pmatrix} \mathbf{b}_1(\hat{\mathbf{U}}) \\ \mathbf{b}_2(\hat{\mathbf{U}}) \\ \mathbf{b}_3(\hat{\mathbf{U}}) \\ \mathbf{b}_4(\hat{\mathbf{U}}) \end{pmatrix},$$
where $\mathbf{A}(\hat{\mathbf{U}})$ can be diagonalized as
$$\mathbf{A}(\hat{\mathbf{U}}) = \mathbf{L}^{-1} \boldsymbol{\Lambda} \mathbf{L},\quad \boldsymbol{\Lambda}(\hat{\mathbf{U}}) = \text{diag}\{\hat{u} - c, \hat{u}, \hat{u}, \hat{u} + c\},$$
and
$$\mathbf{L}(\hat{\mathbf{U}})=\begin{pmatrix}
\mathbf{l}_1(\hat{\mathbf{U}}) \\[2pt]
\mathbf{l}_2(\hat{\mathbf{U}}) \\[2pt]
\mathbf{l}_3(\hat{\mathbf{U}}) \\[2pt]
\mathbf{l}_4(\hat{\mathbf{U}})
\end{pmatrix}= \begin{pmatrix}
l_{1,1}(\hat{\mathbf{U}}) & l_{1,2}(\hat{\mathbf{U}}) & l_{1,3}(\hat{\mathbf{U}}) & l_{1,4}(\hat{\mathbf{U}}) \\[2pt]
l_{2,1}(\hat{\mathbf{U}}) & l_{2,2}(\hat{\mathbf{U}}) & l_{2,3}(\hat{\mathbf{U}}) & l_{2,4}(\hat{\mathbf{U}}) \\[2pt]
l_{3,1}(\hat{\mathbf{U}}) & l_{3,2}(\hat{\mathbf{U}}) & l_{3,3}(\hat{\mathbf{U}}) & l_{3,4}(\hat{\mathbf{U}}) \\[2pt]
l_{4,1}(\hat{\mathbf{U}}) & l_{4,2}(\hat{\mathbf{U}}) & l_{4,3}(\hat{\mathbf{U}}) & l_{4,4}(\hat{\mathbf{U}})
\end{pmatrix}.$$
Using the left-eigenvector matrix $\mathbf{L}$, the characteristic variables are defined as $\hat{\mathbf{V}} = \mathbf{L} \hat{\mathbf{U}} = (\hat{V}_1, \hat{V}_2, \hat{V}_3, \hat{V}_4)^{\top}$. The number of prescribed boundary conditions at $P_B$ depends on the signs of the eigenvalues $\{\hat{u}-c, \hat{u}, \hat{u}, \hat{u}+c\}$, leading to the following four flow states:
\begin{description}
    \item[Case 1.] Supersonic outflow ($\hat{u} - c \ge 0$): No boundary condition is required.
    \item[Case 2.] Subsonic outflow ($\hat{u} - c < 0 \le \hat{u}$): One boundary condition is required.
    \item[Case 3.] Subsonic inflow ($\hat{u} < 0 \le \hat{u} + c$): Three boundary conditions are required.
    \item[Case 4.] Supersonic inflow ($\hat{u} + c < 0$): Four boundary conditions are required.
\end{description}

We still employ the superscript $*$ to present values located at $P_B$ to ease the notation. To illustrate the algorithmic procedure, we consider \textbf{Case 3} as an example, where three Dirichlet boundary conditions are prescribed at $P_B$ as follows:
$$\hat{U}_1^* = g_1(t), \quad \hat{U}_2^* = g_2(t), \quad \hat{U}_3^* = g_3(t), \quad t > 0.$$

\subsection{ WILW procedure}
At the boundary point $P_B$, $\hat{V}_4$ is the outflow component, while $\hat{V}_1$, $\hat{V}_2$ and $\hat{V}_3$ are inflow components. Following the algorithm in Section \ref{2D Outflow}, we extrapolate $\hat{V}_4$ to obtain the $\hat{x}$-direction derivatives $\hat{V}_4^{*(m)}$ for $m=0, 1, \dots, k$. As discussed above, we reconstruct the WILW polynomials along the $\hat{x}$ direction. Consequently, at time level $t_n$, the exterior value at $P_G$ is approximated via a $(k+1)$-th order Taylor expansion:
\begin{equation}\label{2D Taylor expansion}
    \hat{U}_i^{\text{ext}}(P_G)=p_i(P_G)+\mathcal{O}(\hat{h}^{k+2}),\quad i=1,2,3,4.
\end{equation}

To facilitate the local characteristic decomposition, the unknown state at $P_B$ is approximated by the interior state $\hat{\mathbf{U}}_G = \hat{\mathbf{U}}(P_G)$. This implies that the characteristic decomposition is performed at $P_G$, which constitutes a necessary approximation in our procedure. Combining the boundary conditions with the extrapolated $\hat{V}_4^*$ yields the linear system for $\hat{\mathbf{U}}^*$.
\begin{equation}
\begin{cases} \hat{U}_1^* = g_1(t), \\[2pt] \hat{U}_2^* = g_2(t), \\[2pt] \hat{U}_3^* = g_3(t), \\[2pt] \mathbf{l}_4(\hat{\mathbf{U}}_G)\cdot\hat{\mathbf{U}}^* = \hat{V}_4^*. \end{cases}
\end{equation}

Next, the inverse Lax-Wendroff procedure to $\hat{U}_1, \hat{U}_2, \hat{U}_3$ and the extrapolated $\hat{V}_4^{*(1)}$ yields the following linear system for the $\hat{x}$-derivatives $\mathbf{\hat{U}}^{*(1)}$:
\begin{equation}
    \mathcal{A}_1\hat{\mathbf{U}}^{*(1)}=\mathbf{b}_1,
\end{equation}
where the coefficient matrix $\mathcal{A}_1$ is
$$\mathcal{A} = \begin{pmatrix}  \mathbf{a}_1(\hat{\mathbf{U}}^*) \\[2pt]  \mathbf{a}_2(\hat{\mathbf{U}}^*) \\[2pt]  \mathbf{a}_3(\hat{\mathbf{U}}^*) \\[2pt]  \mathbf{l}_4(\hat{\mathbf{U}}_G)  \end{pmatrix} = \begin{pmatrix} 0 & 1 & 0 & 0 \\[6pt] \frac{\gamma - 3}{2} \left( \frac{\hat{U}_2^*}{\hat{U}_1^*} \right)^2 + \frac{\gamma - 1}{2} \left( \frac{\hat{U}_3^*}{\hat{U}_1^*} \right)^2 & (3 - \gamma) \frac{\hat{U}_2^*}{\hat{U}_1^*} & (1 - \gamma) \frac{\hat{U}_3^*}{\hat{U}_1^*} & \gamma - 1 \\[6pt] -\frac{\hat{U}_2^* \hat{U}_3^*}{({\hat{U}_1^*})^2} & \frac{\hat{U}_3^*}{\hat{U}_1^*} & \frac{\hat{U}_2^*}{\hat{U}_1^*} & 0 \\[6pt] l_{4,1}(\hat{\mathbf{U}}_G) & l_{4,2}(\hat{\mathbf{U}}_G) & l_{4,3}(\hat{\mathbf{U}}_G) & l_{4,4}(\hat{\mathbf{U}}_G) \end{pmatrix},$$
and the vectors are
$$\hat{\mathbf{U}}^{*(1)} = \begin{pmatrix}
\hat{U}_1^{*(1)} \\[2pt]
\hat{U}_2^{*(1)} \\[2pt]
\hat{U}_3^{*(1)} \\[2pt]
\hat{U}_4^{*(1)}
\end{pmatrix},\quad\mathbf{b} = \begin{pmatrix}  -g_1'(t) - ( \mathbf{B}(\hat{\mathbf{U}}^{*}) \cdot\hat{\mathbf{U}}_{\hat{y}}^{*} )_1 \\[2pt]  
-g_2'(t) - ( \mathbf{B}(\hat{\mathbf{U}}^{*}) \cdot\hat{\mathbf{U}}_{\hat{y}}^{*} )_2 \\[2pt]  
-g_3'(t) - ( \mathbf{B}(\hat{\mathbf{U}}^{*}) \cdot\hat{\mathbf{U}}_{\hat{y}}^{*} )_3 \\[2pt]  
\hat{V}_4^{*(1)}  \end{pmatrix}.$$

The tangential derivatives in $\mathbf{b}$ are generally unknown. Following the procedure in Section \ref{2D Outflow}, we reconstruct the polynomials and extrapolate them to the boundary to compute the tangential derivatives at $P_B$. We assume that all eigenvalues are well-separated from zero, ensuring that the coefficient matrix $\mathcal{A}$ is well-conditioned. If an eigenvalue approaches zero, an additional characteristic extrapolation equation is incorporated to address the issue, following the same procedure as in the one-dimensional case.

Similarly, by repeatedly applying the Euler equations, we construct a linear system for the higher-order normal derivatives $\hat{\mathbf{U}}^{*(m)}$. The right-hand side is determined by the $m$-th order derivatives of $g_1, g_2, g_3$, the tangential derivatives, the extrapolated normal derivatives of $\hat{V}_4$, and the lower-order normal derivatives. The coefficient matrix depends solely on $\hat{\mathbf{U}}^*$.

Finally, the exterior state $\hat{\mathbf{U}}_G^{\text{ext}}$ at $P_G$ is determined via Taylor expansion and transformed to $\mathbf{U}_G^{\text{ext}}$ by rotating the momentum components.

The complete algorithm for 2D Euler equations is summarized as follows.
\vspace{1em}
\begingroup
\captionsetup{hypcap=false}
\phantomsection
\captionof{algorithm}{WILW method for 2D Euler equations}
\label{2D algorithm}
\begin{enumerate}[label=\textbf{Step \arabic*.}, leftmargin=1.5cm]
    \item For each Gauss point $P_G$ on the computational boundary, identify its normal projection $P_B$ on the physical boundary $\partial\Omega$. A local coordinate system is then established at $P_B$ \eqref{coordinate rotation}.
    \item Compute the eigenvalues $\lambda(\hat{\mathbf{U}}_G)$ and the left eigenvector matrix $\mathbf{L}(\hat{\mathbf{U}}_G)$ of the Jacobian matrix $\mathbf{A}(\hat{\mathbf{U}}_G)$. Based on the signs of these eigenvalues, prescribe the inflow boundary conditions $g_1(t), g_2(t)$, and $g_3(t)$.
    \item Following Section \ref{2D Outflow}, we reconstruct and extrapolate polynomials to determine $V_4^{*(m)}$ ($m=0,1, \dots, k$). This procedure also yields the tangential derivatives of the conserved variables at $P_B$.
    \item Determine the boundary state $\hat{\mathbf{U}}^*$ by solving the linear system coupling the prescribed boundary conditions and the extrapolated $V_4^*$.
    \item Apply the inverse Lax-Wendroff procedure to construct the linear system and determine the higher-order normal derivatives $\hat{\mathbf{U}}^{*(m)}$.
    \item Construct $(k+1)$-th degree polynomials $p_i(\hat{x}, \hat{y}, t)$ ($i=1, \dots, 4$) via WILW, and evaluate them at $P_G$ according to \eqref{2D Taylor expansion} to yield the exterior state $\hat{\mathbf{U}}_G^{\text{ext}}$.
    \item Finally, transform the exterior state $\hat{\mathbf{U}}_G^{\text{ext}}$ from the local coordinate system back to the global system to obtain $\mathbf{U}_G^{\text{ext}}$. 
\end{enumerate}
\endgroup

\subsection{Two-dimensional extrapolation}\label{2D Outflow}
In this section, we propose a reconstruction-based extrapolation method to determine boundary outflow information. This approach utilizes information from the boundary cell containing $P_G$ and its surrounding neighbors to construct a bivariate polynomial $p(x, y, t) \in P^{k+1}$. The core idea is to expand the reconstruction stencil $\mathcal{E}$ with neighboring cells, and determine $p(x, y, t)$ through a weighted least-squares procedure, ensuring numerical stability during extrapolation, especially for subsonic outflow boundaries. Let $(x_0, y_0)$ denote the center of the boundary cell containing $P_G$. The reconstruction stencil $\mathcal{E}$ is defined as follows:
$$\mathcal{E} = \{ I_n \in \tilde{\Omega} : \sqrt{(x_n - x_0)^2 + (y_n - y_0)^2} \le 3 {\color{red} h } \},$$
where $(x_n, y_n)$ represents the center of cell $I_n$, as illustrated in Fig. \ref{Local coordinate system}.

The polynomial $p(x, y, t) \in {P}^{k+1}$ is determined via weighted least squares using stencil $\mathcal{E}$. The overdetermined system is constructed as follows:
\begin{equation}
    \begin{cases} \iint_{I_n} p_i(x,y,t) \, dxdy = \iint_{I_n} U_i(x,y,t) \, dxdy, \\[2pt]
    \iint_{I_n} \partial_xp_i(x,y,t) \, dxdy = \iint_{I_n} \partial_xU_i(x,y,t) \, dxdy, \\[2pt]
    \iint_{I_n} \partial_yp_i(x,y,t) \, dxdy = \iint_{I_n} \partial_yU_i(x,y,t) \, dxdy, \end{cases}\quad \forall I_n\in\mathcal{E},~i=1,2,3,4.
\end{equation}
We empirically assign weights of $\mathcal{O}(1)$ to the equations involving cell averages and $\mathcal{O}(h)$ to those involving first-order derivatives. In contrast to the one-dimensional reconstruction, the current stencil provides a sufficient number of cells for the least-squares procedure, eliminating the need for second-order derivatives.

In practice, the selection of the $3h$ stencil radius is not strictly fixed. The $3h$ radius is primarily necessary for subsonic outflow boundary conditions, while a compact stencil with a radius of $h$ is often sufficient for other cases. Additionally, all Gauss points on a boundary element share the same stencil $\mathcal{E}$, allowing the reconstruction to be performed only once per element and thereby reducing the computational cost.

\section{Numerical examples}
In this section, we present some numerical examples to demonstrate the performance of our proposed scheme. In all computations, the DG method with $V_h^k$ is coupled with a WILW boundary treatment based on a $(k+1)$-th degree reconstruction polynomial. Unless otherwise stated, the third-order Runge-Kutta method (\ref{RK}) is employed for time discretization. For one-dimensional problems, the time step $\Delta t$ is chosen as:
\begin{equation}\label{1Ddt}
\Delta t = \frac{\text{CFL} \cdot \Delta x^\nu}{\alpha_x},
\end{equation}
and for two-dimensional cases:
\begin{equation}\label{2Ddt}
\Delta t =  \frac{\text{CFL}}{\alpha_x / \Delta x^\nu+ \alpha_y / \Delta y^\nu}, 
\end{equation}
where $\nu=\max(\frac{k+1}{3},1)$, $\alpha_x = \max |\lambda(\mathbf{F}^\prime(\mathbf{U}))|$ and $\alpha_y = \max |\lambda(\mathbf{G}^\prime(\mathbf{U}))|$, representing the global maximum eigenvalues in the $x$- and $y$-direction, respectively.
The CFL number $\text{CFL} = \frac{1}{2k+1}$ is the same as that used in the standard RKDG method. 

For problems involving discontinuities, we adopt the numerical damping term from the Oscillation-Free DG (OFDG) scheme proposed in \cite{ofdg2021, ofdg2022} to suppress spurious oscillations. 
In each internal cell, the artificial damping coefficients depend on the jump of the numerical solution and its derivatives at the cell interfaces. For example, for the one-dimensional case \eqref{1D Scalar Equation}, those damping coefficients in $I_j$ are defined as
$$\sigma_j^l = \frac{2 ( 2l + 1 )}{( 2k - 1 )} \frac{h^l}{l!} \left( [\![ \partial_x^l u_h ]\!]_{j+1/2}^2 + [\![ \partial_x^l u_h ]\!]_{j-1/2}^2 \right)^{1/2}, \quad l=0,1,\dots,k,$$
where the jumps are given by  
$$[\![ \partial_x^l u_h ]\!]_{j+1/2} = \partial_x^l u_h(x_{j+1/2}^+) - \partial_x^l u_h(x_{j+1/2}^-).$$
In particular, for the cell interfaces $x_{1/2}$ and $x_{N+1/2}$, the reconstruction polynomials $p(x)$ on boundary cells $\tilde{I}_0$ and $\tilde{I}_N$ are utilized to evaluate the damping terms.
For the left boundary, where $p \in P^{k+1}(\tilde{I}_0)$, the jumps at $x_{1/2}$ are defined as
$$[\![ \partial_x^l u_h ]\!]_{1/2} = \partial_x^l u_h(x_{1/2}^+) - \partial_x^l p(x_{1/2}^-), \quad l=0,1, \dots, k.$$
The jumps at the right boundary point $x_{N+1/2}$ are determined analogously.
Meanwhile, the time step $\Delta t$ is additionally restricted by the damping coefficient $\sigma_{\rm max} = \max_j \sum_{l=0}^k \sigma_j^l$ according to the following relation:
\begin{equation}
\Delta t = \frac{\text{CFL} \cdot \Delta x^\nu}{\alpha_x+\sigma_{\rm max}}. 
\end{equation}
The two-dimensional cases follow a similar formulation.

\subsection{One-dimensional examples}
\begin{example}  \label{ex1}
(\textbf{Linear advection equation.}) 
\end{example}

First, we consider the initial-boundary value problem of the linear advection equation
\begin{equation}\label{Example1}
\begin{cases}
u_t + u_x = 0, & 0 < x < 2\pi, t > 0, \\
u(x, 0) = -\sin(x), & 0 \le x \le 2\pi, \\
u(0, t) = g(t), & t\geq 0.
\end{cases}
\end{equation}
The left boundary $x = 0$ is an inflow boundary all the time and the right boundary $x = 2\pi$ is an outflow boundary. Following the uniform grid partition in \eqref{1Dmesh}, we consider two cases for the boundary offset, $\delta_1 = 0.01h$ and $0.99h$, while keeping $\delta_2 = 0$. 

We first set $g(t)=\sin(t)$ so that the initial-boundary value problem admits a smooth exact solution. Numerical errors are calculated on the computational domain $\tilde{\Omega}$. The $L^2$ errors and $L^\infty$ errors at time $t = 3$ are shown in Table \ref{Example1 Table}. We can see that for all cases, the schemes are stable and achieve the expected $(k + 1)$-th order of accuracy. Moreover, under the proposed scheme, the errors for different $\delta$ show no significant difference.

\begin{table}[htbp!]
\centering
\caption{Example \ref{ex1}: Numerical errors and orders of the linear  advection equation at $t = 3$.}
\label{Example1 Table}
\begin{tabular}{ccccccccc}
\toprule
\multirow{2}{*}{$N$} & \multicolumn{3}{l}{$\delta = 0.01h$} & &\multicolumn{3}{l}{$\delta = 0.99h$} & \\
\cmidrule(lr){2-5} \cmidrule(lr){6-9}
 & $L^2$ error & Order & $L^\infty$ error & Order & $L^2$ error & Order & $L^\infty$ error & Order \\
\midrule
\multicolumn{9}{l}{$P^1$ WILW} \\
\midrule
40 & 2.67E-03 & -- & 3.28E-03 & --  & 4.06E-03 & --  & 3.95E-03 & --  \\
80 & 6.66E-04 & 2.00 & 8.31E-04 & 1.98 & 7.77E-04 & 2.39 & 8.71E-04 & 2.18 \\
160 & 1.66E-04 & 2.00 & 2.09E-04 & 1.99 & 1.73E-04 & 2.16 & 2.10E-04 & 2.05 \\
320 & 4.16E-05 & 2.00 & 5.25E-05 & 2.00 & 4.19E-05 & 2.05 & 5.24E-05 & 2.01 \\
\midrule
\multicolumn{9}{l}{$P^2$ WILW} \\
\midrule
40 & 3.39E-05 & -- & 3.99E-05 & -- & 3.33E-04 & -- & 2.87E-04 & -- \\
80 & 4.24E-06 & 3.00 & 5.03E-06 & 2.99 & 1.88E-05 & 4.15 & 1.71E-05 & 4.07 \\
160 & 5.31E-07 & 3.00 & 6.27E-07 & 3.01 & 1.21E-06 & 3.95 & 1.20E-06 & 3.83 \\
320 & 6.63E-08 & 3.00  & 7.84E-08 & 3.00 & 9.37E-08 & 3.69 & 1.06E-07 & 3.49 \\
\midrule
\multicolumn{9}{l}{$P^3$ WILW} \\
\midrule
40 & 4.49E-07 & -- & 4.91E-07 & -- & 3.72E-06 & -- & 3.32E-06 & -- \\
80 & 2.81E-08 & 4.00 & 3.07E-08 & 4.00 & 1.31E-07 & 4.83 & 1.12E-07 & 4.89 \\
160 & 1.76E-09 & 4.00 & 1.92E-09 & 4.00 & 4.61E-09 & 4.82 & 3.91E-09 & 4.84 \\
320 & 1.10E-10 & 4.00 & 1.20E-10 & 4.00 & 1.81E-10 & 4.67 & 1.59E-10 & 4.62 \\
\bottomrule
\end{tabular}
\end{table}

Next, we define the boundary condition $g(t)$ as 
\begin{equation}\label{Example1_BC2}
g(t) = \begin{cases}
0, & t \leq \pi, \\
1, & t > \pi.
\end{cases}
\end{equation}
Discontinuities enter the domain through the inflow boundary and are well-captured by employing the OFDG scheme in conjunction with our proposed boundary treatment. For brevity we only present the results for $k=2$, as shown in Fig. \ref{fig:Example1}.

\begin{figure}[htbp!]
    \centering
    
    \begin{subfigure}{0.45\textwidth}
        \centering
        \includegraphics[width=\textwidth]{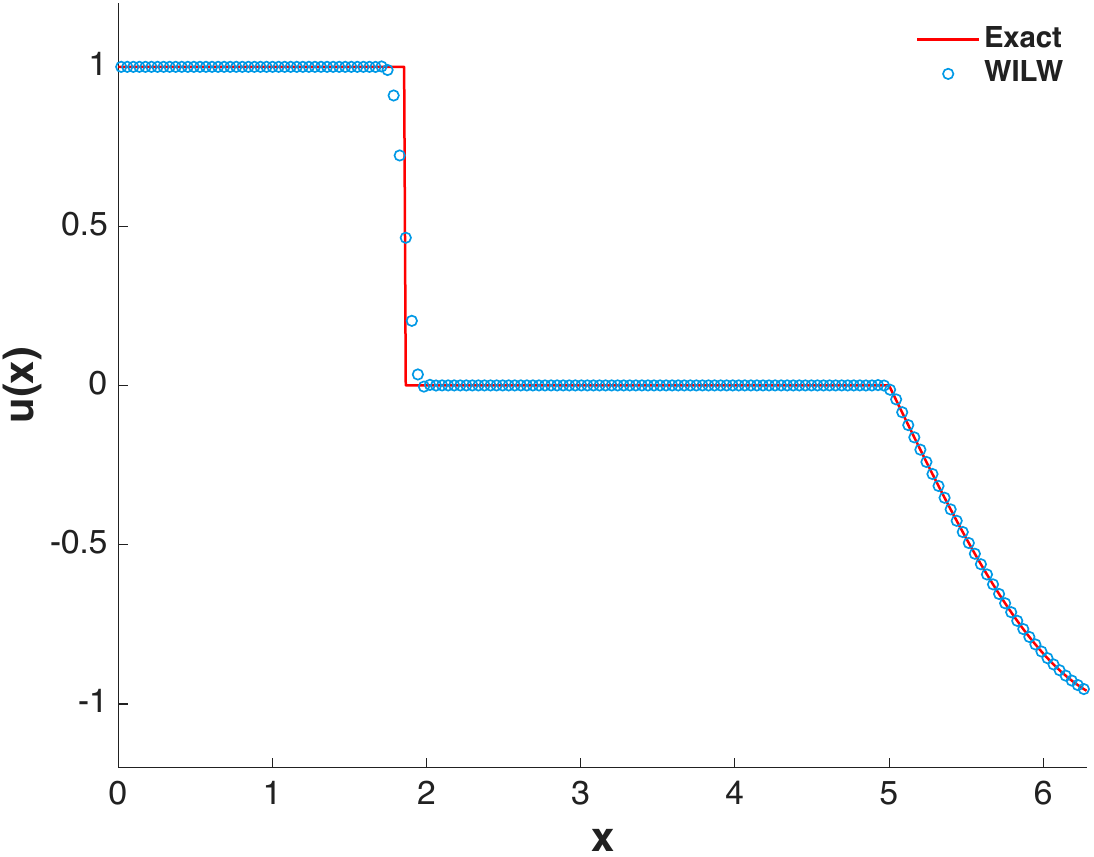}
        \caption{$\delta=0.01h$, WILW}
    \end{subfigure}
    \hspace{0.2cm}
    \begin{subfigure}{0.45\textwidth}
        \centering
        \includegraphics[width=\textwidth]{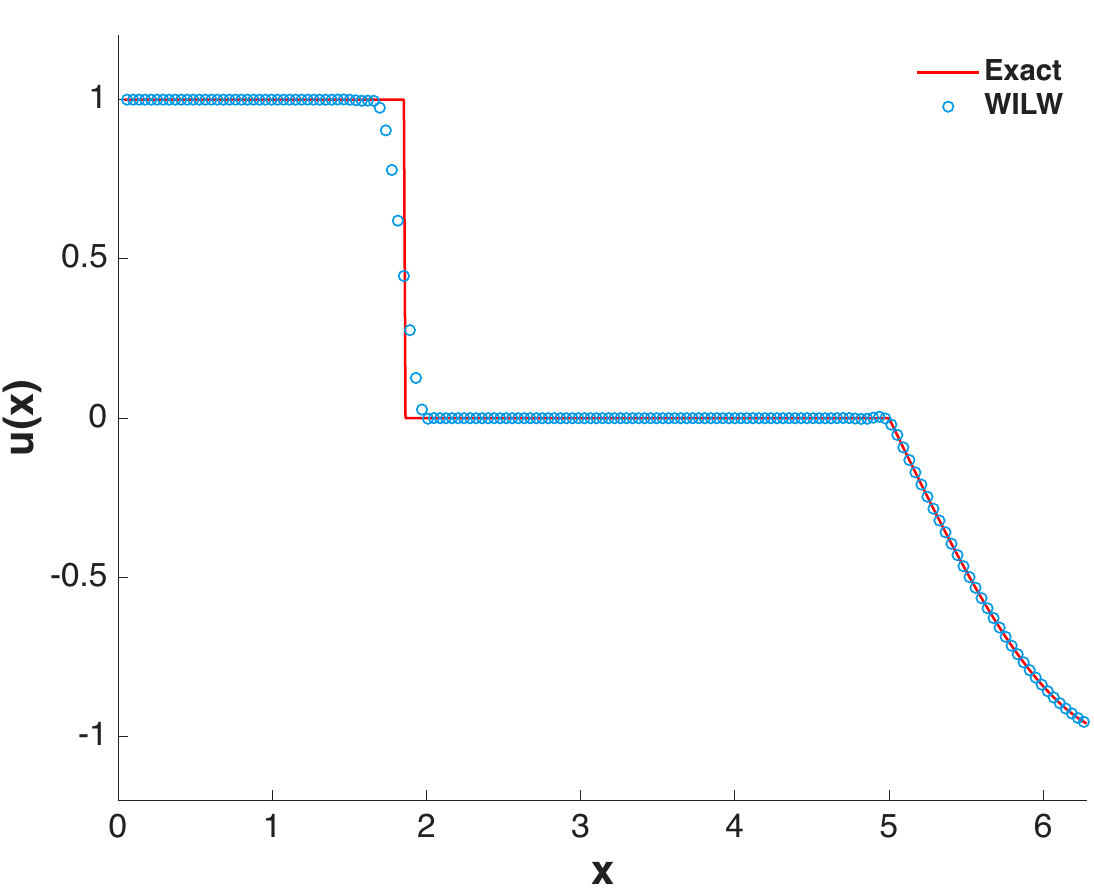}
        \caption{$\delta=0.99h$, WILW}
    \end{subfigure}
    \caption{Example \ref{ex1}: Linear advection equation with boundary condition (\ref{Example1_BC2}). $N=160$, $t=5$, $k=2$. Solid line: exact solution; symbols: numerical solution.}
    \label{fig:Example1}
\end{figure}
\FloatBarrier


\begin{example}\label{ex2}
    \textbf{(1D Burgers' equation.)} 
\end{example}

For nonlinear problems, we consider the Burgers' equation:
\begin{equation}\label{Numerical2}
\begin{cases}
u_t + (\frac{u^2}{2})_x = 0, & -\pi < x < \pi, t > 0, \\
u(x, 0) = 1 + 2\sin(x), & -\pi \le x \le \pi, \\
u(-\pi, t) = g(t), & t\geq 0.
\end{cases}
\end{equation}
We take the boundary condition $g(t) = w(-\pi, t)$, where $w(x, t)$ is the exact solution of the initial value problem on $(- \pi, \pi)$ with periodic boundary conditions and can be obtained by Newton’s method. The mesh configuration remains identical to that in \textbf{Example \ref{ex1}}. 

At $t = 0.3$, the exact solution is still smooth. Errors are shown in Table \ref{Example2 Table}. It is observed that the scheme can always achieve the expected $(k + 1)$-th order accuracy and the ratio $\delta/h$ does not affect the magnitude of errors significantly.

At $t = 2\pi$, a shock enters the inflow boundary and moves to $x = 0$ at $t = 3\pi$. We employ the OFDG scheme to suppress spurious oscillations. For brevity, we only present the results for $k=2$. As illustrated in Fig. \ref{fig:Example2.2}, the shock waves are well-captured.

\begin{table}[htbp!]
\centering
\caption{Example \ref{ex2}: Numerical errors and orders of Burgers’ equation at $t = 0.3$.}
\label{Example2 Table}
\begin{tabular}{ccccccccc}
\toprule
\multirow{2}{*}{$N$} & \multicolumn{3}{l}{$\delta = 0.01h$} & &\multicolumn{3}{l}{$\delta = 0.99h$} & \\
\cmidrule(lr){2-5} \cmidrule(lr){6-9}
 & $L^2$ error & Order & $L^\infty$ error & Order & $L^2$ error & Order & $L^\infty$ error & Order \\
\midrule
\multicolumn{9}{l}{$P^1$ WILW} \\
\midrule
40  & 2.24E-02 & --   & 6.08E-02 & --   & 3.00E-02 & --   & 7.80E-02 & --   \\
80  & 5.29E-03 & 2.08 & 1.75E-02 & 1.80 & 6.45E-03 & 2.22 & 1.75E-02 & 2.16 \\
160 & 1.18E-03 & 2.16 & 4.92E-03 & 1.83 & 1.30E-03 & 2.31 & 4.75E-03 & 1.88 \\
320 & 2.53E-04 & 2.22 & 1.21E-03 & 2.03 & 2.67E-04 & 2.28 & 1.25E-03 & 1.93 \\
\midrule
\multicolumn{9}{l}{$P^2$ WILW} \\
\midrule
40  & 1.41E-03 & --   & 3.62E-03 & --   & 1.56E-03 & --   & 5.14E-03 & --   \\
80  & 1.13E-04 & 3.64 & 3.83E-04 & 3.24 & 3.13E-04 & 2.31 & 1.84E-03 & 1.48 \\
160 & 1.19E-05 & 3.25 & 4.92E-05 & 2.96 & 3.89E-05 & 3.01 & 1.66E-04 & 3.47 \\
320 & 1.37E-06 & 3.11 & 6.10E-06 & 3.01 & 2.79E-06 & 3.80 & 9.28E-06 & 4.16 \\
\midrule
\multicolumn{9}{l}{$P^3$ WILW} \\
\midrule
40  & 2.51E-04 & --   & 7.05E-04 & --   & 1.34E-04 & --   & 3.98E-04 & --   \\
80  & 8.06E-06 & 4.96 & 3.24E-05 & 4.44 & 6.25E-05 & 1.10 & 2.88E-04 & 0.47 \\
160 & 3.35E-07 & 4.59 & 1.91E-06 & 4.09 & 9.57E-07 & 6.03 & 5.93E-06 & 5.60 \\
320 & 1.63E-08 & 4.36 & 9.89E-08 & 4.27 & 4.90E-08 & 4.29 & 2.55E-07 & 4.54 \\
\bottomrule
\end{tabular}
\end{table}

\begin{figure}[htbp!]
    \centering

    \begin{subfigure}{0.45\textwidth}
        \centering
        \includegraphics[width=\textwidth]{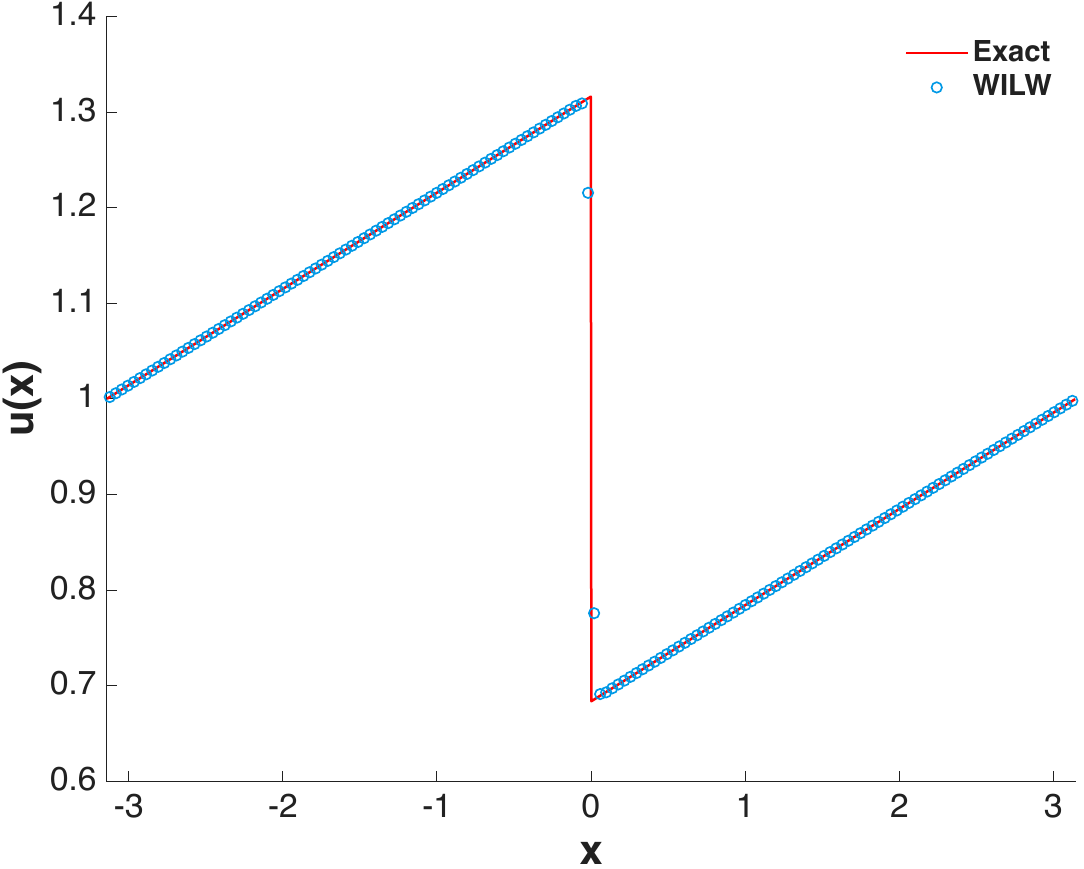}
        \caption{$\delta=0.01h$, WILW}
    \end{subfigure}
    \hspace{0.2cm} 
    \begin{subfigure}{0.45\textwidth}
        \centering
        \includegraphics[width=\textwidth]{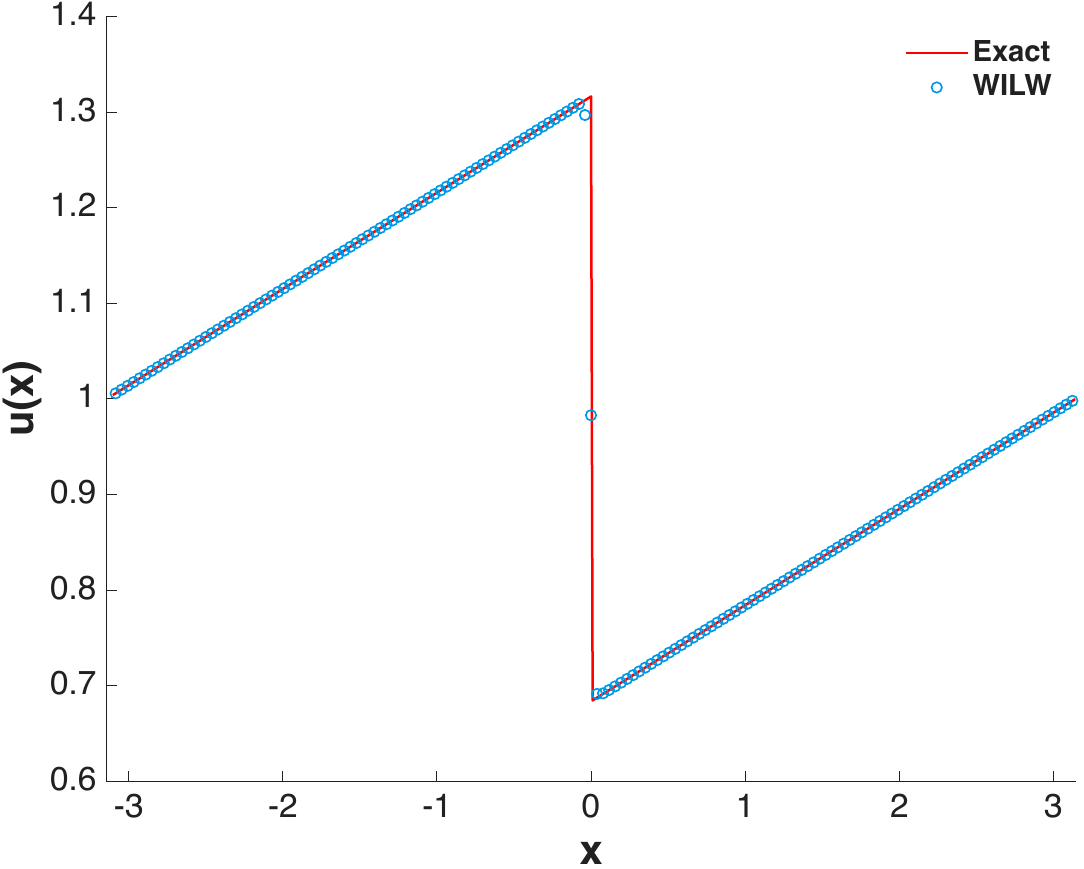}
        \caption{$\delta=0.99h$, WILW}
    \end{subfigure}

    \caption{Example \ref{ex2}: Numerical solutions for Burgers’ equation at $t=3\pi, k=2, N=160$. Solid line: exact solution; symbols: numerical solution.}
    \label{fig:Example2.2}
\end{figure}
\FloatBarrier


\begin{example} 
\textbf{(1D Euler equations.)} \label{ex:1D_Euler}
\end{example}

For systems, we consider 1D compressible Euler equations (\ref{Euler_equations}), the domain is $(0,2)$. The corresponding initial values and boundary conditions are chosen to match the exact solution
$$\begin{cases}
\rho(x,t)=1+0.2\sin\pi (x-t),\\
u(x,t)=1,\\
p(x,t)=2.
\end{cases}$$
In particular, at both boundaries, we have $u-c<0<u<u+c$. Therefore, we need to impose two boundary conditions at $x = 0$ and one boundary condition at $x = 2$. Here, we take the boundary conditions
$$\begin{cases}\rho(0,t)=1-0.2\sin\pi t,\\u(0,t)=1,\end{cases}
\quad \rho(2,t)=1-0.2\sin\pi t.$$
Following the uniform grid partition in \eqref{1Dmesh}, we consider two cases of boundary offsets by setting $\delta_1 = \delta_2 = 0.01h$ and $0.99h$. The numerical errors for density are presented in Table \ref{Example3 Table}, confirming that our method maintains the expected convergence rate and that the ratio $\delta/h$ does not significantly affect the error magnitude.

\begin{table}[htbp]
\centering
\caption{Example \ref{ex:1D_Euler}: Density errors and orders of Euler equations at $t = 2$.}
\label{Example3 Table}
\begin{tabular}{ccccccccc}
\toprule
\multirow{2}{*}{$N$} & \multicolumn{3}{l}{$\delta = 0.01h$} & &\multicolumn{3}{l}{$\delta = 0.99h$} & \\
\cmidrule(lr){2-5} \cmidrule(lr){6-9}
 & $L^2$ error & Order & $L^\infty$ error & Order & $L^2$ error & Order & $L^\infty$ error & Order \\
\midrule
\multicolumn{9}{l}{$P^1$ WILW} \\
\midrule
40 & 3.46E-04 & -- & 5.00E-04 & --   &   1.46E-03 & --  & 1.69E-03   & --  \\
80 & 8.51E-05 & 2.02 & 1.23E-04 & 2.03 & 3.54E-04 & 2.04 & 4.14E-04 & 2.03 \\
160 & 2.11E-05 & 2.01 & 3.04E-05 & 2.01 & 8.82E-05 & 2.01 & 1.03E-04 & 2.00 \\
320 & 5.26E-06 & 2.00 & 7.56E-06 & 2.00 & 2.21E-05 & 2.00 & 2.59E-05 & 2.00 \\
\midrule
\multicolumn{9}{l}{$P^2$ WILW} \\
\midrule
40 & 1.34E-05 & -- & 1.94E-05 & -- & 7.40E-05 & -- & 9.77E-05 & -- \\
80 & 1.73E-06 & 2.96 & 2.49E-06 & 2.96 & 1.05E-05 & 2.82 & 1.34E-05 & 2.87 \\
160 & 2.18E-07 & 2.99 & 3.14E-07 & 2.99 & 1.37E-06 & 2.94 & 1.65E-06 & 3.02 \\
320 & 2.74E-08 & 2.99 & 3.93E-08 & 3.00 & 1.74E-07 & 2.97 & 2.05E-07 & 3.01 \\
\midrule
\multicolumn{9}{l}{$P^3$ WILW} \\
\midrule
40 & 5.20E-08 & -- & 8.16E-08 & -- & 4.54E-07 & -- & 4.81E-07 & -- \\
80 & 3.24E-09 & 4.00 & 5.10E-09 & 4.00 & 1.83E-08 & 4.63 & 2.06E-08 & 4.55 \\
160 & 2.02E-10 & 4.00 & 3.19E-10 & 4.00 & 8.72E-10 & 4.39 & 1.07E-09 & 4.26 \\
320 & 1.25E-11 & 4.01 & 2.01E-11 & 3.99 & 4.90E-11 & 4.15 & 6.37E-11 & 4.07 \\
\bottomrule
\end{tabular}
\end{table}

\begin{example}
  \textbf{(Blast wave problem.)} \label{ex:1D_Euler_shock}  
\end{example}

We then test our method for the Euler equations with shocks. We consider the interaction of two blast waves \cite{twoblast1984}. The initial data are
$$(\rho, u, p) = 
\begin{cases} 
(1, 0, 1000), & x \in [0, 0.1), \\[6pt]
(1, 0, 0.01), & x \in [0.1, 0.9), \\[6pt]
(1, 0, 100),  & x \in [0.9, 1.0].
\end{cases}$$
There are solid wall boundary conditions at both $x = 0$ and $x = 1$. This problem involves multiple reflections of shocks and rarefactions off the walls. There are also multiple interactions of shocks and rarefactions with each other and with contact discontinuities. The reference solution is computed using the fifth order WENO scheme on body-fitted uniform mesh with $\Delta x = 1/16000$. The mesh configuration remains identical to that in \textbf{Example \ref{ex:1D_Euler}}. We employ the OFDG scheme to suppress spurious oscillations. For brevity, we only present the results for $k=2$ and $N=640$ at $t=0.038$. Numerical results with different boundary treatment methods are provided in Fig. \ref{fig:Example4}. The figures demonstrate that the proposed scheme effectively captures the structural features of the solution. The minor deviations at the peaks and troughs are primarily attributed to the inherent numerical dissipation of the OFDG scheme rather than the boundary treatment.

\begin{figure}[htbp]
    \centering
    
    \begin{subfigure}{0.45\textwidth}
        \centering
        \includegraphics[width=\textwidth]{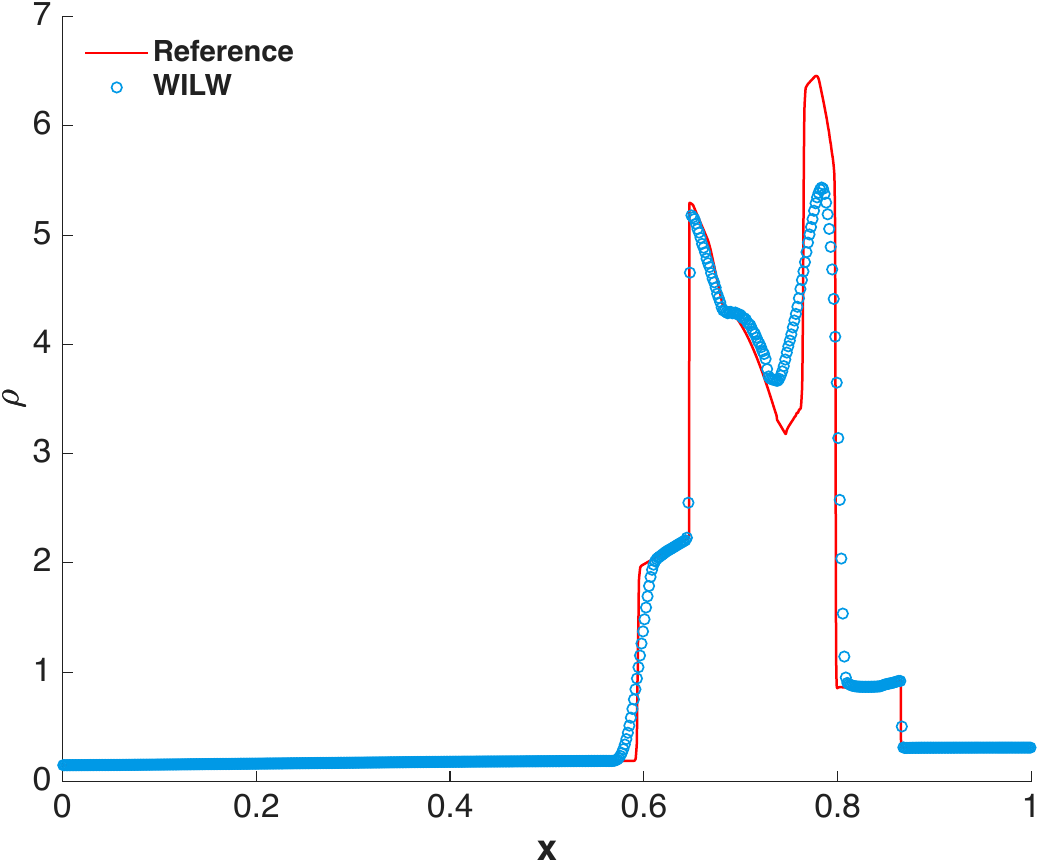}
        \caption{$\delta=0.01h$, WILW}
    \end{subfigure}
    \hspace{0.2cm} 
    \begin{subfigure}{0.45\textwidth}
        \centering
        \includegraphics[width=\textwidth]{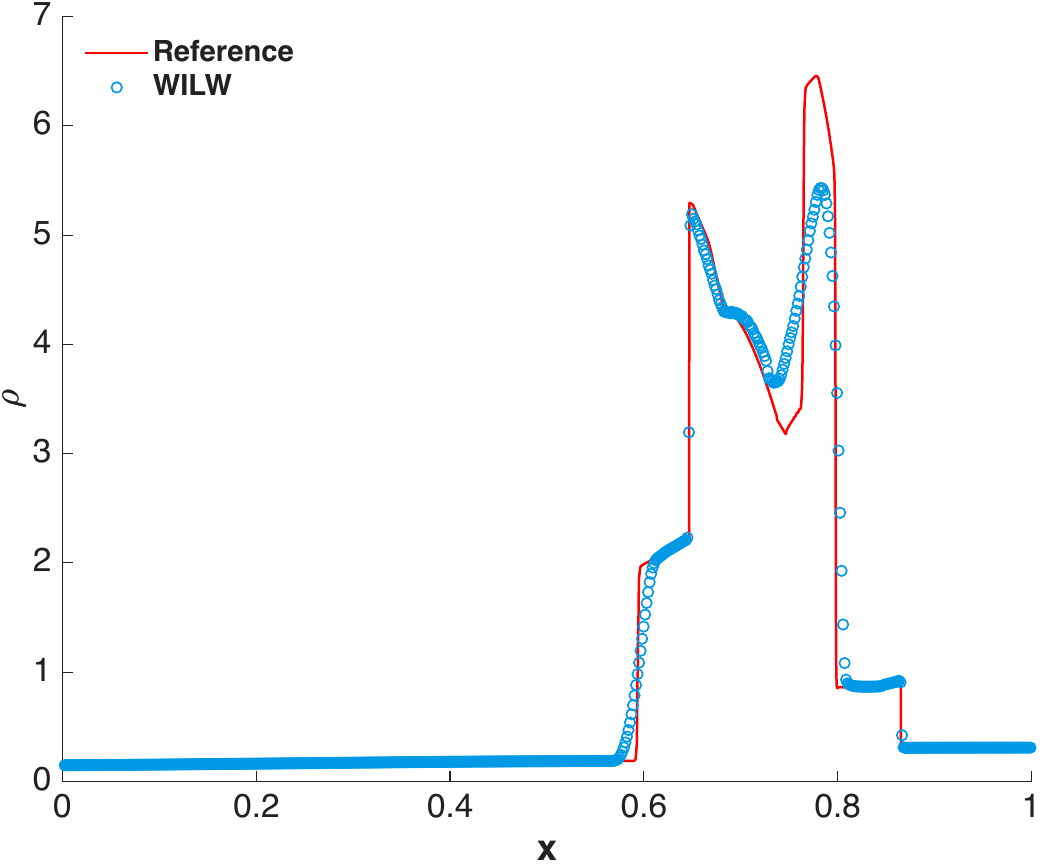}
        \caption{$\delta=0.99h$, WILW}
    \end{subfigure}

    \caption{Example \ref{ex:1D_Euler_shock}: The density profiles of the blast wave problem at $t=0.038$,~$k=2$,~$N=640$. Solid line: reference solution; symbols: numerical solution.}
    \label{fig:Example4}
\end{figure}
\FloatBarrier


\subsection{Two-dimensional examples}

\begin{example}
    \textbf{(2D Burgers' equation.)} \label{eq:2D_burgers}
\end{example}

We start our two-dimensional examples with the 2D Burgers' equation
\begin{equation}\label{2D Burgers equation}
\begin{cases}
u_t + \frac{1}{2}(u^2)_x + \frac{1}{2}(u^2)_y = 0, & (x,y) \in \Omega, \quad t > 0, \\[6pt]
u(x,y,0) = 0.75 + 0.5\sin[\pi(x+y)], & (x,y) \in \bar{\Omega}, \\[6pt]
u(x,y,t) = g(x,y,t), & (x,y) \in \Gamma, \quad t > 0,
\end{cases}
\end{equation}
where
\begin{equation}\label{2D_Wave_Square}
    \Omega = (-1, 1) \times (-1, 1),\quad \Gamma = \{(x, y) : x = -1 \text{ or } y = -1\},
\end{equation}
or
\begin{equation}\label{2D_Wave_Disk}
    \Omega = \{(x,y) : x^2 + y^2 < 0.5\}, \quad \Gamma = \{(x,y) : x^2 + y^2 = 0.5 \text{ and } x + y \le 0\}.
\end{equation}
Here $g(x,y,t) = w(x,y,t)$, where $w(x,y,t)$ is the exact solution of the 2D Burgers' equation with the same initial value and periodic boundary conditions defined on $[-1,1]^2$. 

For the rectangular domain $\Omega$, an unfitted mesh with an offset scale $\delta$ is constructed to test our boundary treatment. For a circular domain $\Omega$, the computational domain is discretized by partitioning its bounding square into a uniform mesh. The effective computational domain $\tilde{\Omega}$ is then formed by selecting all grid cells entirely contained within $\Omega$, as illustrated in Fig. \ref{Square and Disk}. Notably, for circular domains, a Gauss point on the computational inflow boundary may project onto an outflow segment of the physical boundary. In the absence of available inflow conditions, such cases are treated as outflow for the calculation.

\begin{figure}[htbp!]
    \centering 
    \begin{subfigure}{0.30\textwidth}
        \centering
        \includegraphics[width=\textwidth]{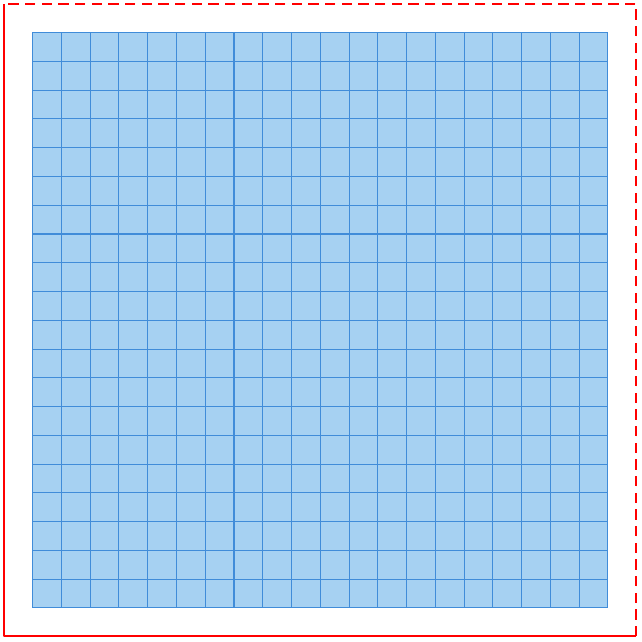}
        \caption{Square domain}
    \end{subfigure}    
    \hspace{1cm}
    \begin{subfigure}{0.30\textwidth}
        \centering
        \includegraphics[width=\textwidth]{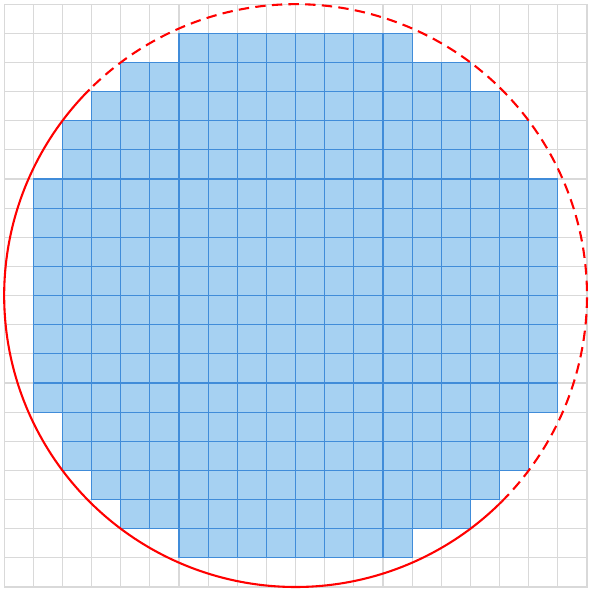}
        \caption{Circular domain}
    \end{subfigure}   
    \caption{Example \ref{eq:2D_burgers}: Domain $\Omega$ and $\tilde{\Omega}$ of the 2D Burgers' equation. Solid lines: inflow boundary; dashed lines: outflow boundary.}
    \label{Square and Disk}
\end{figure}

The solution remains smooth at $t = 0.15$. As shown in Table \ref{Example5 Table1} and Table \ref{Example5 Table2}, the method always attains the designed accuracy.  At $t = 6$, multiple shocks have formed outside and traversed the computational domain, with one specifically moving to $x = 0$. As shown by the diagonal cuts in Fig. \ref{fig:Example5.1} and \ref{fig:Example5.2}, the shock waves in both regions are well-captured.

\begin{table}[htbp!]
\centering
\caption{Example \ref{eq:2D_burgers}: Numerical errors for the 2D Burgers' equation at $t = 0.15$. Square domain (\ref{2D_Wave_Square}) with $\Delta x = 2/(N+2\delta/h)$ and $\Delta y = 2/(N+2\delta/h)$.}
\label{Example5 Table1}
\begin{tabular}{ccccccccc}
\toprule
\multirow{2}{*}{$N$} & \multicolumn{3}{l}{$\delta = 0.01h$} & &\multicolumn{3}{l}{$\delta = 0.99h$} & \\
\cmidrule(lr){2-5} \cmidrule(lr){6-9}
 & $L^2$ error & Order & $L^\infty$ error & Order & $L^2$ error & Order & $L^\infty$ error & Order \\
\midrule
\multicolumn{9}{l}{$P^1$ WILW} \\
\midrule
40  & 4.58E-03 & --   & 2.33E-02 & --   & 4.32E-03 & --   & 2.75E-02 & --   \\
80  & 9.89E-04 & 2.21 & 5.41E-03 & 2.10 & 9.87E-04 & 2.13 & 6.74E-03 & 2.03 \\
160 & 2.24E-04 & 2.14 & 1.20E-03 & 2.17 & 2.25E-04 & 2.14 & 1.48E-03 & 2.19 \\
320 & 5.37E-05 & 2.06 & 2.77E-04 & 2.12 & 5.35E-05 & 2.07 & 3.19E-04 & 2.21 \\
\midrule
\multicolumn{9}{l}{$P^2$ WILW} \\
\midrule
40  & 4.06E-04 & --   & 2.60E-03 & --   & 3.74E-04 & --   & 2.83E-03 & --   \\
80  & 3.89E-05 & 3.39 & 3.62E-04 & 2.85 & 4.17E-05 & 3.16 & 3.82E-04 & 2.89 \\
160 & 4.21E-06 & 3.21 & 4.51E-05 & 3.00 & 4.52E-06 & 3.21 & 4.34E-05 & 3.14 \\
320 & 5.00E-07 & 3.07 & 5.53E-06 & 3.03 & 5.21E-07 & 3.12 & 5.43E-06 & 3.00 \\
\midrule
\multicolumn{9}{l}{$P^3$ WILW} \\
\midrule
40  & 3.33E-05 & --   & 2.23E-04 & --   & 2.87E-05 & --   & 2.64E-04 & --   \\
80  & 1.66E-06 & 4.33 & 1.56E-05 & 3.84 & 1.58E-06 & 4.18 & 1.56E-05 & 4.08 \\
160 & 9.15E-08 & 4.18 & 9.90E-07 & 3.97 & 9.00E-08 & 4.14 & 1.01E-06 & 3.96 \\
320 & 5.48E-09 & 4.06 & 6.06E-08 & 4.03 & 5.40E-09 & 4.06 & 6.08E-08 & 4.05 \\
\bottomrule
\end{tabular}
\end{table}

\begin{table}[htbp!]
\centering
\caption{Example \ref{eq:2D_burgers}: Numerical errors for the 2D Burgers equation at $t = 0.15$. Disk domain (\ref{2D_Wave_Disk}) with $\Delta x = 2\sqrt{0.5}/N$ and $\Delta y = 2\sqrt{0.5}/N$.}
\label{Example5 Table2}
\begin{tabular}{ccccc}
\toprule
 $N$ & $L^2$ error & Order & $L^\infty$ error & Order \\
\midrule
\multicolumn{5}{l}{$P^1$ WILW} \\
\midrule
40  & 1.49E-03 & --   & 1.46E-02 & --    \\
80  & 2.87E-04 & 2.37 & 3.17E-03 & 2.20  \\
160 & 6.28E-05 & 2.19 & 6.74E-04 & 2.23  \\
320 & 1.50E-05 & 2.06 & 1.63E-04 & 2.05  \\
\midrule
\multicolumn{5}{l}{$P^2$ WILW} \\
\midrule
40  & 1.23E-04 & --   & 2.23E-03 & --   \\
80  & 1.26E-05 & 3.28 & 1.86E-04 & 3.58 \\
160 & 1.53E-06 & 3.05 & 2.27E-05 & 3.03 \\
320 & 1.38E-07 & 3.47 & 2.52E-06 & 3.17 \\
\midrule
\multicolumn{5}{l}{$P^3$ WILW} \\
\midrule
40  & 8.27E-06 & --   & 1.23E-04 & --   \\
80  & 3.68E-07 & 4.49 & 9.10E-06 & 3.76 \\
160 & 1.78E-08 & 4.37 & 3.22E-07 & 4.82 \\
320 & 1.05E-09 & 4.09 & 2.20E-08 & 3.88 \\
\bottomrule
\end{tabular}
\end{table}

\begin{figure}[htbp]
    \centering
    \begin{subfigure}{0.45\textwidth}
        \centering
        \includegraphics[width=\textwidth]{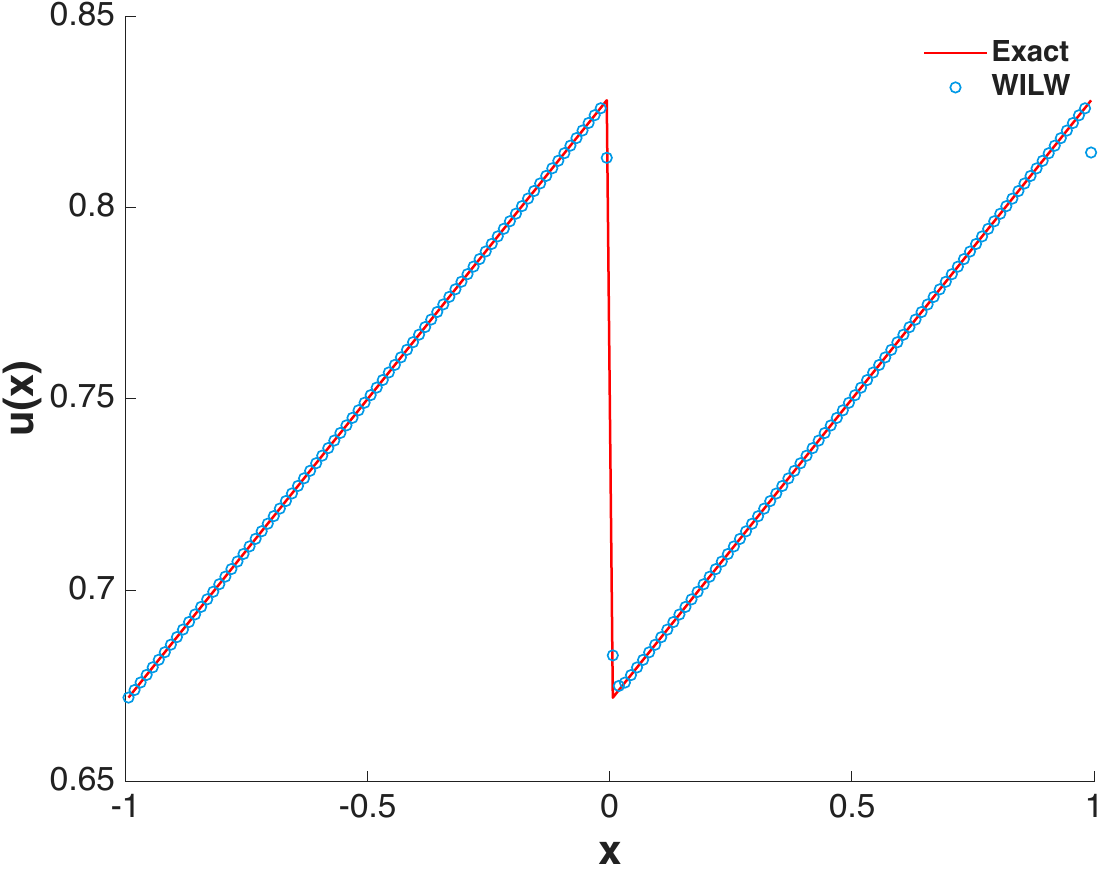}
        \caption{$\delta=0.01h$, WILW}
    \end{subfigure}
    \hspace{0.2cm}
    \begin{subfigure}{0.45\textwidth}
        \centering
        \includegraphics[width=\textwidth]{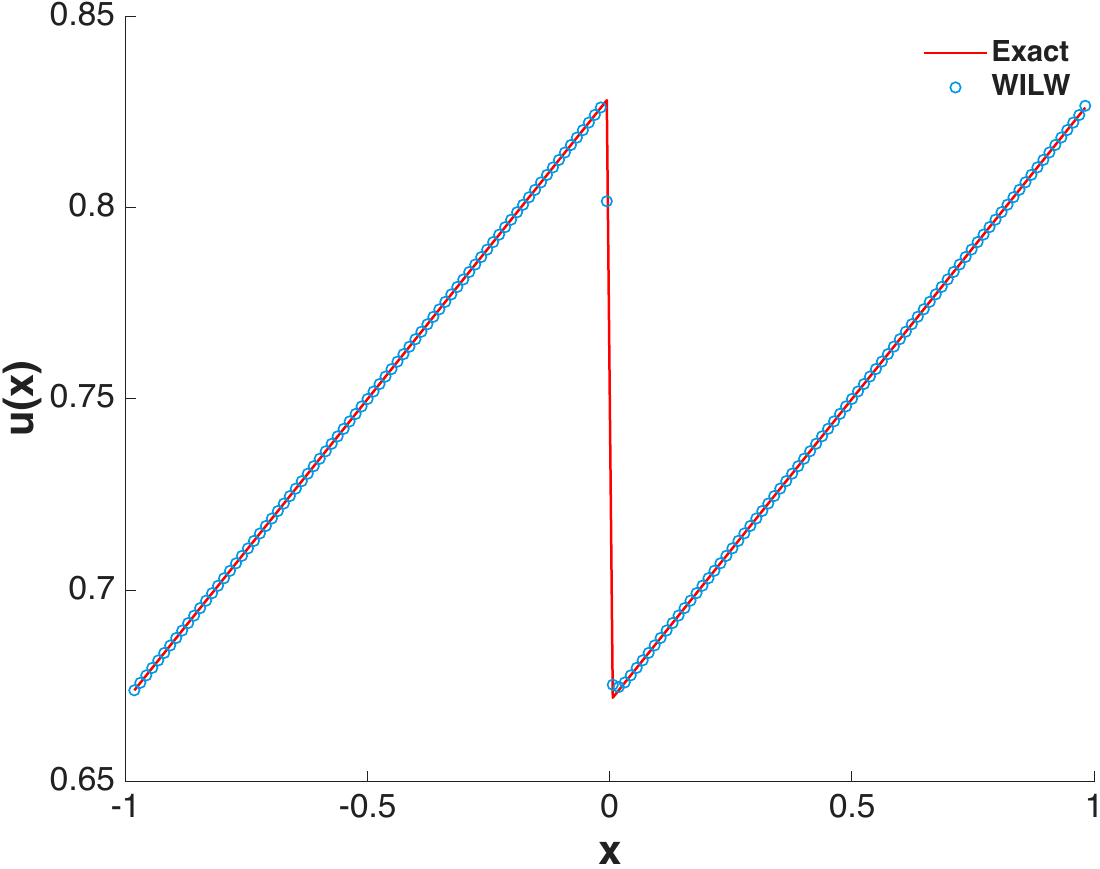}
        \caption{$\delta=0.99h$, WILW}
    \end{subfigure}
    \caption{Example \ref{eq:2D_burgers}: Numerical errors for the 2D Burgers' equation at $t = 6$, $k=2$. Square domain (\ref{2D_Wave_Square}) with $\Delta x = 2/(N+2\delta/h)$ and $\Delta y = 2/(N+2\delta/h)$,~$N=160$. Solid line: exact solution; symbols: numerical solution.}
    \label{fig:Example5.1}
\end{figure}

\begin{figure}[htbp]
    \centering 
     \includegraphics[width=0.45\textwidth]{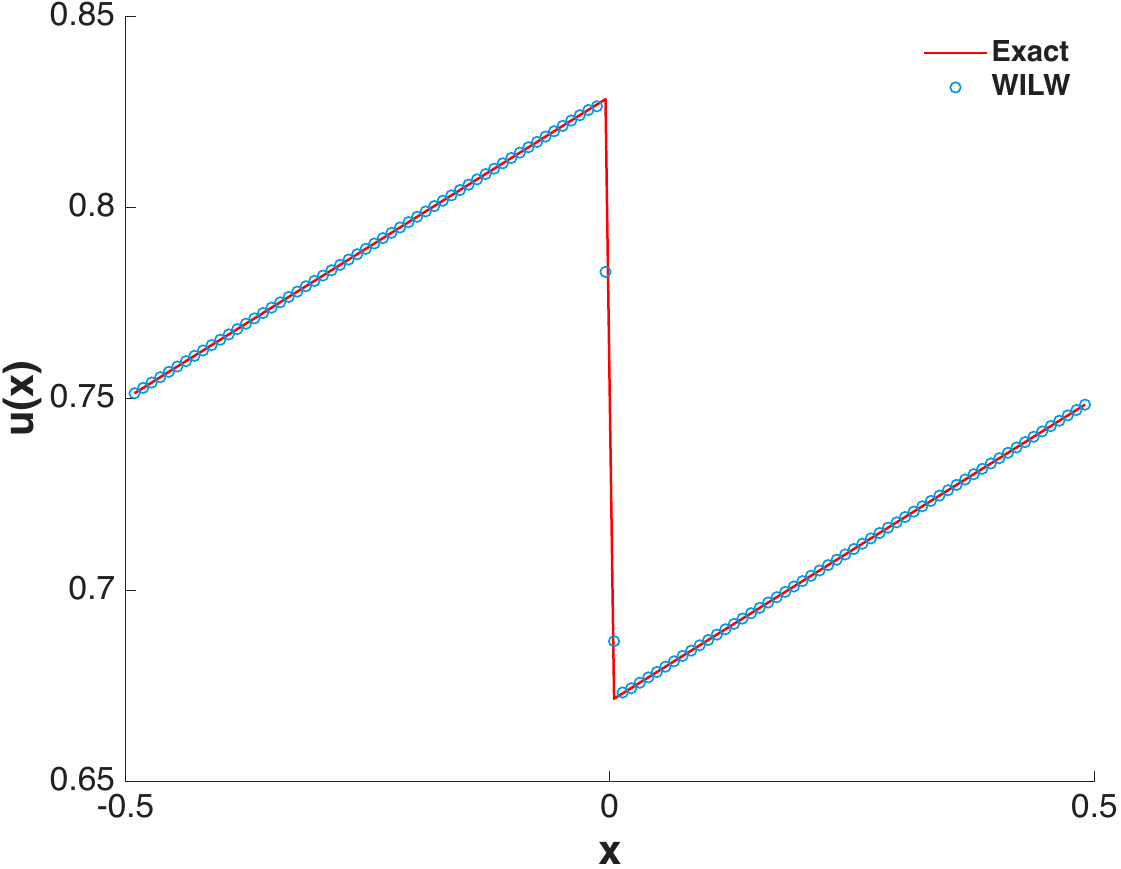}
    \caption{Example \ref{eq:2D_burgers}: Numerical errors for the 2D Burgers' equation at $t = 6$, $k=2$. Disk domain (\ref{2D_Wave_Disk}) with $\Delta x = 2\sqrt{0.5}/N$ and $\Delta y = 2\sqrt{0.5}/N$,~$N=160$. Solid line: exact solution; symbols: numerical solution.}
    \label{fig:Example5.2}
\end{figure}


\begin{example}
\textbf{(Isentropic vortex.)} \label{ex:vortex}   
\end{example}

For two-dimensional systems, we test the isentropic vortex evolution problem for the Euler equation. Similar to the previous cases, the simulations are conducted on both the rectangular domain (\ref{2D_Wave_Square}) and the circular domain (\ref{2D_Wave_Disk}). The mean flow is defined as $\rho = p = u = v = 1$. An isentropic vortex perturbation, centered at $(x_0, y_0)$, is imposed on this mean flow. The perturbations in velocity $(u, v)$ and temperature $T = p / \rho$ are given as follows, while the entropy $S = p / \rho^\gamma$ remains unperturbed ($\delta S = 0$):
$$(\delta u, \delta v) = \frac{\varepsilon}{2\pi} e^{0.5(1 - r^2)} (-\bar{y}, \bar{x}), \quad \delta T = -\frac{(\gamma - 1)\varepsilon^2}{8\gamma\pi^2} e^{1 - r^2},$$
where $(\bar{x}, \bar{y}) = (x - x_0, y - y_0)$, $r^2 = \bar{x}^2 + \bar{y}^2$, and the vortex strength is $\varepsilon = 5$. The exact solution $U(x, y, t)$ corresponds to the passive convection of the vortex with the mean velocity. Boundary conditions are prescribed using this exact solution as required. Notably, the number of boundary conditions is determined by the signs of the four eigenvalues $a_m$, which vary both spatially and temporally.
The $L^2$ and $L^\infty$ errors summarized in Tables \ref{Example6 Table1} and \ref{Example6 Table2} show that the numerical results achieve the expected convergence rates. 
 
\begin{table}[htbp]
\centering
\caption{Example \ref{ex:vortex}: Density errors for the vortex evolution problem in the square domain (\ref{2D_Wave_Square}) at $t = 1.0$, the vortex is initially positioned at $(0, 0)$, with $\Delta x = 2/(N+2\delta/h)$ and $\Delta y = 2/(N+2\delta/h)$.}
\label{Example6 Table1}
\begin{tabular}{ccccccccc}
\toprule
\multirow{2}{*}{$N$} & \multicolumn{3}{l}{$\delta = 0.01h$} & &\multicolumn{3}{l}{$\delta = 0.99h$} & \\
\cmidrule(lr){2-5} \cmidrule(lr){6-9}
 & $L^2$ error & Order & $L^\infty$ error & Order & $L^2$ error & Order & $L^\infty$ error & Order \\
\midrule
\multicolumn{9}{l}{$P^1$ WILW} \\
\midrule
40  & 6.94E-05 & -- & 2.44E-04 & --  & 1.57E-04 & -- & 3.21E-04 & -- \\
80  & 1.73E-05 & 2.00 & 6.20E-05 & 1.98 & 4.17E-05 & 1.91 & 9.41E-05 & 1.77 \\
160 & 4.32E-06 & 2.00 & 1.55E-05 & 2.00 & 1.08E-05 & 1.95 & 2.94E-05 & 1.68 \\
320 & 1.08E-06 & 2.00 & 3.88E-06 & 2.00 & 2.76E-06 & 1.97 & 8.54E-06 & 1.78 \\
\midrule
\multicolumn{9}{l}{$P^2$ WILW} \\
\midrule
40  & 1.57E-06 & -- & 4.81E-06 & -- & 1.49E-06 & -- & 5.50E-06 & -- \\
80  & 2.42E-07 & 2.69 & 7.45E-07 & 2.69 & 2.48E-07 & 2.59 & 9.06E-07 & 2.60 \\
160 & 3.56E-08 & 2.77 & 1.22E-07 & 2.61 & 3.78E-08 & 2.72 & 1.31E-07 & 2.79 \\
320 & 5.08E-09 & 2.81 & 1.92E-08 & 2.67 & 5.82E-09 & 2.70 & 2.05E-08 & 2.68 \\
\midrule
\multicolumn{9}{l}{$P^3$ WILW} \\
\midrule
40  & 4.97E-09 & -- & 2.07E-08 & --  & 3.99E-09 & -- & 1.79E-08 & -- \\
80  & 3.06E-10 & 4.02 & 1.26E-09 & 4.04 & 2.75E-10 & 3.86 & 1.17E-09 & 3.94 \\
160 & 2.10E-11 & 3.86 & 8.85E-11 & 3.83 & 2.01E-11 & 3.77 & 9.26E-11 & 3.66 \\
320 & 1.42E-12 & 3.89 & 6.83E-12 & 3.70 & 1.44E-12 & 3.80 & 7.05E-12 & 3.72 \\
\bottomrule
\end{tabular}
\end{table}
\FloatBarrier

\begin{table}[!ht]
\centering
\caption{Example \ref{ex:vortex}: Density errors for the vortex evolution problem in the disk domain (\ref{2D_Wave_Disk}) at $t = 0.1$, the vortex is initially positioned at $(0.3, 0.3)$, with $\Delta x = 2\sqrt{0.5}/N$ and $\Delta y = 2\sqrt{0.5}/N$.}
\label{Example6 Table2}
\begin{tabular}{ccccc}
\toprule
 $N$ & $L^2$ error & Order & $L^\infty$ error & Order \\
\midrule
\multicolumn{5}{l}{$P^1$ WILW} \\
\midrule
40  & 5.47E-05 & 2.07 & 2.06E-04 & 1.95 \\
80  & 1.39E-05 & 1.98 & 5.24E-05 & 1.98 \\
160 & 3.54E-06 & 1.97 & 1.38E-05 & 1.92 \\
320 & 9.03E-07 & 1.97 & 4.03E-06 & 1.78 \\
\midrule
\multicolumn{5}{l}{$P^2$ WILW} \\
\midrule
40  & 9.11E-07 & --   & 3.03E-06 & --   \\
80  & 1.50E-07 & 2.61 & 5.70E-07 & 2.41 \\
160 & 2.35E-08 & 2.67 & 1.05E-07 & 2.44 \\
320 & 3.70E-09 & 2.67 & 1.64E-08 & 2.68 \\
\midrule
\multicolumn{5}{l}{$P^3$ WILW} \\
\midrule
40  & 3.30E-09 & --   & 1.56E-08 & --   \\
80  & 2.06E-10 & 4.00 & 9.55E-10 & 4.03 \\
160 & 1.39E-11 & 3.89 & 6.36E-11 & 3.91 \\
320 & 9.83E-13 & 3.82 & 4.54E-12 & 3.81 \\
\bottomrule
\end{tabular}
\end{table}
\FloatBarrier


\begin{example}
\textbf{(Double Mach reflection.)}\label{ex:double_mach}
\end{example}

Next we apply our method to the solid wall boundary condition $(u,v) \cdot \mathbf{n} = 0$ by considering the double Mach reflection problem. This case simulates a right moving strong shock wave with Mach number 10 hitting a wedge with an inclination angle of $30^\circ$. This interaction generates a complex double Mach reflection wave system. To implement the solid wall condition using the reflection technique, the standard approach is to solve an equivalent problem. In this setup the solid wall remains horizontal while the shock wave forms a $60^\circ$ angle with the wall.

To demonstrate the capability of our method in handling non-grid-aligned boundaries, we adopt the original physical configuration. We use $(0, 23/12)$ as an anchor point to partition the domain into a uniform grid where $\Delta x = \Delta y$ along the $x$ axis and $y$ axis as shown in Fig. \ref{DMRDomain}. Initially, a Mach 10 strong shock is positioned at $x = 0$, perpendicular to the x-axis and moving to the right. For the bottom boundary $y = 0$, exact post-shock conditions are imposed.  The top boundary $y = 23/12 + \sqrt{3}/2$ is set as a time-dependent Dirichlet boundary. 

\begin{figure}[htbp!]
    \centering 
    \includegraphics[width=0.6\textwidth]{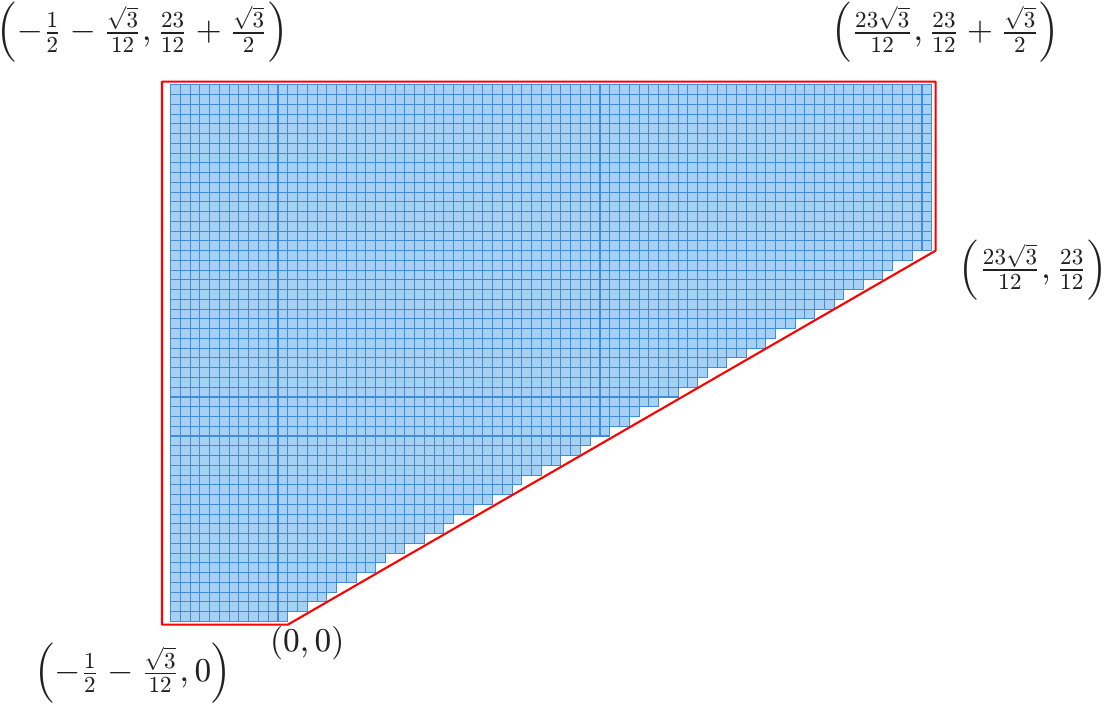}
    \caption{Example \ref{ex:double_mach}: Computational domain for Double Mach Reflection.}
    
    \label{DMRDomain}
\end{figure}

At the solid wall boundaries, WILW is employed for boundary treatment. Based on the physical constraint of zero normal velocity, the boundary states are reconstructed by extrapolating three outflow characteristic variables. To suppress spurious numerical oscillations triggered by the strong shock, we adopt the OFDG scheme.  Notably, the numerical stiffness induced by the damping term in the strong shock region severely restricts the time step size of the standard RK3 method. Consequently, we adopt the third order modified exponential Runge-Kutta method described in Reference \cite{ofdg2022} for time marching. This approach restores the time step to the level of standard DG schemes. Fig. \ref{fig:Example7.1} and \ref{fig:Example7.2} illustrate the density contours and zoomed in views near the double Mach stem for $k=2$ at $t=0.2$ with grid sizes $\Delta x = \Delta y = 1/320$ and $1/640$.
Results show that the resolution of flow structures improves significantly as the mesh is refined. The scheme clearly captures the vortex structures triggered by the Kelvin-Helmholtz instability at the Mach stem.

\begin{figure}[htbp!]
    \centering
    \begin{subfigure}{0.45\textwidth}
        \centering
        \includegraphics[width=\textwidth]{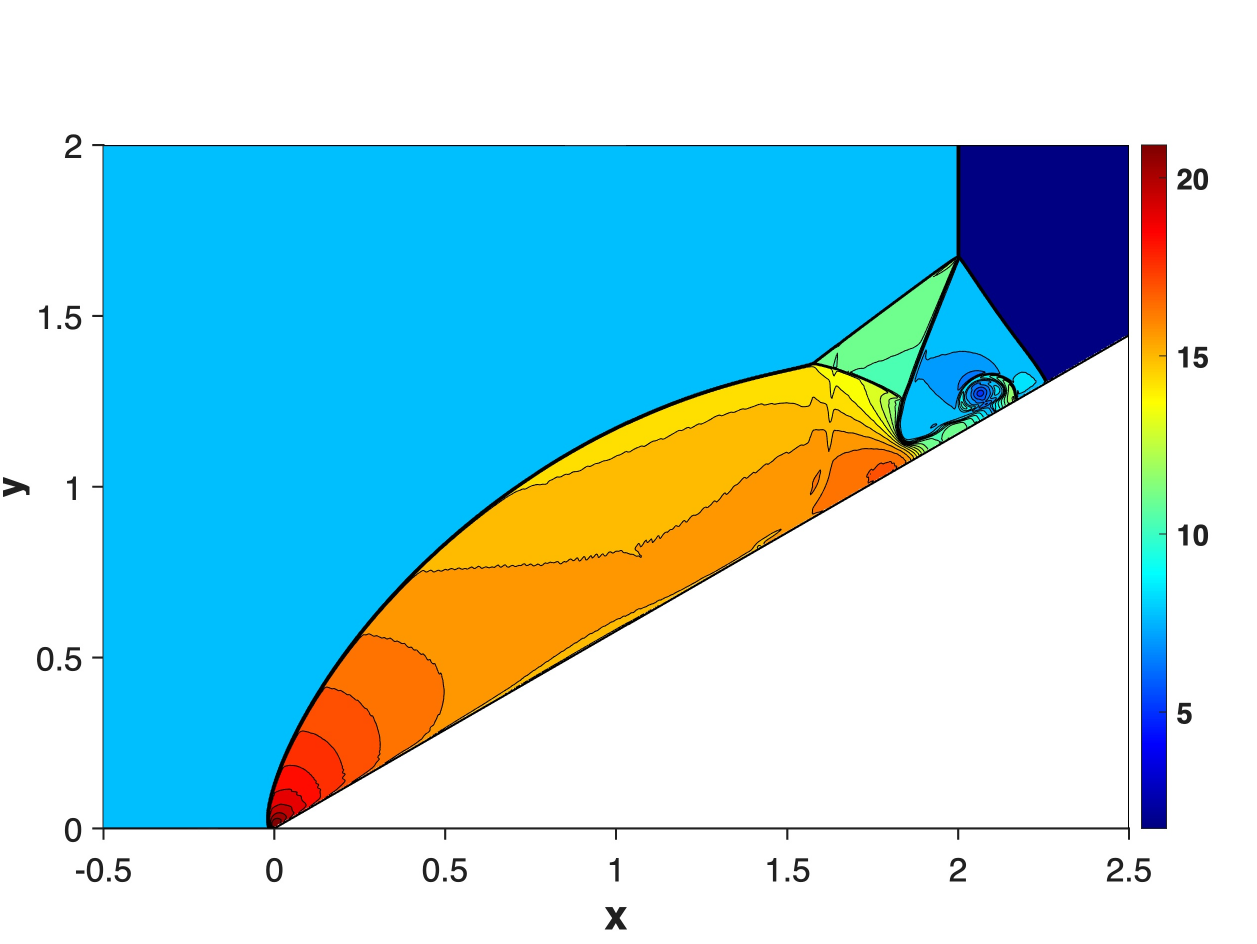}
        \caption{WILW}
    \end{subfigure}
    \hspace{0.1cm}
    \begin{subfigure}{0.45\textwidth}
        \centering
        \includegraphics[width=\textwidth]{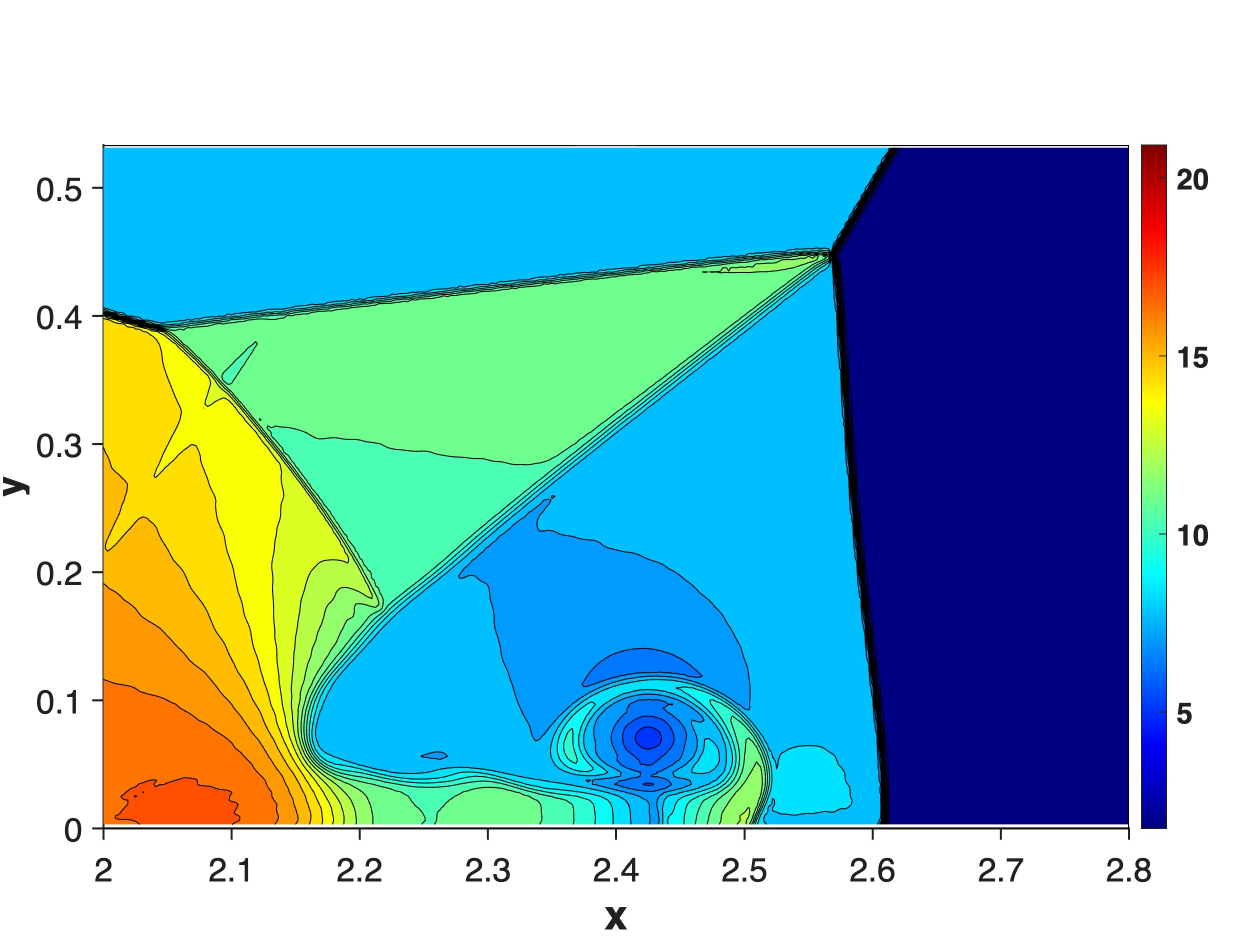}
        \caption{Mach stem in (c)}
    \end{subfigure}
    \caption{Example \ref{ex:double_mach}: Density contours of double Mach reflection at $t=0.2$, 30 contours from 1.731 to 20.92. $\Delta x=\Delta y=1/320$, $k=2$.}
    \label{fig:Example7.1}
\end{figure}

\begin{figure}[htbp!]
    \centering
    \begin{subfigure}{0.45\textwidth}
        \centering
        \includegraphics[width=\textwidth]{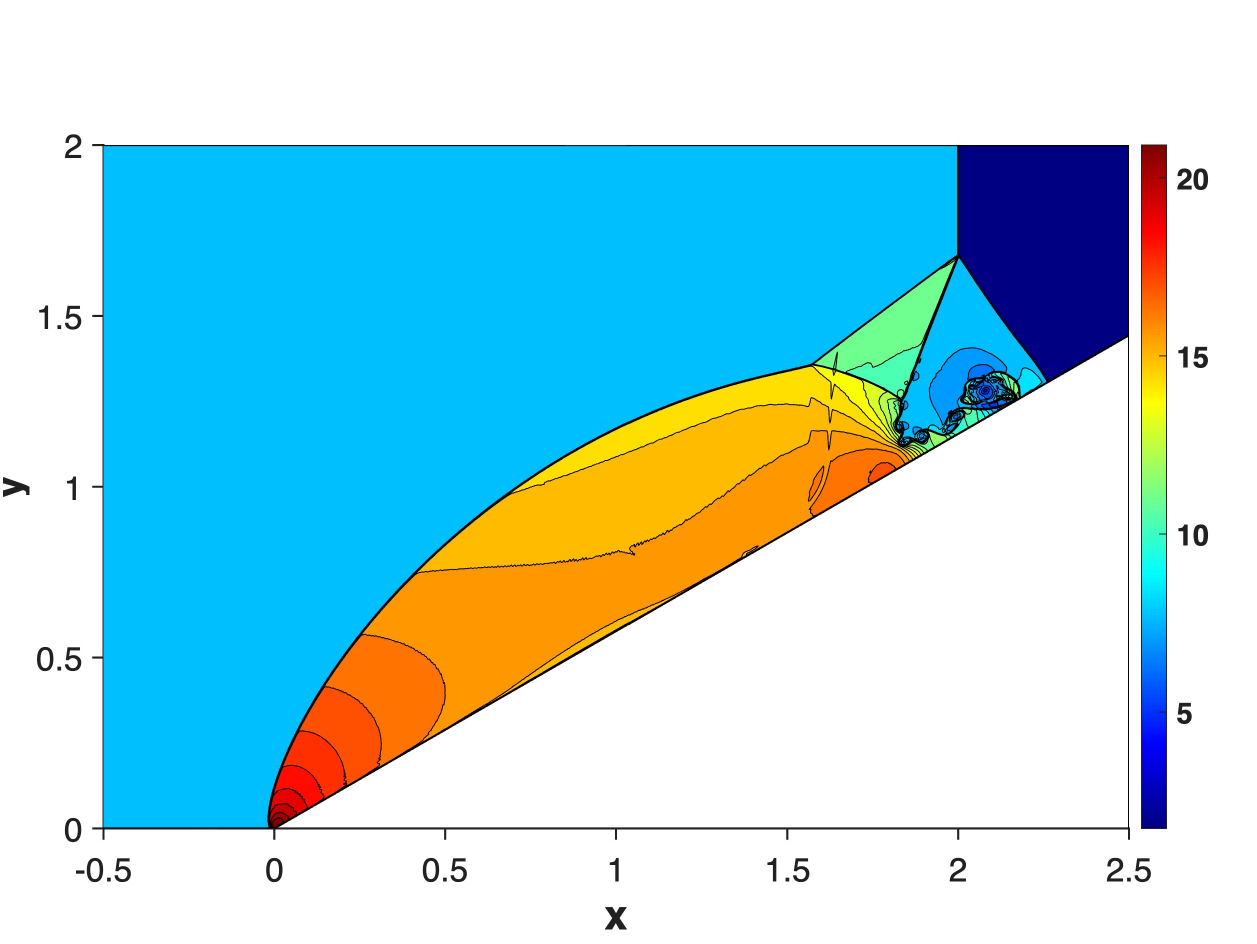}
        \caption{WILW}
    \end{subfigure}
    \hspace{0.1cm}
    \begin{subfigure}{0.45\textwidth}
        \centering
        \includegraphics[width=\textwidth]{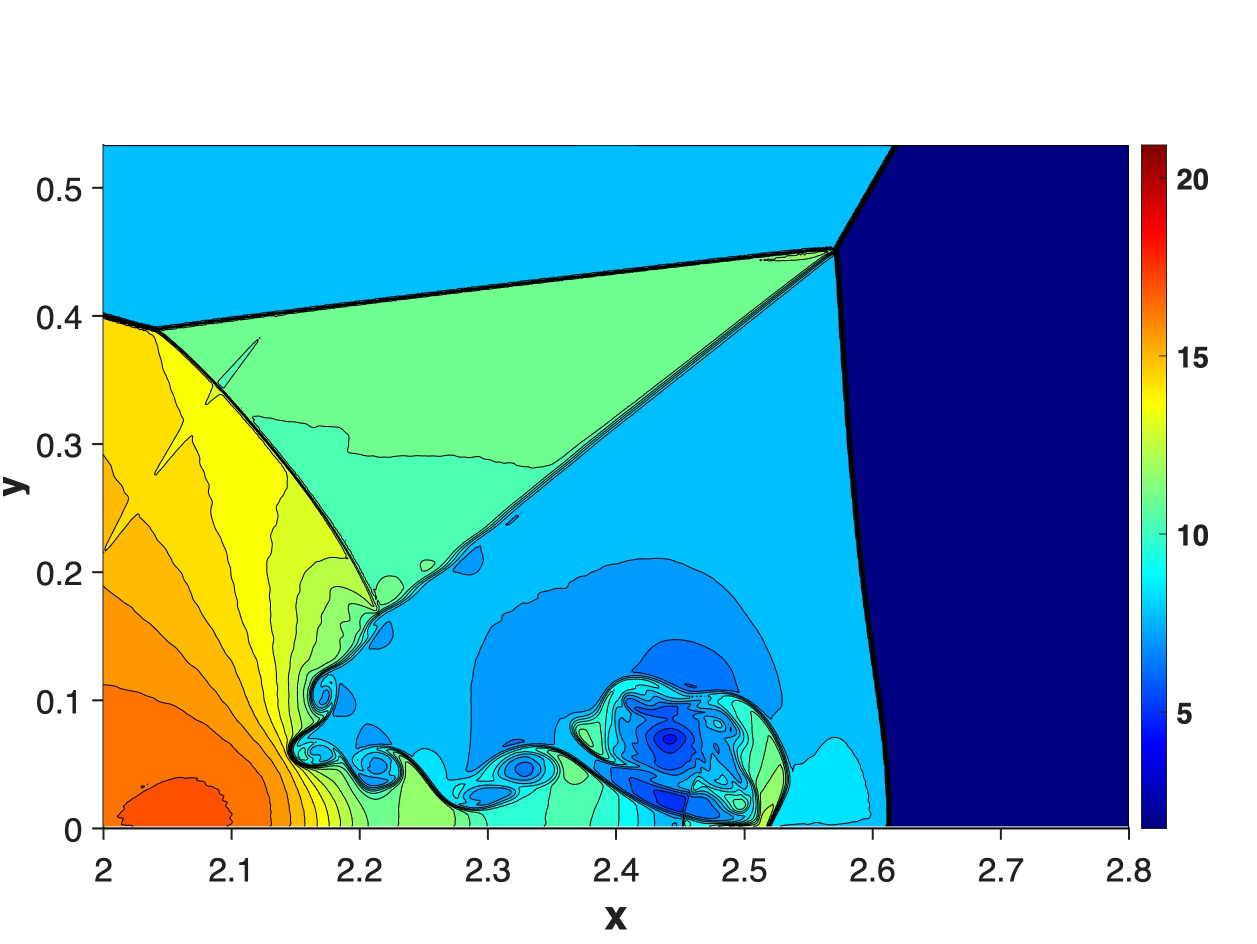}
        \caption{Mach stem in (c)}
    \end{subfigure}
    \caption{Example \ref{ex:double_mach}: Density contours of double Mach reflection at $t=0.2$, 30 contours from 1.731 to 20.92. $\Delta x=\Delta y=1/640$, $k=2$.}
    \label{fig:Example7.2}
\end{figure}
\FloatBarrier


\begin{example}
\textbf{Flow past a cylinder.}\label{ex:cylinder}
\end{example}

The final example considers a Mach 3 flow moving toward a circular cylinder centered at the origin of the $x$-$y$ plane with a radius of 1.  The physical domain is $[-3, 3] \times [-6, 6]$ excluding the cylinder area. We apply inflow boundary conditions at the left boundary $x = -3$ and outflow conditions at $x = 3$ and $y = \pm 6$. On the cylinder surface, we prescribe the solid wall condition $(u,v) \cdot \mathbf{n} = 0$ and implement the boundary treatment. Using the origin as the anchor point, we partition the computational domain into a uniform grid with $\Delta x = \Delta y$ along the x axis and y axis to obtain the effective computational domain as shown in Fig. \ref{FPCDomain}. We adopt the same OFDG scheme as the previous example along with the third order modified exponential Runge-Kutta method for temporal discretization. Fig. \ref{fig:Example8} show the pressure contours for $k=2$ at $t=40$ with mesh sizes $\Delta x = \Delta y = 1/20$ and $1/40$ respectively. It is evident that our method captures the bow shock effectively.

\begin{figure}[htbp!]
    \centering 
    \includegraphics[width=0.4\textwidth]{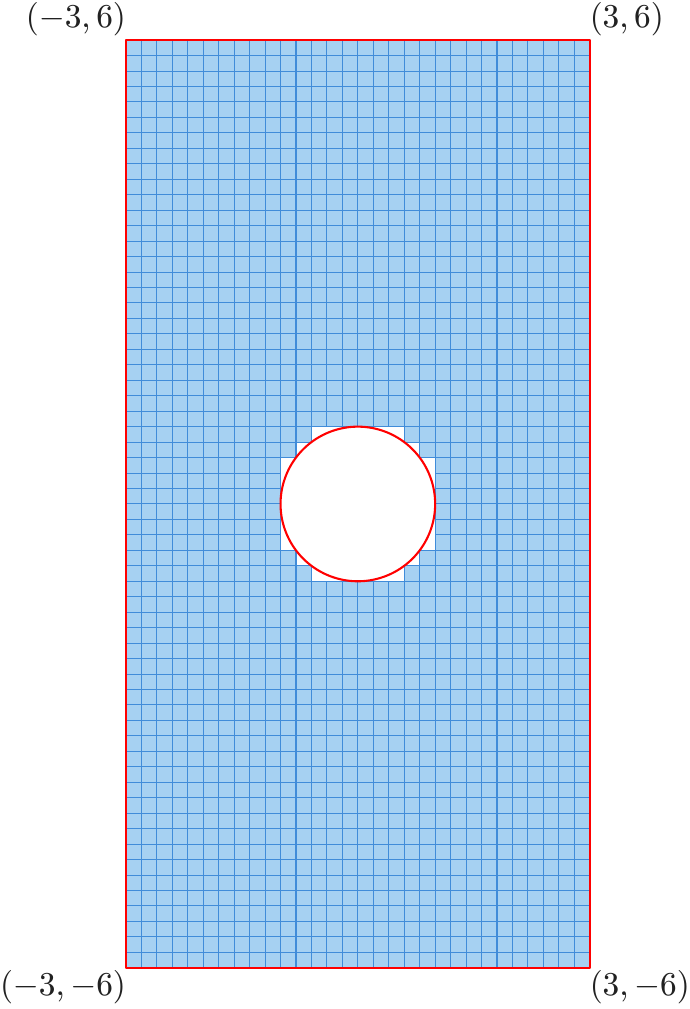}
    \caption{Example \ref{ex:cylinder}: Computational domain for flow past a cylinder.}
    \label{FPCDomain}
\end{figure}

\begin{figure}[htbp!]
    \centering
    \begin{subfigure}{0.40\textwidth}
        \centering
        \includegraphics[width=\textwidth]{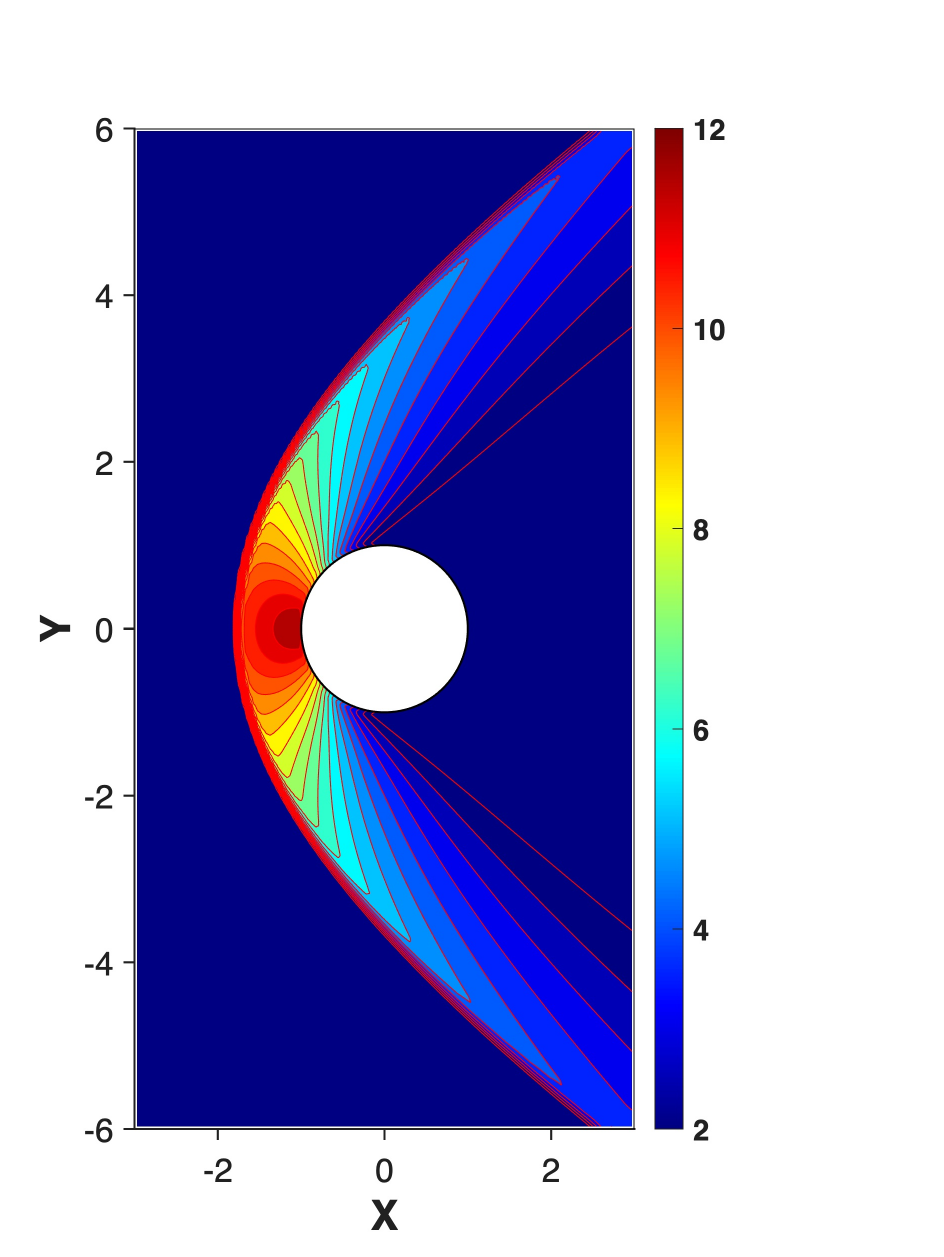}
        \caption{$\Delta x=\Delta y=1/20$, WILW}
    \end{subfigure}
    \hspace{0.2cm}
    \begin{subfigure}{0.40\textwidth}
        \centering
        \includegraphics[width=\textwidth]{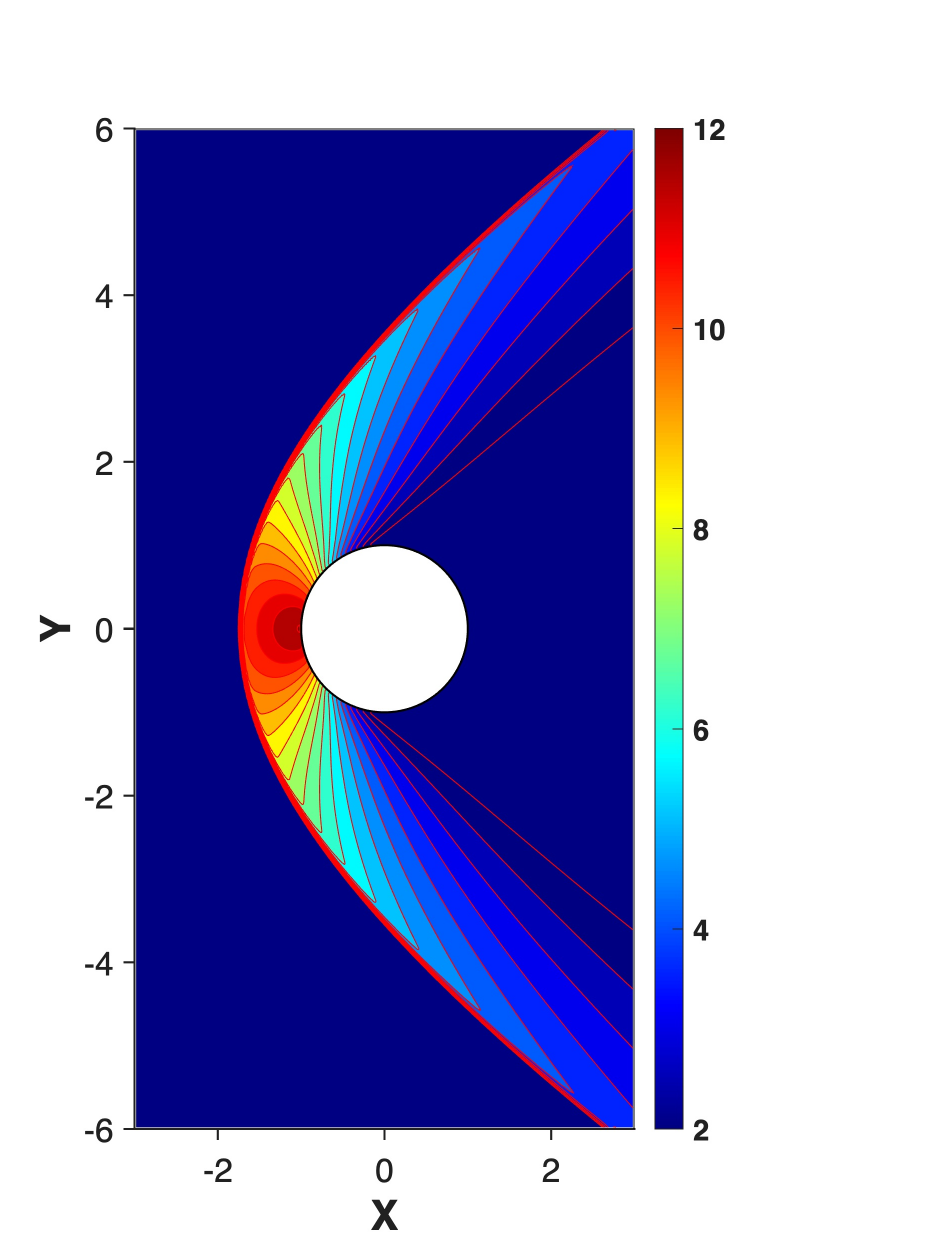}
        \caption{$\Delta x=\Delta y=1/40$, WILW}
    \end{subfigure}
    \caption{Example \ref{ex:cylinder}: Pressure contour of flow past a cylinder at $t=40$, 20 contours from 2 to 12, $k=2$.}
    \label{fig:Example8}
\end{figure}

\section{Conclusion}
In this paper, we propose a novel ILW boundary treatment for the DG method on unfitted meshes to efficiently solve hyperbolic conservation laws. While a standard RKDG scheme is utilized for internal cells, the ILW method is applied to construct reconstruction polynomials for the small cut cells near physical boundaries. To mitigate the sensitivity of truncation errors to the boundary offset scale $\delta/h$, we enhance the boundary reconstruction by one order of accuracy. Furthermore, by incorporating internal information through weighted least squares, a WILW method is developed, which reduces the reliance on high-order boundary derivatives and significantly enhances computational efficiency. Furthermore, we use eigenvalue analysis to evaluate the stability of the boundary treatment, for both semi-discrete and fully discrete schemes. A series of numerical experiments show that our method completely avoids the typical small time step problem, achieves the expected accuracy, and demonstrates the capability to handle shocks passing through the boundaries.

\appendix
\section{Expression of the coefficient matrix \texorpdfstring{$\mathbf{Q}$}{Q}}\label{MatrixQ}
The coefficient matrix $\mathbf{Q}$ of the WILW method is constructed from the matrices $\mathbf{A}$, $\mathbf{B}$, $\mathbf{C}$, and $\mathbf{D}$. For $k=2,3$, their explicit expressions are listed below.

\vspace{0.5cm}

\noindent\textbf{WILW ($k=2$)}
{\scriptsize

{\renewcommand{\arraystretch}{1.25}
$$
\mathbf A=
\begin{pmatrix}
\frac{23}{16} & -\frac{115}{24} & \frac{115}{16}\\
-\frac{9}{16} & \frac{15}{8} & -\frac{45}{16}\\
-\frac{1}{16} & \frac{5}{24} & -\frac{5}{16}
\end{pmatrix},\quad
\mathbf B=
\begin{pmatrix}
-\frac{43}{16} & -\frac{29}{24} & \frac{1}{16}\\
\frac{69}{16} & -\frac{15}{8} & -\frac{15}{16}\\
-\frac{19}{16} & \frac{139}{24} & -\frac{71}{16}
\end{pmatrix},$$}
$$
\mathbf C=\mathbf B+L_1(\eta,h)\mathbf K_2+L_2(\eta,h)\mathbf K_3,
\quad
\mathbf D=L_3(\eta,h)\mathbf K_2+L_4(\eta,h)\mathbf K_3,
$$
where
{\renewcommand{\arraystretch}{1.25}
$$
\mathbf K_2=
\begin{pmatrix}
\frac{23}{16} & \frac{23}{24} & \frac{23}{16}\\
-\frac{9}{16} & -\frac{3}{8} & -\frac{9}{16}\\
-\frac{1}{16} & -\frac{1}{24} & -\frac{1}{16}
\end{pmatrix},
\quad
\mathbf K_3=
\begin{pmatrix}
-\frac{23}{4} & 0 & \frac{23}{4}\\
\frac{9}{4} & 0 & -\frac{9}{4}\\
\frac{1}{4} & 0 & -\frac{1}{4}
\end{pmatrix},$$}

\begin{align*}
L_1(\eta,h)&=
\frac{\eta\left(
\begin{aligned}
&(108+144h^2)\eta^5
+(672+864h^2)\eta^4
+(1605+1872h^2)\eta^3\\
&+(1810+1764h^2)\eta^2
+(933+684h^2)\eta
+(161+72h^2)
\end{aligned}
\right)}
{D(\eta,h)},
\end{align*}

\vspace{-1em}

\begin{align*}
L_2(\eta,h)&=
-\frac{\eta\left(
\begin{aligned}
&(24+240h^2)\eta^5
+(114+1440h^2)\eta^4
+(195+3120h^2)\eta^3\\
&+(145+2862h^2)\eta^2
+(45+897h^2)\eta
+(4-63h^2)
\end{aligned}
\right)}
{D(\eta,h)},
\end{align*}

\vspace{-1em}

\begin{align*}
L_3(\eta,h)&=
-\frac{\eta\left(
\begin{aligned}
&(36-144h^2)\eta^5
+(240-864h^2)\eta^4
+(579-1872h^2)\eta^3\\
&+(622-1692h^2)\eta^2
+(291-468h^2)\eta
+(47+72h^2)
\end{aligned}
\right)}
{D(\eta,h)},
\end{align*}

\vspace{-1em}

\begin{align*}
L_4(\eta,h)&=
\frac{\eta\left(
\begin{aligned}
&(24-48h^2)\eta^5
+(138-288h^2)\eta^4
+(297-624h^2)\eta^3\\
&+(293-558h^2)\eta^2
+(129-147h^2)\eta
+(20+27h^2)
\end{aligned}
\right)}
{D(\eta,h)}.
\end{align*}

\vspace{-1em}

\begin{align*}
D(\eta,h)
&=
(72+288h^2)\eta^6
+(432+1728h^2)\eta^5
+(1026+3744h^2)\eta^4 \\
&\quad
+(1224+3456h^2)\eta^3
+(762+1152h^2)\eta^2
+228\eta
+(26+18h^2).
\end{align*}
}

\noindent\textbf{WILW ($k=3$)}
{\scriptsize

{\renewcommand{\arraystretch}{1.25}
$$
\mathbf A=
\begin{pmatrix}
-\frac{3666765}{2097152} & \frac{12033}{2048} & -\frac{20055}{2048} & \frac{20055}{2048}\\
\frac{3183785}{2097152} & -\frac{7077}{2048} & \frac{11795}{2048} & -\frac{11795}{2048}\\
\frac{4447735}{2097152} & \frac{1533}{2048} & -\frac{2555}{2048} & \frac{2555}{2048}\\
\frac{1141485}{2097152} & \frac{2583}{2048} & -\frac{4305}{2048} & \frac{4305}{2048}
\end{pmatrix},
\quad
\mathbf B=
\begin{pmatrix}
-\frac{15109}{6144} & -\frac{4521}{2048} & \frac{255}{2048} & \frac{223205}{3145728}\\
\frac{43577}{6144} & -\frac{7699}{2048} & -\frac{1115}{2048} & -\frac{148363}{1048576}\\
-\frac{11761}{6144} & \frac{10747}{2048} & -\frac{5629}{2048} & -\frac{3619775}{3145728}\\
-\frac{4723}{6144} & -\frac{7983}{2048} & \frac{21993}{2048} & -\frac{22131013}{3145728}
\end{pmatrix},
$$}

$$
\mathbf C=\mathbf B+L_1(\eta,h)\mathbf K_2+L_2(\eta,h)\mathbf K_3,
\quad
\mathbf D=L_3(\eta,h)\mathbf K_2+L_4(\eta,h)\mathbf K_3,
$$
where
{\renewcommand{\arraystretch}{1.25}
$$
\mathbf K_2=
\begin{pmatrix}
\frac{2483}{2048} & \frac{2101}{2048} & \frac{2101}{2048} & \frac{2483}{2048}\\
-\frac{4381}{6144} & -\frac{3707}{6144} & -\frac{3707}{6144} & -\frac{4381}{6144}\\
\frac{949}{6144} & \frac{803}{6144} & \frac{803}{6144} & \frac{949}{6144}\\
\frac{533}{2048} & \frac{451}{2048} & \frac{451}{2048} & \frac{533}{2048}
\end{pmatrix},
\quad
\mathbf K_3=
\begin{pmatrix}
-\frac{2865}{256} & \frac{4011}{256} & -\frac{4011}{256} & \frac{2865}{256}\\
\frac{1685}{256} & -\frac{2359}{256} & \frac{2359}{256} & -\frac{1685}{256}\\
-\frac{365}{256} & \frac{511}{256} & -\frac{511}{256} & \frac{365}{256}\\
-\frac{615}{256} & \frac{861}{256} & -\frac{861}{256} & \frac{615}{256}
\end{pmatrix},
$$}

\begin{align*}
L_1(\eta,h)
&=
\frac{
\eta^2\left(
\begin{aligned}
&(2700+3600h^2)\eta^{10}
+(33600+43200h^2)\eta^9
+(186900+223200h^2)\eta^8\\
&+(610720+651600h^2)\eta^7
+(1294905+1186560h^2)\eta^6
+(1855680+1405080h^2)\eta^5\\
&+(1813830+1095240h^2)\eta^4
+(1190406+556920h^2)\eta^3
+(502053+179400h^2)\eta^2\\
&+(124138+34560h^2)\eta
+(14172+3360h^2)
\end{aligned}
\right)}
{2D(\eta,h)},
\end{align*}

\vspace{-1em}

\begin{align*}
L_2(\eta,h)
&=
-\frac{
\eta^2\left(
\begin{aligned}
&(300+3000h^2)\eta^{10}
+(2850+36000h^2)\eta^9
+(11760+186240h^2)\eta^8\\
&+(27665+541980h^2)\eta^7
+(41010+966240h^2)\eta^6
+(40035+1063380h^2)\eta^5\\
&+(26146+662734h^2)\eta^4
+(11364+141477h^2)\eta^3
+(3204-82293h^2)\eta^2\\
&+(552-55353h^2)\eta
+(48-8373h^2)
\end{aligned}
\right)}
{D(\eta,h)},
\end{align*}

\vspace{-1em}

\begin{align*}
L_3(\eta,h)
&=
-\frac{
\eta^2\left(
\begin{aligned}
&(900-3600h^2)\eta^{10}
+(12000-43200h^2)\eta^9
+(68100-223200h^2)\eta^8\\
&+(218320-644400h^2)\eta^7
+(439575-1121760h^2)\eta^6
+(582240-1163880h^2)\eta^5\\
&+(515250-616440h^2)\eta^4
+(302346-13320h^2)\eta^3
+(113643+169800h^2)\eta^2\\
&+(25078+80640h^2)\eta
+(2532+11040h^2)
\end{aligned}
\right)}
{2D(\eta,h)},
\end{align*}

\vspace{-1em}

\begin{align*}
L_4(\eta,h)
&=
\frac{
\eta^2\left(
\begin{aligned}
&(300-600h^2)\eta^{10}
+(3450-7200h^2)\eta^9
+(17460-36960h^2)\eta^8\\
&+(51085-105180h^2)\eta^7
+(95490-178920h^2)\eta^6
+(118995-179220h^2)\eta^5\\
&+(100146-88346h^2)\eta^4
+(56388+4371h^2)\eta^3
+(20484+28641h^2)\eta^2\\
&+(4392+12745h^2)\eta
+(432+1701h^2)
\end{aligned}
\right)}
{D(\eta,h)},
\end{align*}

\vspace{-1em}

\begin{align*}
D(\eta,h)
&=
(900+3600h^2)\eta^{12}
+(10800+43200h^2)\eta^{11}
+(59400+223200h^2)\eta^{10} \\
&\quad
+(198000+648000h^2)\eta^9
+(444705+1152720h^2)\eta^8
+(706440+1272960h^2)\eta^7 \\
&\quad
+(810130+822480h^2)\eta^6
+(673020+246240h^2)\eta^5
+(401409+2976h^2)\eta^4 \\
&\quad
+(168116+13824h^2)\eta^3
+(47292+28692h^2)\eta^2
+(8112+10536h^2)\eta
+(656+1156h^2).
\end{align*}
}

\section{The Routh-Hurwitz criterion}\label{Routh-Hurwitz Criterion}
The Routh--Hurwitz criterion \cite{routh} provides a necessary and sufficient condition for determining whether all roots of a polynomial lie in the left half-plane. By constructing the Routh table and examining the sign changes in its first column, stability can be determined without explicitly computing the roots. This makes it particularly useful for polynomials with symbolic coefficients.

Consider the characteristic equation
$$ a_n s^n + a_{n-1} s^{n-1} + \cdots + a_1 s + a_0 = 0.$$
To determine whether all roots lie in the left half-plane, we construct the Routh table as follows:

\begin{itemize}
    \item First two rows: The first row consists of the coefficients of $s^n, s^{n-2}, s^{n-4}, \dots$, and the second row consists of the coefficients of $s^{n-1}, s^{n-3}, s^{n-5}, \dots$.
    $$\begin{array}{c|cccc}
    \text{Row} & \multicolumn{4}{c}{\text{Routh Table Entries}} \\
    \hline
    s^{n}   & a_n       & a_{n-2}     & a_{n-4}     & \cdots \\
    s^{n-1} & a_{n-1}   & a_{n-3}     & a_{n-5}     & \cdots \\
    s^{n-2} & b_1       & b_2         & b_3         & \cdots \\
    s^{n-3} & c_1       & c_2         & c_3         & \cdots \\
    s^{n-4} & d_1       & d_2         & d_3         & \cdots \\
    \vdots & \vdots & \vdots      & \vdots      & \ddots \\
    s^1   & y_1       &             &             &        \\
    s^0   & z_1       &             &             &        
    \end{array}
    $$
    
    \item Subsequent rows: The entries are calculated from the previous two rows using
    $$R_{i,j} = \frac{-1}{R_{i-1,1}} \cdot \left| \begin{array}{cc}
    R_{i-2,1} & R_{i-2,j+1} \\
    R_{i-1,1} & R_{i-1,j+1}
    \end{array} \right|$$
    where $R_{i,j}$ denotes the element in the $i$-th row and $j$-th column of the Routh table (assuming the $s^n$ row is $i=1$ and the $s^{n-1}$ row is $i=2$). Continue this process for $i \ge 3$ until all $n+1$ rows are completed.
\end{itemize}

\begin{proposition}[Routh--Hurwitz Criterion \cite{routh}]
A necessary and sufficient condition for all roots of a polynomial with real coefficients to have negative real parts is that all entries in the first column of the Routh table have the same sign.
\end{proposition}

\section{Positivity analysis via Bernstein basis expansion}\label{Bernstein Basis Expansion}
Given a polynomial $P(x) = \sum_{k=0}^n c_k x^k$, it can be expressed in the Bernstein basis as:
$$P(x) = \sum_{j=0}^n \beta_j B_{j,n}(x),$$
where $B_{j,n}(x) = \binom{n}{j} x^j (1-x)^{n-j}$. The coefficients $\beta_j$ are related to the monomial coefficients $c_k$ via:
$$\beta_j = \sum_{k=0}^j c_k \frac{\binom{j}{k}}{\binom{n}{k}}, \quad j=0, \dots, n.$$
The primary advantage of this expansion lies in the \textbf{Convex Hull Property} \cite{bernstein2002}. For any $x \in [0,1]$, the value of $P(x)$ is bounded by the minimum and maximum of its Bernstein coefficients $\beta_j$. In our stability analysis, we utilize this property to obtain a sufficient condition for positivity. 

\begin{proposition}[Positivity Condition \cite{bernstein2002}]\label{Positivity Condition}
If all Bernstein coefficients satisfy
$\beta_j>0$
(or $\beta_j<0$),
then $P(x)$ is strictly positive
(or strictly negative)
for all $x\in[0,1]$.
\end{proposition}

This framework naturally extends to bivariate polynomials. Consider
$$P(x,y)=\sum_{p=0}^{m}\sum_{q=0}^{n}
c_{p,q}x^p y^q,
\qquad
(x,y)\in[0,1]\times[0,1].$$
By employing tensor-product Bernstein polynomials, $P(x,y)$ can be represented as
$$P(x,y)=\sum_{i=0}^{m}\sum_{j=0}^{n}
\beta_{i,j}
B_{i,m}(x)
B_{j,n}(y),$$
where the Bernstein coefficients $\beta_{i,j}$ are obtained through the two-dimensional monomial-to-Bernstein transformation
$$\beta_{i,j}=\sum_{p=0}^{i}\sum_{q=0}^{j}
c_{p,q}
\frac{\binom{i}{p}}{\binom{m}{p}}
\frac{\binom{j}{q}}{\binom{n}{q}},
\quad
0\le i\le m,\quad
0\le j\le n.$$

\begin{proposition}[Bivariate Positivity Condition]
If all Bernstein coefficients satisfy
$\beta_{i,j}>0$
(or $\beta_{i,j}<0$),
then $P(x, y)$ is strictly positive (or strictly negative) for all $(x, y) \in [0,1] \times [0,1]$.
\end{proposition}

This result follows directly from the convex hull property of tensor-product Bernstein polynomials.In the present work, these positivity conditions are applied to the numerator and denominator polynomials $b_i(\eta,h)$ and $a_i(\eta,h)$ to determine the sign of $R_{i,1}(\eta,h)$.

\section{The first-column entries of the Routh table}\label{Routh-Table}

For the semi-discrete case with $k=1$, the first-column entries of the corresponding Routh table have already been presented in the main text. Here, we list those for the fully discrete case with $k=1$.

For $k=2,3$, due to space constraints, we present only the first four first-column entries for the semi-discrete case with $k=2$. The remaining entries are omitted, but have the same structure, namely rational functions of $(\eta,h)\in[0,1)\times(0,1)$, whose positivity is established through the Bernstein basis expansion described in Appendix~\ref{Bernstein Basis Expansion}.

\vspace{0.5cm}

{\scriptsize

\noindent\textbf{WILW (Full-discrete) ($k=1$)} 
\begin{itemize}
\item $R_{1,1} = \begin{aligned}[t] & 21959667\eta^6 + 127226538\eta^5 + 284298336\eta^4 + 307500624\eta^3 \\ & + 167438340\eta^2 + 44562168\eta + 4648336 \end{aligned}$ 

\item $R_{2, 1} = \begin{aligned}[t] & 22313961\eta^6 + 132834006\eta^5 + 306373914\eta^4 + 344406816\eta^3 \\ & + 196678368\eta^2 + 55122480\eta + 6036800 \end{aligned}$ 

\item $R_{3, 1} = \frac{\left( \begin{aligned} & 150501151268388\eta^{12} + 1823935140898632\eta^{11} + 9835230270319956\eta^{10} \\ & + 31113308630245536\eta^9 + 64129513047069636\eta^8 + 90506042881514208\eta^7 \\ & + 89521488802383768\eta^6 + 62486929575605232\eta^5 + 30564759230051040\eta^4 \\ & + 10232699434788672\eta^3 + 2230307723710080\eta^2 + 284813523816192\eta \\ & + 16153221145600 \end{aligned} \right)}{\left(\begin{aligned} & 22313961\eta^6 + 132834006\eta^5 + 306373914\eta^4 + 344406816\eta^3 \\ & + 196678368\eta^2 + 55122480\eta + 6036800 \end{aligned}\right)}$ 

\item $R_{4,1}
=
\frac{
\left(
\begin{aligned}
&169181053161517781784\eta^{17}
+2828501795356996608099\eta^{16}\\
&+21818807655771378508524\eta^{15}
+103030750368521152653258\eta^{14}\\
&+333236694395354695923636\eta^{13}
+782706350376382736425968\eta^{12}\\
&+1381353717422389541834736\eta^{11}
+1869758579503446384161376\eta^{10}\\
&+1964102465486774313274752\eta^{9}
+1609743340206592272339840\eta^{8}\\
&+1029131442729211586989824\eta^{7}
+510273674606801964155904\eta^{6}\\
&+193836741000380883293184\eta^{5}
+55240609390183510308864\eta^{4}\\
&+11409861112111988563968\eta^{3}
+1610344946594299281408\eta^{2}\\
&+138751580791111680000\eta
+5499573094400000000
\end{aligned}
\right)
}{
2\left(
\begin{aligned}
&37625287817097\eta^{12}
+455983785224658\eta^{11}
+2458807567579989\eta^{10}\\
&+7778327157561384\eta^{9}
+16032378261767409\eta^{8}
+22626510720378552\eta^{7}\\
&+22380372200595942\eta^{6}
+15621732393901308\eta^{5}
+7641189807512760\eta^{4}\\
&+2558174858697168\eta^{3}
+557576930927520\eta^{2}
+71203380954048\eta\\
&+4038305286400
\end{aligned}
\right)
}$

\item $R_{5, 1} = \begin{aligned}[t] & 374706\eta^4 + 1426896\eta^3 + 1764180\eta^2 + 772488\eta + 110224 \end{aligned}$ 
\end{itemize}

\noindent\textbf{WILW (semi-discrete) ($k=2$)} 
\begin{itemize}
\item $R_{1,1} = 1$
\item $R_{2,1} = \frac{
    \left(
   \begin{aligned}
    &3600 \eta^6 h^2 + 1044 \eta^6 
    + 21600 \eta^5 h^2 + 6420 \eta^5 \\
    &+ 46800 \eta^4 h^2 + 15693 \eta^4 
    + 43272 \eta^3 h^2 + 19352 \eta^3 \\
    &+ 14670 \eta^2 h^2 + 12513 \eta^2 
    + 306 \eta h^2 + 3919 \eta \\
    &+ 324 h^2 + 468
  \end{aligned}
  \right)}{
  2\left(
    \begin{aligned}
      &144 \eta^6 h^2 + 36 \eta^6 
      + 864 \eta^5 h^2 + 216 \eta^5 \\
      &+ 1872 \eta^4 h^2 + 513 \eta^4 
      + 1728 \eta^3 h^2 + 612 \eta^3 \\
      &+ 576 \eta^2 h^2 + 381 \eta^2 
      + 114 \eta + 9 h^2 + 13
    \end{aligned}
  \right)}$

\item $R_{3,1} = \frac{
    3\left(
    \begin{aligned}
    &7119360\eta^{12}h^{4} + 5380992\eta^{12}h^{2} + 993600\eta^{12} + 85432320\eta^{11}h^{4} + 65832192\eta^{11}h^{2} + 12411936\eta^{11} \\
    &+ 441400320\eta^{10}h^{4} + 354042432\eta^{10}h^{2} + 69523992\eta^{10} + 1282395456\eta^{9}h^{4} + 1100354928\eta^{9}h^{2}\\
    &+ 230484552\eta^{9} + 2293837056\eta^{8}h^{4} + 2184628032\eta^{8}h^{2} + 502589250\eta^{8}+ 2596123872\eta^{7}h^{4}\\
    &+ 2889389772\eta^{7}h^{2} + 757477230\eta^{7}
    + 1835752464\eta^{6}h^{4} + 2574302508\eta^{6}h^{2} + 806672538\eta^{6} \\
    &+ 772423128\eta^{5}h^{4} + 1527991470\eta^{5}h^{2} + 609566547\eta^{5}
    + 183679920\eta^{4}h^{4} + 587119923\eta^{4}h^{2} \\
    &+ 323289833\eta^{4}
    + 32446116\eta^{3}h^{4} + 142024332\eta^{3}h^{2} + 117035395\eta^{3}
    + 7344918\eta^{2}h^{4}\\
    &+ 22971285\eta^{2}h^{2} + 27415804\eta^{2}
    + 395766\eta h^{4} + 3155823\eta h^{2} + 3732677\eta \\
    &+ 107244h^{4} + 309816h^{2} + 223756
    \end{aligned}
    \right)
}{
    \left(
    \begin{aligned}
    &144\eta^{6}h^{2} + 36\eta^{6} + 864\eta^{5}h^{2} + 216\eta^{5} \\
    &+ 1872\eta^{4}h^{2} + 513\eta^{4} + 1728\eta^{3}h^{2} + 612\eta^{3} \\
    &+ 576\eta^{2}h^{2} + 381\eta^{2} + 114\eta + 9h^{2} + 13
    \end{aligned}
    \right)
    \left(
    \begin{aligned}
    &3600\eta^{6}h^{2} + 1044\eta^{6} + 21600\eta^{5}h^{2} + 6420\eta^{5} \\
    &+ 46800\eta^{4}h^{2} + 15693\eta^{4} + 43272\eta^{3}h^{2} + 19352\eta^{3} \\
    &+ 14670\eta^{2}h^{2} + 12513\eta^{2} + 306\eta h^{2} + 3919\eta \\
    &+ 324h^{2} + 468
    \end{aligned}
    \right)
}$

\item $R_{4,1} = \frac{
    12\left(
    \begin{aligned}
    &5125939200\eta^{18}h^{6} + 10281738240\eta^{18}h^{4} + 5558284800\eta^{18}h^{2} + 894240000\eta^{18} \\
    &+ 92266905600\eta^{17}h^{6} + 185699796480\eta^{17}h^{4} + 102349889280\eta^{17}h^{2} + 16820719200\eta^{17} \\
    &+ 753513062400\eta^{16}h^{6} + 1533153277440\eta^{16}h^{4} + 869728693248\eta^{16}h^{2} + 147351204000\eta^{16} \\
    &+ 3693894865920\eta^{15}h^{6} + 7664791000320\eta^{15}h^{4} + 4523894243712\eta^{15}h^{2} + 798136467840\eta^{15} \\
    &+ 12119107818240\eta^{14}h^{6} + 25910055066240\eta^{14}h^{4} + 16108509905424\eta^{14}h^{2} + 2992893093360\eta^{14} \\
    &+ 28085349335040\eta^{13}h^{6} + 62630232179328\eta^{13}h^{4} + 41600111677632\eta^{13}h^{2} + 8242720758846\eta^{13} \\
    &+ 47290649748480\eta^{12}h^{6} + 111650834845152\eta^{12}h^{4} + 80531038487400\eta^{12}h^{2} + 17260653893010\eta^{12} \\
    &+ 58622408640000\eta^{11}h^{6} + 149272985728512\eta^{11}h^{4} + 119117089590228\eta^{11}h^{2} + 28067027303367\eta^{11} \\
    &+ 53653634001792\eta^{10}h^{6} + 150846737463936\eta^{10}h^{4} + 136006697821755\eta^{10}h^{2} + 35880760063377\eta^{10} \\
    &+ 36114316285248\eta^{9}h^{6} + 115482068788680\eta^{9}h^{4} + 120339102072624\eta^{9}h^{2} + 36282684412872\eta^{9} \\
    &+ 17791941924000\eta^{8}h^{6} + 67014015316992\eta^{8}h^{4} + 82452239438400\eta^{8}h^{2} + 29049211011674\eta^{8} \\
    &+ 6480328315680\eta^{7}h^{6} + 29631448494612\eta^{7}h^{4} + 43593787823706\eta^{7}h^{2} + 18345964269200\eta^{7} \\
    &+ 1847292950772\eta^{6}h^{6} + 10176249637926\eta^{6}h^{4} + 17725325737089\eta^{6}h^{2} + 9060118915455\eta^{6} \\
    &+ 457903799568\eta^{5}h^{6} + 2812322208249\eta^{5}h^{4} + 5551779875742\eta^{5}h^{2} + 3447685454633\eta^{5} \\
    &+ 100833775344\eta^{4}h^{6} + 640613906838\eta^{4}h^{4} + 1355929558665\eta^{4}h^{2} + 987958000234\eta^{4} \\
    &+ 18043758450\eta^{3}h^{6} + 117735414327\eta^{3}h^{4} + 262613420166\eta^{3}h^{2} + 205632680413\eta^{3} \\
    &+ 2665312938\eta^{2}h^{6} + 17239361787\eta^{2}h^{4} + 39607338726\eta^{2}h^{2} + 29274555453\eta^{2} \\
    &+ 284960268\eta h^{6} + 2042821296\eta h^{4} + 4117843548\eta h^{2} + 2544603256\eta \\
    &+ 33743952h^{6} + 146223792h^{4} + 211212144h^{2} + 101694736
    \end{aligned}
    \right)
}{
    \left(
    \begin{aligned}
    &144\eta^{6}h^{2} + 36\eta^{6} + 864\eta^{5}h^{2} + 216\eta^{5} \\
    &+ 1872\eta^{4}h^{2} + 513\eta^{4} + 1728\eta^{3}h^{2} + 612\eta^{3} \\
    &+ 576\eta^{2}h^{2} + 381\eta^{2} + 114\eta + 9h^{2} + 13
    \end{aligned}
    \right)
    \left(
    \begin{aligned}
    &7119360\eta^{12}h^{4} + 5380992\eta^{12}h^{2} + 993600\eta^{12} \\
    &+ 85432320\eta^{11}h^{4} + 65832192\eta^{11}h^{2} + 12411936\eta^{11} \\
    &+ 441400320\eta^{10}h^{4} + 354042432\eta^{10}h^{2} + 69523992\eta^{10} \\
    &+ 1282395456\eta^{9}h^{4} + 1100354928\eta^{9}h^{2} + 230484552\eta^{9} \\
    &+ 2293837056\eta^{8}h^{4} + 2184628032\eta^{8}h^{2} + 502589250\eta^{8} \\
    &+ 2596123872\eta^{7}h^{4} + 2889389772\eta^{7}h^{2} + 757477230\eta^{7} \\
    &+ 1835752464\eta^{6}h^{4} + 2574302508\eta^{6}h^{2} + 806672538\eta^{6} \\
    &+ 772423128\eta^{5}h^{4} + 1527991470\eta^{5}h^{2} + 609566547\eta^{5} \\
    &+ 183679920\eta^{4}h^{4} + 587119923\eta^{4}h^{2} + 323289833\eta^{4} \\
    &+ 32446116\eta^{3}h^{4} + 142024332\eta^{3}h^{2} + 117035395\eta^{3} \\
    &+ 7344918\eta^{2}h^{4} + 22971285\eta^{2}h^{2} + 27415804\eta^{2} \\
    &+ 395766\eta h^{4} + 3155823\eta h^{2} + 3732677\eta \\
    &+ 107244h^{4} + 309816h^{2} + 223756
    \end{aligned}
    \right)}$

\end{itemize}
}

\clearpage
\bibliographystyle{plain}
\bibliography{reference}

\end{document}